\documentclass[11pt]{article}

\usepackage{amssymb}
\usepackage{amsmath}
\usepackage{graphicx}
\usepackage{color}
\usepackage{amsmath,amsthm,amsfonts,epsfig,setspace}
\numberwithin{equation}{section}
\usepackage{pgfplots}
\pgfplotsset{compat=newest}
\pgfplotsset{plot coordinates/math parser=false}
\newlength\figureheight
\newlength\figurewidth 
\newcommand{\ds}{\displaystyle}
\def\nm{\noalign{\medskip}}
\newtheorem{thm}{Theorem}[section]
\newtheorem{rmk}{Remark}[section]

\newtheorem{definition}{Definition} [section]
\newtheorem{lem}{Lemma}[section]

\makeatletter
\newtheorem*{rep@theorem}{\rep@title}
\newcommand{\newreptheorem}[2]{%
\newenvironment{rep#1}[1]{%
 \def\rep@title{#2 \ref{##1}}%
 \begin{rep@theorem}}%
 {\end{rep@theorem}}}
\makeatother

\newreptheorem{theorem}{Theorem}

 \def\p{\partial}
\def \Vh0{\stackrel{\circ}{V}_h} \def\to{\rightarrow}
   
  \def\om{\omega}
   
\def\l{\label}  \def\f{\frac} \def\df{\dfrac} 

   \def\eps{\varepsilon}

\def\l|{\left|}
\def\r|{\right|}

\newcommand{\R}{\mathbb{R}}

\newcommand{\lc}
{\mathrel{\raise2pt\hbox{${\mathop<\limits_{\raise1pt\hbox
{\mbox{$\sim$}}}}$}}}

\newcommand{\gc}
{\mathrel{\raise2pt\hbox{${\mathop>\limits_{\raise1pt\hbox{\mbox{$\sim$}}}}$}}}

\newcommand{\ec}
{\mathrel{\raise2pt\hbox{${\mathop=\limits_{\raise1pt\hbox{\mbox{$\sim$}}}}$}}}

\def\be{\begin{equation}} \def\ee{\end{equation}}

\def\bea{\begin{eqnarray}}  \def\eea{\end{eqnarray}}

\def\beas{\begin{eqnarray*}} \def\eeas{\end{eqnarray*}}

\def\bn{\begin{enumerate}} \def\en{\end{enumerate}}

\def\bd{\begin{description}} \def\ed{\end{description}}

\title{Heat generation with plasmonic nanoparticles}
\date{}

\author{
Habib Ammari\thanks{\footnotesize Department of Mathematics, 
ETH Z\"urich, 
R\"amistrasse 101, CH-8092 Z\"urich, Switzerland (habib.ammari@math.ethz.ch, francisco.romero@sam.math.ethz.ch). }
\and
Francisco Romero\footnotemark[2]
\and   Matias Ruiz\thanks{\footnotesize Department of Mathematics and Applications,
Ecole Normale Sup\'erieure, 45 Rue d'Ulm, 75005 Paris, France
(matias.ruiz@ens.fr).}
}

\begin{document}
\maketitle

\begin{abstract}
In this paper we use layer potentials and asymptotic analysis techniques to analyze the heat generation due to nanoparticles when illuminated at their plasmonic resonance. We consider arbitrary-shaped particles and both single and multiple particles. For close-to-touching nanoparticles, we show that the temperature field deviates significantly from the one generated by a single nanoparticle. The results of this paper open a door for solving the challenging problems of detecting plasmonic nanoparticles in biological media and monitoring temperature elevation in tissue generated by nanoparticle heating. 

\end{abstract}

\medskip

\bigskip

\noindent {\footnotesize Mathematics Subject Classification
(MSC2000): 35R30, 35C20.}

\noindent {\footnotesize Key words: plasmonic nanoparticle, plasmonic resonance, heat generation, Neumann-Poincar\'e operator.}


\section{Introduction} \label{sec-introduction heat}

Our aim in this paper is to provide a mathematical and numerical framework for analyzing photothermal effects using plasmonic nanoparticles.  A remarkable feature of plasmonic nanoparticles is that they exhibit quasi-static optical resonances, called plasmonic resonances. At or near these resonant frequencies, strong enhancement of scattering and absorption occurs \cite{matias,matias2, SC10}.   The plasmonic resonances are related to the spectra of the non-self adjoint Neumann-Poincar\'e type operators associated with the particle shapes \cite{matias,matias2,kang1,hyeonbae, Gri12, kang3}. Plasmonic nanoparticles efficiently generate heat in the presence of electromagnetic radiation.  Their biocompatibility  makes them  suitable for use  in nanotherapy \cite{baffou2010mapping}. 

Nanotherapy relies on a simple mechanism. First nanoparticles become attached to tumor cells using selective biomolecular linkers. Then heat generated by optically-simulated plasmonic nanoparticles destroys the tumor cells \cite{heat1}.  In this nanomedical application, the temperature increase is the most important parameter \cite{heat3,heat4}.  It depends on a highly nontrivial way on the shape, the number, and organization of the nanoparticles.  Moreover, it is challenging to measure it at the surface of the nanoparticles \cite{heat1}. 

In this paper, we derive an asymptotic formula for the temperature at the surface of plasmonic nanoparticles of arbitrary shape. Our formula holds for clusters  of simply connected nanoparticles.
It allows to estimate the collective response of plasmonic nanoparticles.

The paper is organized as follows. 
In section \ref{sec-setteing heat} we describe the mathematical setting for the physical phenomena we are modeling. To this end, we use the Helmholtz equation to model the propagation of light which we couple to the heat equation. Later on, we present our main results in this paper which consist on original asymptotic formulas for the inner field and the temperature on the boundaries of the nanoparticles. In section \ref{sec-heat generation heat} we prove Theorems \ref{thm12d heat} and \ref{thm-Heat small volume heat}. These results clarify the strong dependency of the heat generation on the geometry of the particles as it depends on the eigenvalues of the Neumann-Poincar\'e operator. In section \ref{sec-numeric heat} we present numerical examples of the temperature at the boundary of single and multiple particles. Appendix \ref{append1 heat} is devoted to the asymptotic analysis of layer potentials for the Helmholtz equation in dimension two. We also include an analysis for the invertibility of the single-layer potential for the Laplacian for the case of multiple particles.

\section{Setting of the problem and the main results} \label{sec-setteing heat}
In this paper, we use the Helmholtz equation for modeling the propagation of light. This can be thought of as a special case of Maxwell's equations, when the incident wave $u^i$ is a transverse electric or transverse magnetic (TE or TM) polarized wave. This approximation, also called paraxial approximation \cite{paraxial}, is a good model for a laser beam which are used, in particular, in full-field optical coherence tomography. We will therefore model the propagation of a laser beam in a host domain (tissue), hosting a nanoparticle.

Let the nanoparticle occupy a bounded domain $D\Subset\R^2$ of class $\mathcal{C}^{1,\alpha}$ for some $0<\alpha<1$. Furthermore, let $D = z + \delta B$, where $B$ is centered at the origin and $|B| = O(1)$.

We denote by $\varepsilon_c(x)$ and $\mu_c(x)$, $x\in D$, the electric permittivity and magnetic permeability of the particle, respectively, both of which may depend on the frequency $\om$ of the incident wave. Assume that $\varepsilon_c(x) = \varepsilon_0\varepsilon_c'$, $\mu_c(x) = \mu_0\mu_c'$ and that $\Re \eps_c' <0, \Im \eps_c' >0, \Re \mu_c' <0, \Im \mu_c' >0$. Here and throughout, $\varepsilon_0$ and $\mu_0$ are the permittivity and permeability of vacuum.

Similarly, we denote by  $\varepsilon_m(x) = \varepsilon_0\varepsilon_m'$ and $\mu_m(x) = \mu_0\mu_m'$, $x\in \R^2\backslash \overline{D}$ the permittivity and permeability of the host medium, both of which do not depend on the frequency $\om$ of the incident wave. Assume that $\eps_m$ and $\mu_m$ are real and strictly positive.

The index of refraction of the medium (with the nanoparticle) is given by
\beas
n(x) = \sqrt{\varepsilon_c'\mu_c'}\chi(D)(x) +  \sqrt{\varepsilon_m'\mu_m'}\chi(\R^2\backslash \overline{D})(x),
\eeas
where $\chi$ denotes the indicator function. 


The scattering problem for a TE incident wave $u^i$ is modeled as follows:
\be \label{eq-Helm heat}
 \left\{
\begin{array} {ll}
&\ds \nabla \cdot \f{c^2}{n^2} \nabla  u+ \omega^2 u  = 0 \quad \mbox{in } \R^2\backslash \p D, \\
\nm
& u_{+} -u_{-}  =0   \quad \mbox{on } \partial D, \\
\nm
&  \ds \f{1}{\eps_{m}} \f{\p u}{\p \nu} \bigg|_{+} - \f{1}{\eps_{c}} \f{\p u}{\p \nu} \bigg|_{-} =0 \quad \mbox{on } \partial D,\\
\nm
&  u^s:= u- u^{i}  \,\,\,  \mbox{satisfies the Sommerfeld radiation condition at infinity},
\end{array}
 \right.
\ee
where $\f{\p }{\p \nu}$ denotes the outward normal derivative and $c = \f{1}{\sqrt{\eps_0\mu_0}} $ is the speed of light in vacuum.  We use the notation $ \f{\p }{\p \nu} \Big|_{\pm}$ indicating
$$ \f{\p u}{\p \nu}\Big|_{\pm}(x)= \lim_{t \to 0^+} \nabla u(x\pm t\nu(x))\cdot \nu(x),$$ 
with $\nu$ being the outward unit normal vector   to $\p D$. 

The interaction of the electromagnetic waves with the medium produces a heat flow of energy which translates into a change of temperature governed by the heat equation \cite{heat2}
\be \label{eq-heat heat}
 \left\{
\begin{array} {ll}
&\ds \rho C \df{\partial \tau }{\partial t} - \nabla \cdot \gamma \nabla \tau = \f{\om}{2\pi}\Im (\eps)| u|^2 \quad \mbox{in } ( \R^2\backslash \p D )\times(0,T), \\
\nm
&\tau_{+} -\tau_{-}  = 0   \quad \mbox{on } \partial D, \\
\nm
&  \ds  \gamma_m \f{\p \tau}{\p \nu} \bigg|_{+} -  \gamma_c \f{\p \tau}{\p \nu} \bigg|_{-} =0 \quad \mbox{on } \partial D,\\
\nm
&\tau(x,0) = 0,
\end{array}
 \right.
\ee
where $\rho = \rho_c\chi(D) +   \rho_m\chi(\R^2\backslash \overline{D})$ is the mass density, $C = C_c\chi(D) +   C_m\chi(\R^2\backslash \overline{D})$ is the thermal capacity, $\gamma =  \gamma_c\chi(D) +   \gamma_m\chi(\R^2\backslash \overline{D})$ is the thermal conductivity, $T\in \R$ is the final time of measurements and $\eps = \varepsilon_c\chi(D) +  \varepsilon_m\chi(\R^2\backslash \overline{D})$.

We further assume that $\rho_c, \rho_m, C_c, C_m, \gamma_c, \gamma_m$ are real positive constants.

Note that $\Im(\eps) = 0$ in $\R^2\backslash  \overline{D}$ and so, outside $D$, the heat equation is homogeneous.

The coupling of equations \eqref{eq-Helm heat} and \eqref{eq-heat heat} describes the physics of our problem.

We remark that, in general, the index of refraction varies with temperature; hence, a solution to the above equations would imply a dependency on time for the electric field $u$, which contradicts the time-harmonic assumption leading to model \eqref{eq-Helm heat}. Nevertheless, the time-scale on the dynamics of the index of refraction is much larger than the time-scale on the dynamics of the interaction of the electromagnetic wave with the medium. Therefore, we will not integrate a time-varying component into the index of refraction. 

Let $G(\cdot,k)$ be the Green function for the Helmholtz operator $\Delta +   k^2$ satisfying the Sommerfeld radiation condition. In dimension two, $G$ is given by
$$
G(x, k)= -\dfrac{i}{4}H_0^{(1)}( k|x|),
$$
where $H_0^{(1)}$ is the Hankel function of first kind and order $0$. We denote $G(x,y, k):= G(x-y, k)$.

Define the following single-layer potential and Neumann-Poincar\'e integral operator
\[
\begin{array}{lr}
\ds \mathcal{S}_{D}^{ k} [\varphi](x)=  \int_{\p D} G(x, y,  k) \varphi(y) d\sigma(y), & \quad x \in \p D \mbox{ or }  x \in \R^2,
\end{array}
\]
and 
\[
\begin{array}{lr}
\ds
(\mathcal{K}_{D}^{ k})^* [\varphi](x)  = \int_{\p D } \f{\p G(x, y,  k)}{ \p \nu(x)} \varphi(y) d\sigma(y) ,  & \quad x \in \p D.
\end{array}
\]

Let $I$ denote the identity operator and let $\mathcal{S}_D$ and $\mathcal{K}_D^*$ respectively denote the single-layer potential and the Neumann-Poincar\'e operator associated to the Laplacian. Our main results in this paper are the following. 


\begin{thm} \label{thm12d heat}
For an incident wave $u^i \in \mathcal{C}^2(\R^2)$, the solution $u$ to \eqref{eq-Helm heat}, inside a plasmonic particle occupying a domain $D = z+ \delta B$,  has the following asymptotic expansion as $\delta \rightarrow 0$ in $L^2(D)$, 
$$
u = u^i(z) + \left(\delta(x-z) + \mathcal{S}_D \big(\lambda_{\eps}I - \mathcal{K}_D^*\big)^{-1} [\nu]\right)\cdot \nabla u^{i}(z) + O\left(\f{\delta^3 }{\textnormal{dist}(\lambda_{\eps},\sigma(\mathcal{K}^*_{D}))}\right),
$$
where $\nu$ is the outward normal to $D$, $\sigma(\mathcal{K}^*_{D})$ denotes the spectrum of $\mathcal{K}^*_{D}$ in $H^{-\f{1}{2}}(\p D)$ and  
\beas
\lambda_{\eps}: = \f{\eps_c+ \eps_m}{2(\eps_c-\eps_m)}.
\eeas
\end{thm}

\begin{thm} \label{thm-Heat small volume heat}
Let $u$ be the solution to \eqref{eq-Helm heat}. The solution $\tau$ to \eqref{eq-heat heat} on the boundary $\partial D$ of a plasmonic particle occupying the domain $D = z+ \delta B$ has the following asymptotic expansion as $\delta \rightarrow 0$, uniformly in $(x,t)\in \p D\times (0,T)$,
\beas
  \tau(x,t) = F_D(x,t,b_c) - \mathcal{V}_D^{b_c}(\lambda_{\gamma} I - \mathcal{K}_D^*)^{-1}[\df{\p F_D(\cdot,\cdot,b_c)}{\p \nu}](x,t) +  O\left(\f{\delta^4 \log\delta }{\textnormal{dist}(\lambda_{\eps},\sigma(\mathcal{K}^*_{D}))^2}\right),
\eeas
where $\nu$ is the outward normal to $D$ and  
\beas
\lambda_{ \gamma} &:=& \f{ \gamma_c+  \gamma_m}{2( \gamma_c- \gamma_m)},\\
b_c &:=& \f{ \rho_c C_c}{\gamma_c},\\
F_D(x,t, b_c) &:=& \f{\om}{2\pi \gamma_c} \Im (\eps_c)\int_0^t \int_{D} \f{e^{-\f{|x-y|^2}{4 b_c (t-t')}}}{4\pi b_c (t-t')}|u|^2(y) dy dt',\\
\mathcal{V}_D^{b_c}[f](x,t) &:=& \int_0^t \int_{\p D} \f{e^{-\f{|x-y|^2}{4 b_c (t-t')}}}{4\pi b_c (t-t')}f(y,t') dy dt'.
\eeas
\end{thm} 

\begin{rmk} We remark that Theorem \ref{thm12d heat} and Theorem \ref{thm-Heat small volume heat} are independent. A generalization of Theorem  \ref{thm-Heat small volume heat} to $\R^3$ is straightforward and the same type of small volume approximation can be found using the techniques presented in this paper. In fact, in $\R^3$, the operators involved in the first term of the temperature small volume expansion are
\beas
F_D(x,t, b_c) &:=& \f{\om}{2\pi \gamma_c} \Im (\eps_c)\int_0^t \int_{D} \f{e^{-\f{|x-y|^2}{4 b_c (t-t')}}}{\big(4\pi b_c (t-t')\big)^{\f{3}{2}}}|E|^2(y) dy dt',\\
\mathcal{V}_D^{b_c}[f](x,t) &:=& \int_0^t \int_{\p D} \f{e^{-\f{|x-y|^2}{4 b_c (t-t')}}}{\big(4\pi b_c (t-t')\big)^{\f{3}{2}}}f(y,t') dy dt'.
\eeas
Here $E$ is the vectorial electric field as a result of Maxwell equations. A small volume expansion for $E$ inside the nanoparticle for the plasmonic case can be found using the same techniques as those of \cite{matias2}.
\end{rmk}

Throughout this paper, we denote by $\mathcal{L}(E,F)$ the set of bounded linear applications from $E$ to $F$ and let $\mathcal{L}(E) := \mathcal{L}(E,E)$ and let  $H^s(\p D)$ to be the standard Sobolev space of order $s$ on $\p D$.

\section{Preliminaries}
\subsection{Layer potentials for the Helmholtz equation in two dimensions} \label{sec-Helm layer heat}

Let us recall some properties of the single-layer potential and the Neumann-Poincar\'e integral operator \cite{book3}:
\begin{enumerate}
\item[(i)] $\mathcal{S}_{D}^{ k}: H^{-\f{1}{2}}(\p D)\rightarrow  H^{\f{1}{2}}(\p D), H^{1}_{loc}(\R^2\backslash\p D)$ is bounded;
\item[(ii)] $(\Delta +  k^2)\mathcal{S}_{D}^{ k}[\varphi](x) = 0$ for $x \in \R^2\backslash \p D$, $\varphi \in  H^{-\f{1}{2}}(\p D)$;
\item[(iii)] $(\mathcal{K}_{D}^{ k})^* : H^{-\f{1}{2}}(\p D)\rightarrow  H^{-\f{1}{2}}(\p D)$ is compact;
\item[(iv)] $\mathcal{S}_{D}^{ k}[\varphi]$, $\varphi\in H^{-\f{1}{2}}(\p D)$, satisfies the Sommerfeld radiation condition at infinity;
\item[(v)] $\df{\p \mathcal{S}_{D}^{ k}[\varphi]}{\p \nu}\Big\vert_{\pm} = (\pm\f{1}{2}I +(\mathcal{K}_{D}^{ k})^*)[\varphi]$.
\end{enumerate}

We have that, for any $\psi, \phi \in  H^{-\f{1}{2}}(\p D)$,
\be \label{Helm-solution heat}
u := \left\{
\begin{array}{cc}
u^i + \mathcal{S}_{D}^{ k_m} [\psi], & \quad x \in \R^2 \backslash \overline{D},\\
\nm
\mathcal{S}_{D}^{ k_c} [\phi] ,  & \quad x \in {D},
\end{array}\right.
\ee
with $ k_m = \om\sqrt{\eps_m\mu_m}$ and  $ k_c = \om\sqrt{\eps_c\mu_c}$, satisfies $\nabla \cdot \f{c^2}{n^2} \nabla  u+ \omega^2 u  = 0 \mbox{ in } \R^2\backslash \p D$ and $u-u^i$ satisfies the Sommerfeld radiation condition.

To satisfy the boundary transmission conditions, $\psi, \phi \in  H^{-\f{1}{2}}(\p D)$ need to satisfy the following system of integral equations on $\partial D$
\be \label{Helm-syst heat}
\left\{
\begin{array}{rcl}
\mathcal{S}_D^{ k_m} [\psi] - \mathcal{S}_D^{ k_c} [\phi] &=& -u^{i},  \\
\nm
 \f{1}{\eps_m}\big(\f{1}{2}I + (\mathcal{K}_D^{ k_m})^*\big)[\psi] +  \f{1}{\eps_c} \big(\f{1}{2}I- (\mathcal{K}_D^{ k_c})^*\big)[\phi] &=& -\dfrac{1}{\eps_m}\df{\p u^{i}}{\p \nu}.
\end{array} \right.
\ee
The following result  shows the existence of such a representation \cite{BoundaryLayerHelmholtz}.
\begin{thm}
The operator
\beas
\mathcal{T}: \left(H^{-\f{1}{2}}(\p D)\right)^2 &\rightarrow& H^{\f{1}{2}}(\p D)\times H^{-\f{1}{2}}(\p D) \\
(\psi,\phi) &\mapsto& \left(\mathcal{S}_D^{ k_m} [\psi] - \mathcal{S}_D^{ k_c} [\phi],\; \f{1}{\eps_m}\big(\f{1}{2}I + (\mathcal{K}_D^{ k_m})^*\big)[\psi] +  \f{1}{\eps_c} \big(\f{1}{2}I - (\mathcal{K}_D^{ k_c})^*\big)[\phi]\right)
\eeas
is invertible.
\end{thm}

\section{Heat generation} \label{sec-heat generation heat}
In this section we consider the coupling of equations \eqref{eq-Helm heat} and \eqref{eq-heat heat}, that is, 
\bea \label{eq-helmholtz & heat heat}
 \left\{
\begin{array} {ll}
&\ds \nabla \cdot \f{c^2}{n^2} \nabla  u+ \omega^2 u  = 0 \quad \mbox{in } \R^2\backslash \p D, \\
\nm
& u_{+} -u_{-} = 0   \quad \mbox{on } \partial D, \\
\nm
&  \ds \f{1}{\eps_{m}} \f{\p u}{\p \nu} \bigg|_{+} - \f{1}{\eps_{c}} \f{\p u}{\p \nu} \bigg|_{-} =0 \quad \mbox{on } \partial D,\\
\nm
&  u^s:= u- u^{i}  \,\,\,  \mbox{satisfies the Sommerfeld radiation condition at infinity},\\
\nm
&\ds \f{\rho_c C_c}{\gamma_c}\df{\partial \tau }{\partial t} - \Delta \tau = \f{\om}{2\pi \gamma_c}\Im (\eps_c)| u|^2  \quad \mbox{in }  D \times(0,T),\\
\nm
&\ds \f{\rho_m C_m}{\gamma_m} \df{\partial \tau }{\partial t} - \Delta \tau = 0 \quad \mbox{in } ( \R^2\backslash \overline{D} )\times(0,T),\\
\nm
&\tau_{+} -\tau_{-}  = 0   \quad \mbox{on } \partial D, \\
\nm
&  \ds  \gamma_m \f{\p \tau}{\p \nu} \bigg|_{+} -  \gamma_c \f{\p \tau}{\p \nu} \bigg|_{-} =0 \quad \mbox{on } \partial D,\\
\nm
&\tau(x,0) = 0.
\end{array}
 \right.
\eea


Under the assumption that the index of refraction $n$ does not depend on the temperature, we can solve equation \eqref{eq-Helm heat} separately from equation \eqref{eq-heat heat}. 

Our goal is to establish  a small volume expansion for the resulting temperature at the surface of the nanoparticule as a function of time. To do so, we first need to compute the electric field inside the nanoparticule as a result of a plasmonic resonance. We make use of layer potentials for the Helmholtz equation, described in subsection \ref{sec-Helm layer heat}.

\subsection{Small volume expansion of the inner field} \label{sec-field expansions heat}
We proceed in this section to prove Theorem \ref{thm12d heat}.

\subsubsection{Rescaling}
Since we are working with nanoparticles, we want to rescale equation \eqref{Helm-syst heat} to study the solution for a small volume approximation by using representation \eqref{Helm-solution heat}.

Recall that $D = z + \delta B$. For any $x\in \p D$, $\widetilde{x} := \f{x-z}{\delta}\in \p B$ and for each function $f$ defined on $\p D$, we introduce a corresponding function defined on $\p B$ as follows
\be \label{defeta heat} \begin{array}{rcl}
\eta(f)(\widetilde{x}) = f(z + \delta \widetilde{x}).
\end{array}
\ee

It follows that
\be \begin{array}{rcl} \label{eq-small volume single layer heat}
\mathcal{S}_D^{ k} [\varphi](x) &=& \delta\mathcal{S}_B^{\delta k} [\eta(\varphi)](\widetilde{x}),\\
\nm
(\mathcal{K}_D^{ k})^*[\varphi](x) &=& (\mathcal{K}_B^{\delta k})^*[\eta(\varphi)](\widetilde{x}),
\end{array}
\ee
so system \eqref{Helm-syst heat} becomes
\be \label{Helm-syst re-scaled heat}
\left\{
\begin{array}{rcl}
\mathcal{S}_B^{\delta k_m} [\eta(\psi)] - \mathcal{S}_B^{\delta k_c} [\eta(\phi)] &=& -\df{\eta(u^{i})}{\delta},  \\
\nm
 \f{1}{\eps_m}\big(\f{1}{2}I + (\mathcal{K}_B^{\delta k_m})^*\big)[\eta(\psi)] +  \f{1}{\eps_c} \big(\f{1}{2}I - (\mathcal{K}_B^{\delta k_c})^*\big)[\eta(\phi)] &=& -\dfrac{1}{\eps_m}\eta(\df{\p u^{i}}{\p \nu}).
\end{array} \right.
\ee

Note that the system is defined on $\p B$.

 
For $\delta$ small enough $\mathcal{S}_B^{\delta k_m}$ is invertible (see Appendix \ref{append1 heat}). Therefore, 
\beas
\eta(\psi) = (\mathcal{S}_B^{\delta k_m})^{-1}\mathcal{S}_B^{\delta k_c}[\eta(\phi)] -  (\mathcal{S}_B^{\delta k_m})^{-1}[\df{\eta(u^{i})}{\delta}].
\eeas
Hence, we have the following equation for $\eta(\phi)$:
\beas
\mathcal{A}^I_{B}(\delta)[\eta(\phi)] = f^I,
\eeas
where
\be
\begin{array}{rcl} \label{eq-A f^I heat}
 \mathcal{A}^I_B(\delta) &=& \f{1}{\eps_m}\big(\f{1}{2}I + (\mathcal{K}_B^{\delta k_m})^*\big)(\mathcal{S}_B^{\delta k_m})^{-1}\mathcal{S}_B^{\delta k_c} +  \f{1}{\eps_c} \big(\f{1}{2}I - (\mathcal{K}_B^{\delta k_c})^*\big) , \\
 \nm
 f^I &=& -\dfrac{1}{\eps_m}\eta(\df{\p u^{i}}{\p \nu})+\f{1}{\eps_m}\big(\f{1}{2}I + (\mathcal{K}_B^{\delta k_m})^*\big)(\mathcal{S}_B^{\delta k_m})^{-1}[\df{\eta(u^{i})}{\delta}].
 \end{array}
\ee

\subsubsection{Proof of Theorem \ref{thm12d heat}}
To express the solution to \eqref{eq-Helm heat} in $D$, asymptotically on the size of the nanoparticle $\delta$, we make use of the representation \eqref{Helm-solution heat}. We derive an asymptotic expansion for $\eta(\phi)$ on $\delta$ to later compute $\delta\mathcal{S}_B^{\delta k_c}[\eta(\phi)]$ and scale back to $D$. We divide the proof into three steps.  

\medskip
\textbf{Step 1.} We first compute an asymptotic for $\mathcal{A}^I_B(\delta)$ and $f^I$.
\medskip

Let $\mathcal{H}^*(\p B)$ be defined by (\ref{addeq5}) with $D$ replaced by $B$. In $\mathcal{L}(\mathcal{H}^*(\p B))$, we have the following asymptotic expansion as $\delta \rightarrow 0$ (see Appendix \ref{append1 heat})
\beas
(\mathcal{S}_B^{\delta k_m})^{-1} \mathcal{S}_B^{\delta k_c} &=& \mathcal{P}_{\mathcal{H}^*_0} + \mathcal{U}_{\delta k_m}(\widetilde{\mathcal{S}}_{B} + \Upsilon_{\delta k_c}) +  O(\delta^2\log \delta), \\
\f{1}{2}I  \pm (\mathcal{K}_B^{\delta k})^* &=& \big( \f{1}{2}I \pm \mathcal{K}_B^* \big) + O(\delta^2\log \delta).
\eeas

Let $\varphi_0$ be an eigenfunction of $\mathcal{K}_B^*$  associated to the eigenvalue $1/2$ (see Appendix \ref{append1 heat}) and let $\mathcal{U}_{\delta k_m}$ be defined by (\ref{addeq1}) with $k$ replaced with $\delta k_m$.  Then it follows that
\beas
\big(\f{1}{2}I + \mathcal{K}_B^*\big)\mathcal{U}_{\delta k_m} = \mathcal{U}_{\delta k_m}.
\eeas
Therefore, in $\mathcal{L}(\mathcal{H}^*(\p B))$, 
\beas
\mathcal{A}^I_{B}(\delta) = \left(\big( \f{1}{2\eps_m} +  \f{1}{2\eps_c}\big)I + \big( \f{1}{\eps_m} -  \f{1}{\eps_c} \big) \mathcal{K}_B^*\right)\mathcal{P}_{\mathcal{H}^*_0} + \f{1}{\eps_m}\mathcal{U}_{\delta k_m}(\widetilde{\mathcal{S}}_{B} + \Upsilon_{\delta k_c}) + O(\delta^2 \log \delta),
\eeas
and from the definition of $\mathcal{U}_{\delta k_m}$ we get
\be \label{eq-A asymptotic heat}
\mathcal{A}^I_{B}(\delta) = \left(\big( \f{1}{2\eps_m} +  \f{1}{2\eps_c}\big)I + \big( \f{1}{\eps_m} -  \f{1}{\eps_c} \big) \mathcal{K}_B^*\right)\mathcal{P}_{\mathcal{H}^*_0} + \f{1}{\eps_m}\f{\mathcal{S}_B[\varphi_0] + \tau_{\delta k_c}}{\mathcal{S}_B[\varphi_0] + \tau_{\delta k_m}}(\cdot,\varphi_0)_{\mathcal{H}^*}\varphi_0 + O(\delta^2 \log \delta).
\ee

In the same manner, in the space $\mathcal{H}^*(\partial B)$,
\beas
 f^I = \dfrac{1}{\eps_m}\left(-\eta(\df{\p u^{i}}{\p \nu})+\big(\f{1}{2}I + \mathcal{K}_B^*\big)\mathcal{P}_{\mathcal{H}^*_0}\widetilde{\mathcal{S}}_{B}^{-1}[\df{\eta(u^{i})}{\delta}] + \mathcal{U}_{\delta k_m}[\df{\eta(u^{i})}{\delta}] + O(\delta^2\log \delta)\right).
\eeas
We can further develop $f^I$. Indeed, for every $\tilde{x} \in \p B$, a Taylor expansion yields
\beas
\eta(\df{\p u^{i}}{\p \nu})(\tilde{x}) &=& \nu(\tilde{x}) \cdot \nabla u^{i}(\delta\tilde{x}+z) \;=\;  \nu(\tilde{x}) \cdot \nabla u^{i}(z) + O(\delta),\\
\f{\eta(u^{i})}{\delta}(\tilde{x}) &=& \f{u^{i}(\delta\tilde{x}+z)}{\delta} \;=\; \f{u^{i}(z)}{\delta} +  \tilde{x}\cdot\nabla u^{i}(z) + O(\delta).
\eeas
The regularity of $u^i$ ensures that the previous formulas hold in $\mathcal{H}^*(\p B)$.

The fact that $\tilde{x}\cdot\nabla u^i(z)$ is harmonic in $B$ and Lemma \ref{lem-d/dn(u) 1/2-K^*u heat} imply that $$-\nu\cdot\nabla u^{i}(z) = (\dfrac{1}{2}I - \mathcal{K}_{B}^*)\mathcal{P}_{\mathcal{H}^*_0} \widetilde{\mathcal{S}}_{B}^{-1}[\tilde{x}\cdot\nabla u^{i}(z)]$$ in $\mathcal{H}^*(\p B)$. 

Thus, in $\mathcal{H}^*(\p B)$,
\beas
f^I = \dfrac{1}{\eps_m}\left(\mathcal{P}_{\mathcal{H}^*_0}\widetilde{\mathcal{S}}_{B}^{-1}[\tilde{x}\cdot\nabla u^{i}(z)] + \mathcal{U}_{\delta k_m}[\f{u^i(z)}{\delta}+\tilde{x}\nabla u^i(z)] + O(\delta)\right).
\eeas
From the definition of $\mathcal{U}_{\delta k_m}$ we get
\be \label{eq-f^I asymptotic heat}
f^I = \dfrac{1}{\eps_m}\left(\mathcal{P}_{\mathcal{H}^*_0}\widetilde{\mathcal{S}}_{B}^{-1}[\tilde{x}\cdot\nabla u^{i}(z)] + \f{u^i(z)\varphi_0}{\delta(\mathcal{S}_B[\varphi_0] + \tau_{\delta  k_m})} - \f{(\widetilde{\mathcal{S}}_{B}^{-1}[\tilde{x}\cdot\nabla u^{i}(z)],\varphi_0)_{\mathcal{H}^*}\varphi_0}{\mathcal{S}_B[\varphi_0] + \tau_{\delta  k_m}}  + O(\delta)\right).
\ee

\medskip
\textbf{Step 2.} We compute $(\mathcal{A}^I_B(\delta))^{-1}f^I$. 
\medskip

We begin by computing an asymptotic expansion of $(\mathcal{A}^I_B(\delta))^{-1}$.

The operator $\mathcal{A}^I_0:=\left(\big( \f{1}{2\eps_m} +  \f{1}{2\eps_c}\big)I + \big( \f{1}{\eps_m} -  \f{1}{\eps_c} \big) \mathcal{K}_B^*\right)$ maps $\mathcal{H}^*_0$ into $\mathcal{H}^*_0$. Hence, the operator defined by (which appears in the expansion of $\mathcal{A}^I_B(\delta)$) 
\beas
\mathcal{A}^I_{B,0}:=\mathcal{A}^I_0\mathcal{P}_{\mathcal{H}^*_0} + \f{1}{\eps_m}\f{\mathcal{S}_B[\varphi_0] + \tau_{\delta k_c}}{\mathcal{S}_B[\varphi_0] + \tau_{\delta k_m}}(\cdot,\varphi_0)_{\mathcal{H}^*}\varphi_0,
\eeas
is invertible of inverse
\beas
(\mathcal{A}^I_{B,0})^{-1}=(\mathcal{A}^I_0)^{-1}\mathcal{P}_{\mathcal{H}^*_0} + \eps_m\f{\mathcal{S}_B[\varphi_0] + \tau_{\delta k_m}}{\mathcal{S}_B[\varphi_0] + \tau_{\delta k_c}}(\cdot,\varphi_0)_{\mathcal{H}^*}\varphi_0.
\eeas
Therefore, we can write
\beas
(\mathcal{A}^I_{B})^{-1}(\delta) = \big(I + (\mathcal{A}^I_{B, 0})^{-1}O(\delta^2 \log \delta)\big)^{-1}(\mathcal{A}^I_{B, 0})^{-1}.
\eeas
Since $\mathcal{K}_B^*$ is a compact self-adjoint operator in $\mathcal{H}^*(\p B)$ it follows that \cite{pierre,matias}
\be \label{eq-operator estimator heat}
\|(\mathcal{A}^I_{0})^{-1}\|_{\mathcal{L}(\mathcal{H}^*(\p B))} \leq \f{c}{\mathrm{dist}(0,\sigma(\mathcal{A}^I_{0}))} ,
\ee
for a constant $c$. Therefore, for $\delta$ small enough, we obtain
\beas
(\mathcal{A}^I_B(\delta))^{-1}f^I &=& \big(I + (\mathcal{A}^I_{B, 0})^{-1}O(\delta^2 \log \delta)\big)^{-1}(\mathcal{A}^I_{B, 0})^{-1}f^I\\
&=& \big(I + (\mathcal{A}^I_{B, 0})^{-1}O(\delta^2 \log \delta)\big)^{-1}\left(\f{u^i(z)\varphi_0}{\delta(\mathcal{S}_B[\varphi_0] + \tau_{\delta  k_c})} - \f{(\widetilde{\mathcal{S}}_{B}^{-1}[\tilde{x}\cdot\nabla u^{i}(z)],\varphi_0)_{\mathcal{H}^*}\varphi_0}{\mathcal{S}_B[\varphi_0] + \tau_{\delta  k_c}} + \right. \\
&& \quad \; (\mathcal{A}^I_0)^{-1}\f{1}{\eps_m}\mathcal{P}_{\mathcal{H}^*_0}\widetilde{\mathcal{S}}_{B}^{-1}[\tilde{x}\cdot\nabla u^{i}(z)]+ O\left(\f{\delta}{	\textnormal{dist}(0,\sigma({\mathcal{A}^I}_{0}))}\right)\Bigg)\\
&=& \f{u^i(z)\varphi_0}{\delta(\mathcal{S}_B[\varphi_0] + \tau_{\delta  k_c})} - \f{(\widetilde{\mathcal{S}}_{B}^{-1}[\tilde{x}\cdot\nabla u^{i}(z)],\varphi_0)_{\mathcal{H}^*}\varphi_0}{\mathcal{S}_B[\varphi_0] + \tau_{\delta  k_c}} + (\mathcal{A}^I_0)^{-1}\f{1}{\eps_m}\mathcal{P}_{\mathcal{H}^*_0}\widetilde{\mathcal{S}}_{B}^{-1}[\tilde{x}\cdot\nabla u^{i}(z)]+\\ &&\quad \; O\left(\f{\delta}{	\textnormal{dist}(0,\sigma({\mathcal{A}^I}_{0}))}\right).
\eeas
Using the representation formula of $\mathcal{K}^*_B$ described in Lemma \ref{lem-K_star_properties2d heat} we can further develop the third term in the above expression to obtain
\beas
(\mathcal{A}^I_0)^{-1}\mathcal{P}_{\mathcal{H}^*_0}\widetilde{\mathcal{S}}_{B}^{-1}[\tilde{x}\cdot\nabla u^{i}(z)] &=& \sum_{j=1}^{\infty}\f{(\widetilde{\mathcal{S}}_{B}^{-1}[\tilde{x}\cdot\nabla u^{i}(z)],\varphi_j)_{\mathcal{H}^*}\varphi_j}{\big( \f{1}{2} +  \f{\eps_m}{2\eps_c}\big) - \big( \f{\eps_m}{\eps_c} -  1 \big)\lambda_j}\\
&=&
 \sum_{j=1}^{\infty}\left(\f{(\widetilde{\mathcal{S}}_{B}^{-1}[\tilde{x}\cdot\nabla u^{i}(z)],\varphi_j)_{\mathcal{H}^*}\varphi_j}{\big( \f{1}{2} +  \f{\eps_m}{2\eps_c}\big) - \big( \f{\eps_m}{\eps_c} -  1 \big)\lambda_j} - (\widetilde{\mathcal{S}}_{B}^{-1}[\tilde{x}\cdot\nabla u^{i}(z)],\varphi_j)_{\mathcal{H}^*}\varphi_j \right) \\
 && \quad \; + \mathcal{P}_{\mathcal{H}^*_0}\widetilde{\mathcal{S}}_{B}^{-1}[\tilde{x}\cdot\nabla u^{i}(z)] \\
&=& \mathcal{P}_{\mathcal{H}^*_0}\widetilde{\mathcal{S}}_{B}^{-1}[\tilde{x}\cdot\nabla u^{i}(z)] + \sum_{j=1}^{\infty}(\lambda_j-\f{1}{2})\f{(\widetilde{\mathcal{S}}_{B}^{-1}[\tilde{x}\cdot\nabla u^{i}(z)],\varphi_j)_{\mathcal{H}^*}\varphi_j}{\lambda - \lambda_j}.
\eeas
Using the same arguments as those  in the proof of Lemma \ref{lem-d/dn(u) 1/2-K^*u heat}, we have
\beas
(\lambda_j-\f{1}{2})(\widetilde{\mathcal{S}}_{B}^{-1}[\tilde{x}\cdot\nabla u^{i}(z)],\varphi_j)_{\mathcal{H}^*} =  \f{(\nu\cdot\nabla u^{i}(z),\varphi_j)_{\mathcal{H}^*}}{\lambda_j-\f{1}{2}},
\eeas
and consequently, 
\beas
(\mathcal{A}^I_0)^{-1}\f{1}{\eps_m}\mathcal{P}_{\mathcal{H}^*_0}\widetilde{\mathcal{S}}_{B}^{-1}[\tilde{x}\cdot\nabla u^{i}(z)] = \mathcal{P}_{\mathcal{H}^*_0}\widetilde{\mathcal{S}}_{B}^{-1}[\tilde{x}\cdot\nabla u^{i}(z)] + (\lambda_{\eps} I - \mathcal{K}_B^*)^{-1}[\nu]\cdot \nabla u^{i}(z).
\eeas
Therefore, 
\beas
(\mathcal{A}^I_B(\delta))^{-1}f^I = &\df{u^i(z)\varphi_0}{\delta(\mathcal{S}_B[\varphi_0] + \tau_{\delta  k_c})} - \df{(\widetilde{\mathcal{S}}_{B}^{-1}[\tilde{x}\cdot\nabla u^{i}(z)],\varphi_0)_{\mathcal{H}^*}\varphi_0}{\mathcal{S}_B[\varphi_0] + \tau_{\delta  k_c}} + \mathcal{P}_{\mathcal{H}^*_0}\widetilde{\mathcal{S}}_{B}^{-1}[\tilde{x}\cdot\nabla u^{i}(z)] + \\ &\quad \; (\lambda_{\eps} I - \mathcal{K}_B^*)^{-1}[\nu]\cdot \nabla u^{i}(z) + O\left(\df{\delta}{	\textnormal{dist}(0,\sigma({\mathcal{A}^I}_{0}))}\right).
\eeas

\medskip
\textbf{Step 3.} Finally, we compute $\eta(u) = \delta \mathcal{S}_B^{\delta  k_c}(\mathcal{A}^I_B(\delta))^{-1}f^I$.
\medskip

From Appendix \ref{append1 heat}, the following holds when $\mathcal{S}_B^{\delta  k_c}$ is viewed as an operator from the space $\mathcal{H}^*(\p B)$ to $\mathcal{H}(\p B)$:
\beas
\mathcal{S}_B^{\delta  k_c} = \widetilde{\mathcal{S}}_{B} + \Upsilon_{\delta k_c} + O(\delta^2 \log \delta).
\eeas
In particular, we have
\beas
\mathcal{S}_B^{\delta  k_c}[\varphi_0] = \mathcal{S}_B[\varphi_0] + \tau_{\delta k_c} + O(\delta^2 \log \delta).
\eeas
It can be verified that the same expansion holds when viewed as an operator from $\mathcal{H}^*(\p B)$ into $L^2(B)$.

Note that the following identity holds
\beas
- \df{(\widetilde{\mathcal{S}}_{B}^{-1}[\tilde{x}\cdot\nabla u^{i}(z)],\varphi_0)_{\mathcal{H}^*}\varphi_0}{\mathcal{S}_B[\varphi_0] + \tau_{\delta  k_c}} + \mathcal{P}_{\mathcal{H}^*_0}\widetilde{\mathcal{S}}_{B}^{-1}[\tilde{x}\cdot\nabla u^{i}(z)] =  - 
\f{\Upsilon_{\delta k_c}\big[\widetilde{\mathcal{S}}_{B}^{-1}[\tilde{x}\cdot\nabla u^{i}(z)]\big]\varphi_0}{\mathcal{S}_B[\varphi_0] + \tau_{\delta  k_c}} + \widetilde{\mathcal{S}}_{B}^{-1}[\tilde{x}\cdot\nabla u^{i}(z)].
\eeas
Straightforward calculations and the fact that $\mathcal{S}_B$ is harmonic in $B$ yields
\beas
\delta \mathcal{S}_B^{\delta  k_c}(\mathcal{A}^I_B(\delta))^{-1}f^I = u^i(z) + \delta\big(\tilde{x} + \mathcal{S}_B \big(\lambda_{\eps}I - \mathcal{K}_B^*\big)^{-1} [\nu]\big)\cdot \nabla u^{i}(z) + O\left(\f{\delta^2 }{\textnormal{dist}(\lambda_{\eps},\sigma(\mathcal{K}^*_{B}))}\right)
\eeas
in $L^2(B)$.
Using Lemma \ref{lem-change norm scaling heat} to scale back the estimate to $D$ leads to the desired result.

\subsection{Small volume expansion of the temperature}
We proceed in this section to prove Theorem \ref{thm-Heat small volume heat}. To do so, we make use of the Laplace transform method \cite{CostabelHeat,HohageSayas,LubichSchneider}.

Consider equation \eqref{eq-helmholtz & heat heat} and define the Laplace transform of a function $g(t)$ by
\beas
L(g)(s) = \int_0^{\infty}e^{-st}g(t)dt.
\eeas
 Taking the Laplace transform of the equations on $\tau$ in (\ref{eq-helmholtz & heat heat}) we formally obtain the following system:
\be \label{eq-heat Laplace transformed}
 \left\{
\begin{array} {ll}
&\ds s \f{\rho_c C_c}{\gamma_c} \hat{\tau}(\cdot,s) - \Delta \hat{\tau}(\cdot,s) = L(g_u)(\cdot,s) \quad \mbox{in }  D , \\
\nm
&\ds s \f{\rho_m C_m}{\gamma_m}\hat{\tau}(\cdot,s) - \Delta \hat{\tau}(\cdot,s) = 0 \quad \mbox{in }  \R^2\backslash \overline{D},  \\
\nm
&\hat{\tau}_{+}(\cdot,s) - \hat{\tau}_{-}(\cdot,s)  = 0   \quad \mbox{on } \partial D, \\
\nm
&  \ds  \gamma_m \f{\p \hat{\tau}}{\p \nu} \bigg|_{+} -  \gamma_c \f{\p \hat{\tau}}{\p \nu} \bigg|_{-} =0 \quad \mbox{on } \partial D,\\
\nm
&\hat{\tau}(\cdot,s) \mbox{ satisfies the Sommerfeld radiation condition at infinity},
\end{array}
 \right.
\ee
where $\hat{\tau}(\cdot,s)$ and $L(g_u)(\cdot,s)$ are the Laplace transforms of $\tau$ and $g_u := \f{\om}{2\pi \gamma_c}\Im (\eps_c)| u|^2$, respectively, and $s\in \mathbb{C} \backslash (-\infty, 0]$.

A rigorous justification for the derivation of system \eqref{eq-heat Laplace transformed} and the validity of the inverse transform of the solution can be found in \cite{HohageSayas}.

Using layer potential techniques we have that, for any $\hat{p}, \hat{q} \in  H^{-\f{1}{2}}(\p D)$,
$\hat{\tau}$ defined by 
\be \label{Heat-solution heat}
\hat{\tau} := \left\{
\begin{array}{cc}
-\mathcal{S}_D^{\beta_{\gamma_m}} [\hat{p}], & \quad x \in \R^2 \backslash \overline{D},\\
-\hat{F}_D(\cdot,y,\beta_{\gamma_c}) - \mathcal{S}_D^{\beta_{\gamma_c}} [\hat{q}] ,  & \quad x \in {D},
\end{array}\right.
\ee
satisfies the differential equations in \eqref{eq-heat Laplace transformed} together with the Sommerfeld radiation condition. Here $ \beta_{\gamma_m} := i\sqrt{s \f{ \rho_m C_m}{\gamma_m}}$, $ \beta_{\gamma_c} := i\sqrt{s \f{ \rho_c C_c}{\gamma_c}}$ and 
\beas
\hat{F}_D(\cdot,\beta_{\gamma_c}) := \int_{D}G(\cdot,y,\beta_{\gamma_c})L(g_u)(y)dy.
\eeas

To satisfy the boundary transmission conditions, $\hat{p}$ and $\hat{q} \in  H^{-\f{1}{2}}(\p D)$ should satisfy the following system of integral equations on $\partial D$:
\be \label{eq-heat layer system heat}
\left\{
\begin{array}{rcl}
 -\mathcal{S}_D^{\beta_{\gamma_m}} [\hat{p}] + \mathcal{S}_D^{\beta_{\gamma_c}} [\hat{q}] &=& -\hat{F}_D(\cdot,\beta_{\gamma_c}),  \\
\nm
  - \gamma_m \big(\f{1}{2}I + (\mathcal{K}_D^{\beta_{\gamma_m}})^*\big)[\hat{p}] +  \gamma_c\big(-\f{1}{2}I + (\mathcal{K}_D^{\beta_{\gamma_c}})^*\big)[\hat{q}] &=&  -\gamma_c \df{\p \hat{F}_D(\cdot,\beta_{\gamma_c})}{\p \nu}.
\end{array} \right.
\ee

\subsubsection{Rescaling of the equations}

Recall that $D = z + \delta B$, for any $x\in \p D$, $\widetilde{x} := \f{x-z}{\delta}\in \p B$, for each function $f$ defined on $\p D$, $\eta$ is such that $\eta(f)(\widetilde{x}) = f(z + \delta \widetilde{x})$ and
\beas \begin{array}{rcl} 
\mathcal{S}_D^{ k} [\varphi](x) &=& \delta\mathcal{S}_B^{\delta k} [\eta(\varphi)](\widetilde{x}),\\
(\mathcal{K}_D^{ k})^*[\varphi](x) &=& (\mathcal{K}_B^{\delta k})^*[\eta(\varphi)](\widetilde{x}).
\end{array}
\eeas
We can also verify that  
\beas
\hat{F}_D(x,\beta_{\gamma_c}) &=& \delta^2 \hat{F}_B(\hat{x},\delta\beta_{\gamma_c}),\\
\df{\p \hat{F}_D}{\p \nu}(x,\beta_{\gamma_c}) &=& \delta \df{\p \hat{F}_B}{\p \nu}(\hat{x}.\delta\beta_{\gamma_c}).
\eeas
Note that in the above identity, in the left-hand side we differentiate with respect to $x$ while in the right-hand side we differentiate with respect to $\tilde{x}$. To simplify the notation, we will use $\hat{F}_B$ to refer to $\hat{F}_B(\cdot,\delta\beta_{\gamma_c})$.

We rescale system \eqref{eq-heat layer system heat} to arrive at
\beas
\left\{
\begin{array}{rcl}
 -\mathcal{S}_B^{\delta \beta_{\gamma_m}} [\eta(\hat{p})] + \mathcal{S}_B^{\delta \beta_{\gamma_c}} [\eta(\hat{q})] &=& -\delta\hat{F}_B,  \\
\nm
   -\gamma_m \big(\f{1}{2}I + (\mathcal{K}_B^{\delta \beta_{\gamma_m}})^*\big)[\eta(\hat{p})] +  \gamma_c\big(-\f{1}{2}I + (\mathcal{K}_B^{\delta \beta_{\gamma_c}})^*\big)[\eta(\hat{q})] &=&  -\gamma_c \delta \df{\p \hat{F}_B}{\p \nu}.
\end{array} \right.
\eeas

For $\delta$ small enough, $\mathcal{S}_B^{\delta  \beta_{\gamma_c}}$ is invertible (see Appendix \ref{append1 heat}). Therefore, it follows that
\beas
\eta(\hat{p}) = (\mathcal{S}_B^{\delta\beta_{\gamma_m}})^{-1}\mathcal{S}_B^{\delta\beta_{\gamma_c}}[\eta(\hat{q})] +  (\mathcal{S}_B^{\delta\beta_{\gamma_m}})^{-1}\left[\delta\hat{F}_B\right].
\eeas
Hence, we have the following equation for $\eta(\hat{q})$:
\beas
\mathcal{A}^h_{B}(\delta)[\eta(\hat{q})] = f^h,
\eeas
where
\be
\begin{array}{rcl} \label{eq-A^h f^h heat}
 \mathcal{A}^h_B(\delta) &=&  -\gamma_m\big(\f{1}{2}I + (\mathcal{K}_B^{\delta \beta_{\gamma_m}})^*\big)(\mathcal{S}_B^{\delta\beta_{\gamma_m}})^{-1}\mathcal{S}_B^{\delta\beta_{\gamma_c}} +  \gamma_c \big(-\f{1}{2}I + (\mathcal{K}_B^{\delta \beta_{\gamma_c}})^*\big),\\
\nm
 f^h &=& - \gamma_c \delta \df{\p \hat{F}_B}{\p \nu} + \gamma_m\big(\f{1}{2}I + (\mathcal{K}_B^{\delta \beta_{\gamma_m}})^*\big)(\mathcal{S}_B^{\delta\beta_{\gamma_m}})^{-1}\left[\delta\hat{F}_B\right].
 \end{array}
\ee


\subsubsection{Proof of Theorem \ref{thm-Heat small volume heat}}
To express the solution of \eqref{eq-heat heat} on $\p D\times(0,T)$, asymptotically on the size of the nanoparticle $\delta$, we make use of the representation \eqref{Heat-solution heat}. We will compute an asymptotic expansion for $\eta(\hat{q})$ on $\delta$ to later compute $\delta\mathcal{S}_B^{\delta\beta_{\gamma_c}}[\eta(\hat{q})]$ on $\p B$, scale back to $D$ and take Laplace inverse. 

Using the asymptotic expansions of Appendix \ref{append1 heat} the following asymptotic for $\mathcal{A}^h_{B}(\delta)$ holds in $\mathcal{L}(\mathcal{H}^*(\p B))$ 
\beas
\mathcal{A}^h_{B}(\delta) = \mathcal{A}^h_0\mathcal + O(\delta^2 \log \delta),
\eeas
where
\beas
\mathcal{A}^h_0 = -\left(\f{1}{2} \big( \gamma_c + \gamma_m  \big)I - \big( \gamma_c - \gamma_m \big) \mathcal{K}_B^*\right).
\eeas
In the same manner, in $\mathcal{H}^*(\p B)$, 
\beas
f^h &=& -\gamma_c \delta \df{\p \hat{F}_B}{\p \nu} + \gamma_m\big(\f{1}{2}I + \mathcal{K}_B^*\big)\widetilde{\mathcal{S}}_B^{-1}[\delta\hat{F}_B] + O\left(\f{\delta^5 \log\delta }{\textnormal{dist}(\lambda_{\eps},\sigma(\mathcal{K}^*_{D}))^2}\right)\\
&=& -\gamma_c \delta \df{\p \hat{F}_B}{\p \nu} - \gamma_m\big(\f{1}{2}I - \mathcal{K}_B^*\big)\widetilde{\mathcal{S}}_B^{-1}[\delta\hat{F}_B] +\gamma_m\widetilde{\mathcal{S}}_B^{-1}[\delta\hat{F}_B] + O\left(\f{\delta^5 \log\delta }{\textnormal{dist}(\lambda_{\eps},\sigma(\mathcal{K}^*_{D}))^2}\right).
\eeas
Here the remainder comes from the fact that $\hat{F}_B = O\left(\f{\delta^2 }{\textnormal{dist}(\lambda_{\eps},\sigma(\mathcal{K}^*_{D}))^2}\right)$.


Note that $\Delta \hat{F}_B = \eta(L(g_u)) - \delta^2\beta_{\gamma_c}^2\hat{F}_B$ in $B$ and $\Delta \hat{F}_B = 0$ in $\R^2\backslash \bar{D}$. We can further verify that $\hat{F}_B$ satisfies the assumption required in Lemma \ref{lem-d/dn(u) 1/2-K^*u heat}. Thus we have
\beas
\big(\f{1}{2}I - \mathcal{K}_B^*\big)\widetilde{\mathcal{S}}_B^{-1}[\delta\hat{F}_B] = -\delta \df{\p \hat{F}_B}{\p \nu} + C_u\varphi_0 +\gamma_m\widetilde{\mathcal{S}}_B^{-1}[\delta\hat{F}_B]  + O\left(\f{\delta^5 }{\textnormal{dist}(\lambda_{\eps},\sigma(\mathcal{K}^*_{D}))^2}\right),
\eeas
where $C_u$ is a constant such that $C_u = O\left(\f{\delta^3}{\textnormal{dist}(\lambda_{\eps},\sigma(\mathcal{K}^*_{D}))^2}\right)$.

After replacing the above in the expression of $f^h$ we find that
\be \begin{array}{lcl} \label{eq-middle equation heat}
\eta(\hat{q}) &=& (\mathcal{A}^h_{B}(\delta))^{-1}f^h \\ 
&=&  (\lambda_{\gamma} I - \mathcal{K}_B^*)^{-1}[\delta \df{\p \hat{F}_B}{\p \nu}] + \df{C_u \gamma_m}{(\gamma_c-\gamma_m)(\lambda_{\gamma}-\f{1}{2})}\varphi_0 + O\left(\df{\delta^5 \log\delta}{\textnormal{dist}(\lambda_{\eps},\sigma(\mathcal{K}^*_{D}))^2}\right),
\end{array}
\ee
where
\beas
\lambda_{\gamma} = \f{\gamma_c + \gamma_m}{2(\gamma_c - \gamma_m)}.
\eeas
Finally, in $\mathcal{H}^*(\p B)$, 
\be \label{eq-asymtotic density heat}
\eta(\hat{\tau}) = -\delta^2\hat{F}_B - \delta\mathcal{S}_B^{\delta\beta_{\gamma_c}}(\lambda_{\gamma} I - \mathcal{K}_B^*)^{-1}[\df{\p \delta\hat{F}_B}{\p \nu}] - \df{C_u \gamma_m}{(\gamma_c-\gamma_m)(\lambda_{\gamma}-\f{1}{2})}\delta\mathcal{S}_B^{\delta\beta_{\gamma_c}}[\varphi_0]+ O\left(\f{\delta^6 \log\delta }{\textnormal{dist}(\lambda_{\eps},\sigma(\mathcal{K}^*_{D}))^2}\right).
\ee
It can be shown, from the regularity of the remainders, that the previous identity also holds in $L^2(\p B)$. 

Using Holder's inequality we can prove that
\beas
\|\mathcal{S}_B^{\delta\beta_{\gamma_c}}[\varphi]\|_{L^{\infty}(\p B)} \leq C \|\varphi\|_{L^2(\p B)},
\eeas
for some constant $C$. Hence, we find that identity \eqref{eq-asymtotic density heat} also holds true uniformly on $\p B$ and $C_u \delta\mathcal{S}_B^{\delta\beta_{\gamma_cf}}[\varphi_0](\tilde{x}) = O\left(\f{\delta^4 \log\delta }{\textnormal{dist}(\lambda_{\eps},\sigma(\mathcal{K}^*_{D}))^2}\right)$, uniformly in $\p B$. Scaling back to $D$ gives
\bea \label{eq-temperature laplace heat}
\hat{\tau}(x,s) = -\hat{F}_D(x,\beta_{\gamma_c}) - \mathcal{S}_D^{\beta_{\gamma_c}}(\lambda_{\gamma} I - \mathcal{K}_D^*)^{-1}[\df{\p \hat{F}_D(\cdot,\beta_{\gamma_c})}{\p \nu}] +  O\left(\f{\delta^4 \log\delta }{\textnormal{dist}(\lambda_{\eps},\sigma(\mathcal{K}^*_{D}))^2}\right).
\eea

Before we take the inverse Laplace transform to \eqref{eq-temperature laplace heat} we note that (see \cite{LubichSchneider}) 
\beas
L\big(K(x,\cdot, b_c)\big) = -G(x,\beta_{\gamma_c}),
\eeas
where $b_c := \f{ \rho_c C_c}{\gamma_c}$ and $K(x,\cdot,b_c)$ is the fundamental solution of the heat equation. In dimension two, $K$ is given by
$$
K(x,t, \gamma)= \f{e^{-\f{|x|^2}{4 b_c t}}}{4\pi b_c t}.
$$
We denote $K(x,y,t,t', b_c):=K(x-y,t-t', b_c)$. By the properties of the Laplace transform, we have
\beas
-\hat{F}_D(x,\beta_{\gamma_c}) = -\int_{D}G(x,y,\beta_{\gamma_c})L(g_u)(y)dy = L\left(\int_0^{\cdot} \int_{D} K(x,y,\cdot,t', b_c) g_u(y) dy dt'\right).
\eeas
We define $F_D$ as follows
\bea \label{def-F heat}
F_D(x,t, b_c) := \int_0^t \int_{D} K(x,y,t,t', b_c) g_u(y) dy dt'.
\eea
Similarly, we have that for a function $f$
\beas
-\int_{\p D}G(x,y,\beta_{\gamma_c})L(f)(y)dy = L\left(\int_0^{\cdot} \int_{\p D} K(x,y,\cdot,t', b_c) f(y,t') dy dt'\right).
\eeas
We define $\mathcal{V}_D^{b_c}$ as follows
\bea \label{def-Single heat layer heat}
\mathcal{V}_D^{b_c}[f](x,t) := \int_0^t \int_{\p D} K(x,y,t,t', b_c) f(y,t') dy dt'.
\eea
Finally, using Fubini's theorem and taking Laplace inverse we find that
\beas
 \tau(x,t) = F_D(x,t,b_c) - \mathcal{V}_D^{b_c}(\lambda_{\gamma} I - \mathcal{K}_D^*)^{-1}[\df{\p F_D(\cdot,\cdot,b_c)}{\p \nu}](x,t) +  O\left(\f{\delta^4 \log\delta }{\textnormal{dist}(\lambda_{\eps},\sigma(\mathcal{K}^*_{D}))^2}\right),
\eeas
uniformly in $(x,t)\in \p D\times (0,T)$.

\subsection{Temperature elevation at the plasmonic resonance}

Suppose that the incident wave is $u^i(x) = e^{i k_m d \cdot x}$, where $d$ is a unit vector. For a nanoparticle occupying a domain $D = z + \delta B$, the inner field $u$ solution to \eqref{eq-Helm heat} is given by Theorem \ref{thm12d heat}, which states that, in $L^2(D)$, 
\beas
u \approx e^{ik_m d \cdot z}\big(1 + i k_m\mathcal{S}_D \big(\lambda_{\eps}I - \mathcal{K}_D^*\big)^{-1} [\nu]\cdot d\big),
\eeas
and hence
\be \label{eq-u^2 heat}
|u|^2 \approx 1 + 2 k_m\Re\left(i\mathcal{S}_D \big(\lambda_{\eps}I - \mathcal{K}_D^*\big)^{-1} [\nu]\cdot d\right) + \Big|k_m\mathcal{S}_D \big(\lambda_{\eps}I - \mathcal{K}_D^*\big)^{-1} [\nu]\big)\cdot d\Big|^2.
\ee
Using Lemma \ref{lem-K_star_properties2d heat}, we can write
\beas
\mathcal{S}_D \big(\lambda_{\eps}I - \mathcal{K}_D^*\big)^{-1} [\nu]\cdot d = \sum_{j=1}^{\infty} \f{(\nu\cdot d,\varphi_j)_{\mathcal{H}^*}\mathcal{S}_D[\varphi_j]}{\lambda_{\eps}-\lambda_j},
\eeas
and therefore, for a given plasmonic frequency $\om$,  we have
\beas
\mathcal{S}_D \big(\lambda_{\eps}I - \mathcal{K}_D^*\big)^{-1} [\nu]\cdot d \approx \f{(\nu\cdot d,\varphi_{j^*})_{\mathcal{H}^*}\mathcal{S}_D[\varphi_{j^*}]}{\lambda_{\eps}(\om)-\lambda_{j^*}}.
\eeas
Here $j^*$ is such that $\lambda_{j^*} = \Re(\lambda_{\eps}(\om))$ and the eignevalue $\lambda_{j^*}$ is assumed to be simple. If this was not the case, $(\nu\cdot d,\varphi_{j^*})_{\mathcal{H}^*}\mathcal{S}_D[\varphi_{j^*}]$ should be replaced by the corresponding sum over an orthonormal basis of eigenfunctions for the eigenspace associated to $\lambda_{j^*}$.

Replacing in \eqref{eq-u^2 heat} we find
\beas
|u|^2 \approx 1 + 2k_m\f{(\nu\cdot d,\varphi_{j^*})_{\mathcal{H}^*}\mathcal{S}_D[\varphi_{j^*}]}{|\lambda_{\eps}(\om)-\lambda_{j^*}|} + k_m^2\f{(\nu\cdot d,\varphi_{j^*})^2_{\mathcal{H}^*}\mathcal{S}_D[\varphi_{j^*}]^2}{|\lambda_{\eps}(\om)-\lambda_{j^*}|^2}.
\eeas

Thus, at a plasmonic resonance $\om$,
\beas \label{eq-conv noyau chaleur at plasmonic heat}
F_D[g_u](x,t, b_c) &\approx & \left(F_D[1] + 2k_m\f{(\nu\cdot d,\varphi_{j^*})_{\mathcal{H}^*}}{|\lambda_{\eps}(\om)-\lambda_{j^*}|}F_D[\mathcal{S}_D[\varphi_{j^*}]] + k_m^2\f{(\nu\cdot d,\varphi_{j^*})^2_{\mathcal{H}^*}}{|\lambda_{\eps}(\om)-\lambda_{j^*}|^2}F_D[\mathcal{S}_D[\varphi_{j^*}]^2]\right)(x,t, b_c) ,\\
\df{\p F_D(x,t,b_c)}{\p \nu} &\approx & \left(2k_m\f{(\nu\cdot d,\varphi_{j^*})_{\mathcal{H}^*}}{|\lambda_{\eps}(\om)-\lambda_{j^*}|}\df{\p F_D[\mathcal{S}_D[\varphi_{j^*}]]}{\p \nu}+ k_m^2\f{(\nu\cdot d,\varphi_{j^*})^2_{\mathcal{H}^*}}{|\lambda_{\eps}(\om)-\lambda_{j^*}|^2}\df{\p F_D[\mathcal{S}_D[\varphi_{j^*}]^2]}{\p \nu}\right)(x,t, b_c).
\eeas
Then, the temperature on the boundary of a nanoparticle at the plasmonic resonance can be estimated by plugging the above approximations of $F_D$ and $\df{\p F_D(x,t,b_c)}{\p \nu} $ into
\beas
  \tau(x,t) = F_D(x,t,b_c) - \mathcal{V}_D^{b_c}(\lambda_{\gamma} I - \mathcal{K}_D^*)^{-1}[\df{\p F_D(\cdot,\cdot,b_c)}{\p \nu}](x,t) +  O\left(\f{\delta^4 \log\delta }{\textnormal{dist}(\lambda_{\eps},\sigma(\mathcal{K}^*_{D}))^2}\right).
\eeas

\subsection{Temperature elevation for two close-to-touching particles}

Lemma \ref{lem-d/dn(u) 1/2-K^*u heat} implies that $$\df{\p F_D(x,t,b_c)}{\p \nu} = -\big(\f{1}{2}I - \mathcal{K}_D^*\big)\widetilde{\mathcal{S}}_D^{-1}[F_D](x,t) + O\left(\f{\delta^4 \log\delta }{\textnormal{dist}(\lambda_{\eps},\sigma(\mathcal{K}^*_{D}))^2}\right).$$
Therefore, we can write the temperature on the boundary of the nanoparticle as
\be \label{eq-temperature bord nanoparticle heat}
\tau(x,t) = F_D(x,t,b_c) + \mathcal{V}_D^{b_c}(\lambda_{\gamma} I - \mathcal{K}_D^*)^{-1}\mathcal{P}_{\mathcal{H}^*\backslash E_{\f{1}{2}}}[\df{\p F_D(\cdot,\cdot,b_c)}{\p \nu}](x,t) +  O\left(\f{\delta^4 \log\delta }{\textnormal{dist}(\lambda_{\eps},\sigma(\mathcal{K}^*_{D}))^2}\right),
\ee
where $\mathcal{P}_{\mathcal{H}^*\backslash E_{\f{1}{2}}}$ is the projection into $ \mathcal{H}^*\backslash E_{\f{1}{2}}$: the complement in $\mathcal{\mathcal{H}^*(\p D)}$ of the eigenspace associated to the eigenvalue $\f{1}{2}$ of $\mathcal{K}^*_D$. This implies that, even if $\lambda_{\gamma}$ is close to $\f{1}{2}$, the quantity $(\lambda_{\gamma} I - \mathcal{K}_D^*)^{-1}\mathcal{P}_{\mathcal{H}^*\backslash E_{\f{1}{2}}}[\df{\p F_D(\cdot,\cdot,b_c)}{\p \nu}](x,t)$ will remain of order $O\left(\f{\delta^2 }{\textnormal{dist}(\lambda_{\eps},\sigma(\mathcal{K}^*_{D}))^2}\right)$, provided that the second largest eigenvalue of $\mathcal{K}_D^*$ is not close to $\f{1}{2}$. 

Even if this is in general the case for smooth boundaries $\p D$, it turns out that for nanoparticles with two connected close-to-touching subparts with contact of order $m$, a family of eigenvalues of $\mathcal{K}^*_D$ in $\mathcal{H}^*\backslash E_{\f{1}{2}}$ approaches $\f{1}{2}$ as (see \cite{triki})
\beas
\lambda_n^\zeta \sim \f{1}{2} - c_n\zeta^{1-\f{1}{m}} + o(\zeta^{1-\f{1}{m}}),
\eeas 
where $\zeta$ is the distance between connected subparts and $c_n$ is an increasing sequence of positive numbers.

Now, $\lambda_{\gamma}\approx \f{1}{2}$ is the kind of situations encountered for metallic nanoparticles immersed in water or some biological tissue. As an example, the thermal conductivity of gold is $\gamma_c = 318 \f{W}{m K}$ and that of pure water is $\gamma_m = 0.6 \f{W}{m K}$. This gives $\lambda_{\gamma} \approx 0.5019$.

In view of this, the second term in \eqref{eq-temperature bord nanoparticle heat} may increase considerably for some type of close-to-touching particles.

We stress, nevertheless, that this is not the general case. For a more refined analysis, asymptotics of the eigenfunctions of $\mathcal{K}_D^*$ should be also studied.

\section{Numerical results} \label{sec-numeric heat}

The numerical experiments for this work can be divided into two  parts. The first one is the Helmholtz equation solution approximation, which is obtained by using Theorem \ref{thm12d heat}. The second part is the Heat equation solution computation, which is obtained using Theorem \ref{thm-Heat small volume heat}. 

The major tasks surrounding the numerical implementation of these formulas are integrating against a singular kernel. The numerical computations of the operators $F_D[\cdot]$ and $\partial_\nu F_D[\cdot]$ can be achieved by meshing the domain $D$ and integrating semi-analytically inside the triangles that are close to the singularities. We used  the following formula to avoid numerical differentiation: 

\begin{equation}
\label{partialF} 
\f{\p F_D (x,t,b_c)}{\p \nu} = \frac{1}{2\pi b_c} \displaystyle \int_D \exp\left(\frac{-|x-y|^2}{4 b_c t}\right) \frac{\left< y-x,\nu_x \right>}{|x-y|^2} g_u(y) dy, \quad x\in \partial D.
\end{equation}

For all the presented simulations, we considered an incident plane wave given by $$u^i(x) = e^{i k_m d \cdot x},$$ where $d = (1,1) / \sqrt{2} \in \mathbb{R}^2$ is the illumination direction and $k_m = 2\pi/750\cdot 10^9$ is the frequency (in the red range). The considered nanoparticles are ellipses with semi-axes $30nm$ and $20nm$, respectively.

It is worth noticing that the illumination direction $d$ is relevant solely in the asymptotic formula in  Theorem \ref{thm12d heat}. Its role is to define the coefficients of a linear combination of both components of $\mathcal{S}_D(\lambda_\epsilon I - \mathcal{K}_D^*)^{-1} [v] \in \mathbb{R}^2$. We will see from the numerical simulations that this is fundamental if we wish to maximize the produced electromagnetic field, and therefore the generated heat inside the nanoparticles.

With respect to the asymptotic formula established in Theorem \ref{thm12d heat}, besides the nanoparticle's shape $D$, the sole parameter that is left is $\lambda_\epsilon$. For all the following simulations we will consider this as a free parameter that we will use to excite the eigenvalues of the Neumann-Poincar\'e operator and hence to generate resonances. The physical justification that allows us to do this is based on the Drude model \cite{pierre}. Whenever we mention that we approach a particular eigenvalue $\lambda_j$ of $\mathcal{K}_D^*$, we will adopt $\lambda_\epsilon = \lambda_j + 0.001i$.

With respect to the heat equation coefficients, we use realistic values of gold for nanoparticles, and water for tissues.

\subsection{ Single-particle simulation }

We consider one elliptical nanoparticle $D \Subset \mathbb{R}^2$ centered at the origin, with its  semi-major axis aligned with the $x$-axis.

\subsubsection{ Single-particle Helmholtz resonance  }

Resonance is achieved by approaching the eigenvalues of the Neumann-Poincar\'e operator $\mathcal{K}_D^*$ with $\lambda_\epsilon$, and afterwards applying it to  each of the components of the normal  $\nu$ to $\partial D$. It turns out that for some eigenfunctions of $\mathcal{K}_D^*$, the normal of the shape is almost orthogonal, in $\mathcal{H}^*(\p D)$, to them. Therefore, we cannot observe resonance for their associated eigenvalues; see \cite{algebraic}. In Figure \ref{fig:1pResonance} we can see values of the inner product between the eigenfunctions of $\mathcal{K}_D^*$ and the  components $\nu_x$ and $\nu_y$ of $\nu$. Figure \ref{fig:1pResonance}  suggests us which are the available resonant modes with the respective strength of each coordinate.
\begin{figure}[h!]
	\centering
	\begin{minipage}[c]{.6\textwidth}
		\centering
		\includegraphics[width=\linewidth]{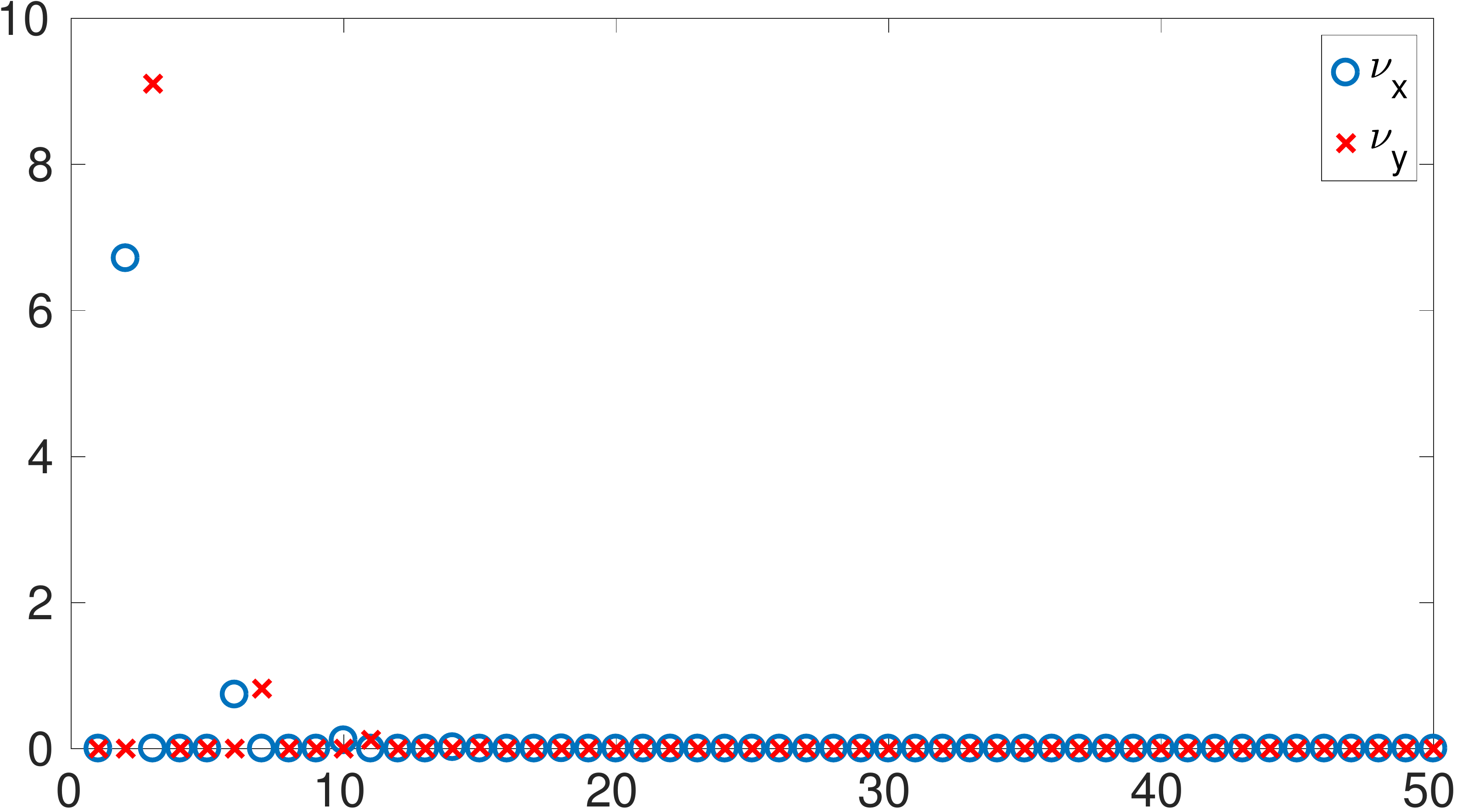}
	\end{minipage}
	\caption{Inner product in $\mathcal{H}^*(\p D)$ between the eigenfunctions of $\mathcal{K}_D^*$ and the components $\nu_x$ and $\nu_y$ of the normal $\nu$ to $\partial D$. }
	\label{fig:1pResonance}
\end{figure}
In Figure \ref{fig:1pHelmResonance} we present the absolute value of the inner field for the first three resonant modes, corresponding to the second, third and sixth eigenvalue of $\mathcal{K}_D^*$, respectively. In Figure \ref{fig:1pHelmResonance2} we decompose the inner field into the zeroth-order and the first-order terms respectively given by $u^i(z) + \delta(x-z) \nabla u^{i}(z)$ and $\mathcal{S}_D \big(\lambda_{\eps}I - \mathcal{K}_D^*\big)^{-1} [\nu]\cdot \nabla u^{i}(z)$.  Figure \ref{fig:1pHelmResonance3}  shows  the components of the vector   $\mathcal{S}_D(\lambda_\epsilon I - \mathcal{K}_D^*)^{-1} [v]$. 

\begin{figure}[h!]
	\centering
	\begin{minipage}[c]{.28\textwidth}
		\centering
		\scriptsize First resonance mode \ \ \ \\ \
	\end{minipage}
	\hspace{.5cm}
	\begin{minipage}[c]{.28\textwidth}
		\centering
		\scriptsize Second resonance mode  \ \ \ \\ \
	\end{minipage}
	\hspace{.5cm}
	\begin{minipage}[c]{.28\textwidth}
		\centering
		\scriptsize Third resonance mode  \ \ \ \\ \
	\end{minipage}
	\\
	\begin{minipage}[c]{.28\textwidth}
		\centering
		\includegraphics[width=\linewidth]{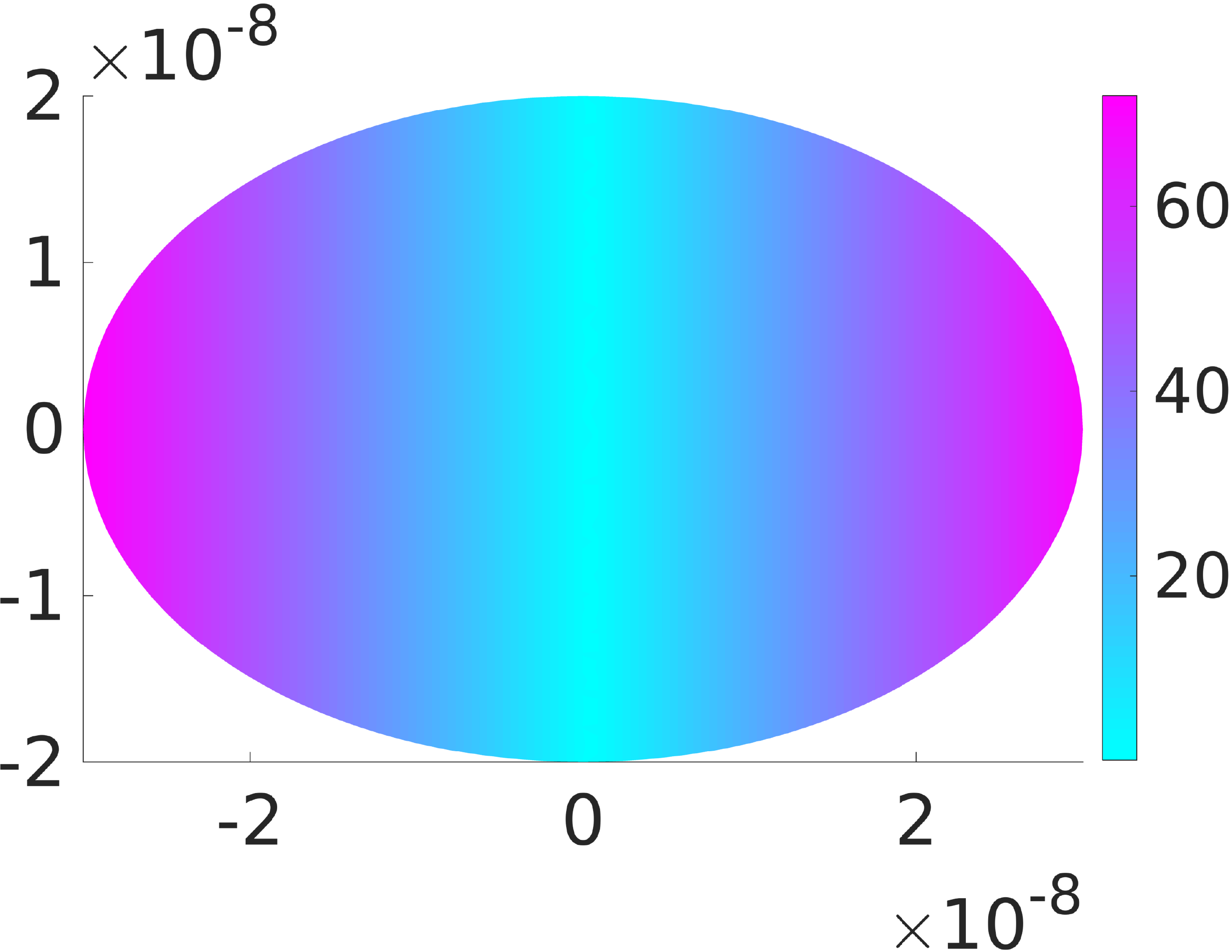}
	\end{minipage}
	\hspace{.5cm}
	\begin{minipage}[c]{.28\textwidth}
		\centering
		\includegraphics[width=\linewidth]{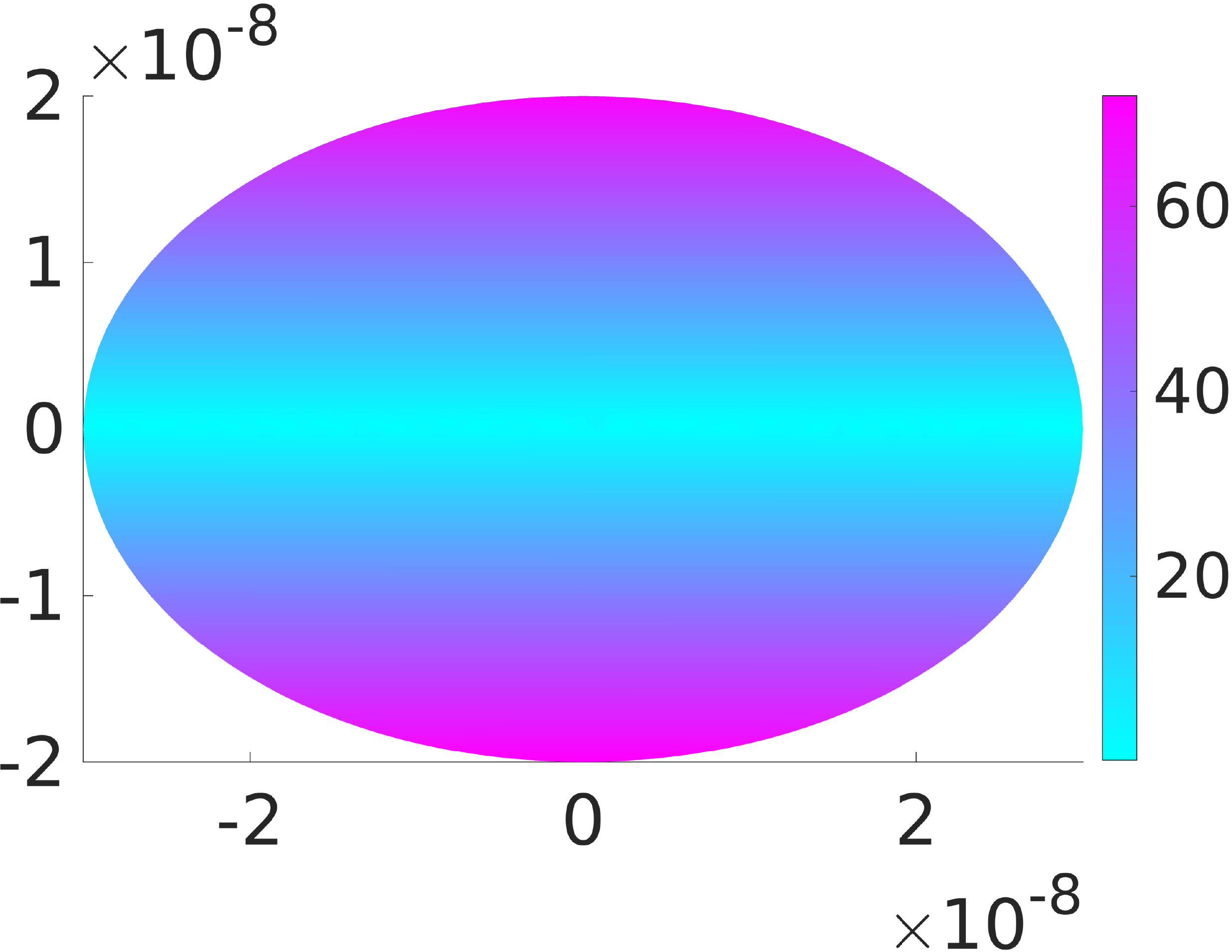}
	\end{minipage}
	\hspace{.5cm}
	\begin{minipage}[c]{.28\textwidth}
		\centering
		\includegraphics[width=\linewidth]{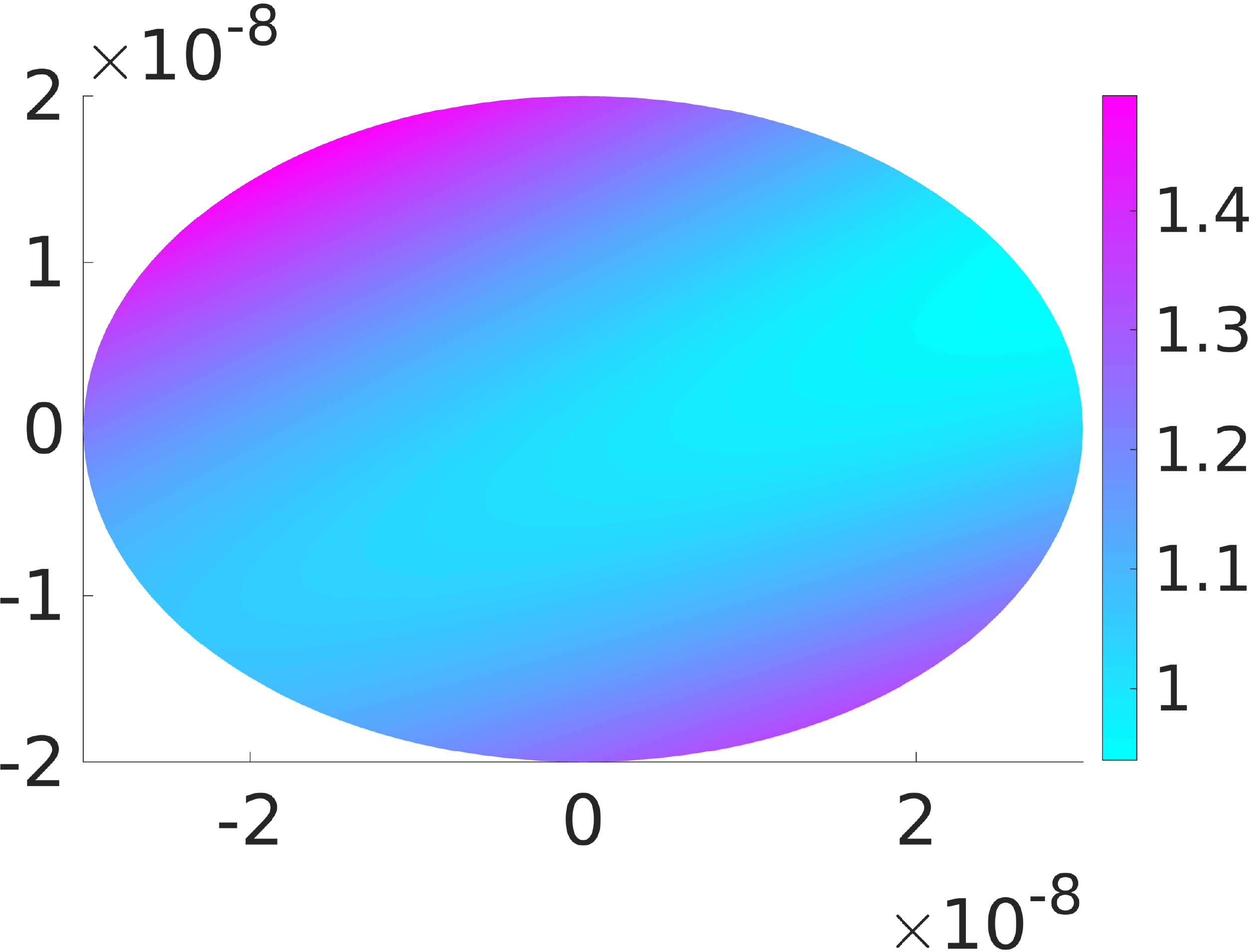}
	\end{minipage}
	\caption{Absolute value of the electromagnetic field inside the nanoparticle at the first resonant modes, being those when $\lambda_\epsilon$ approaches the second, third and sixth eigenvalue of $\mathcal{K}_D^*$.}
	\label{fig:1pHelmResonance}
\end{figure}

\begin{figure}[h!]
	\centering
	\begin{minipage}[c]{.28\textwidth}
		\centering
		\scriptsize Zeroth-order component  \ \ \
	\end{minipage}
	\hspace{.5cm}
	\begin{minipage}[c]{.28\textwidth}
		\centering
		\scriptsize First-order component \ \ \
	\end{minipage}
	\hspace{.5cm}
	\\
	\begin{minipage}[c]{.28\textwidth}
		\centering
		\includegraphics[width=\linewidth]{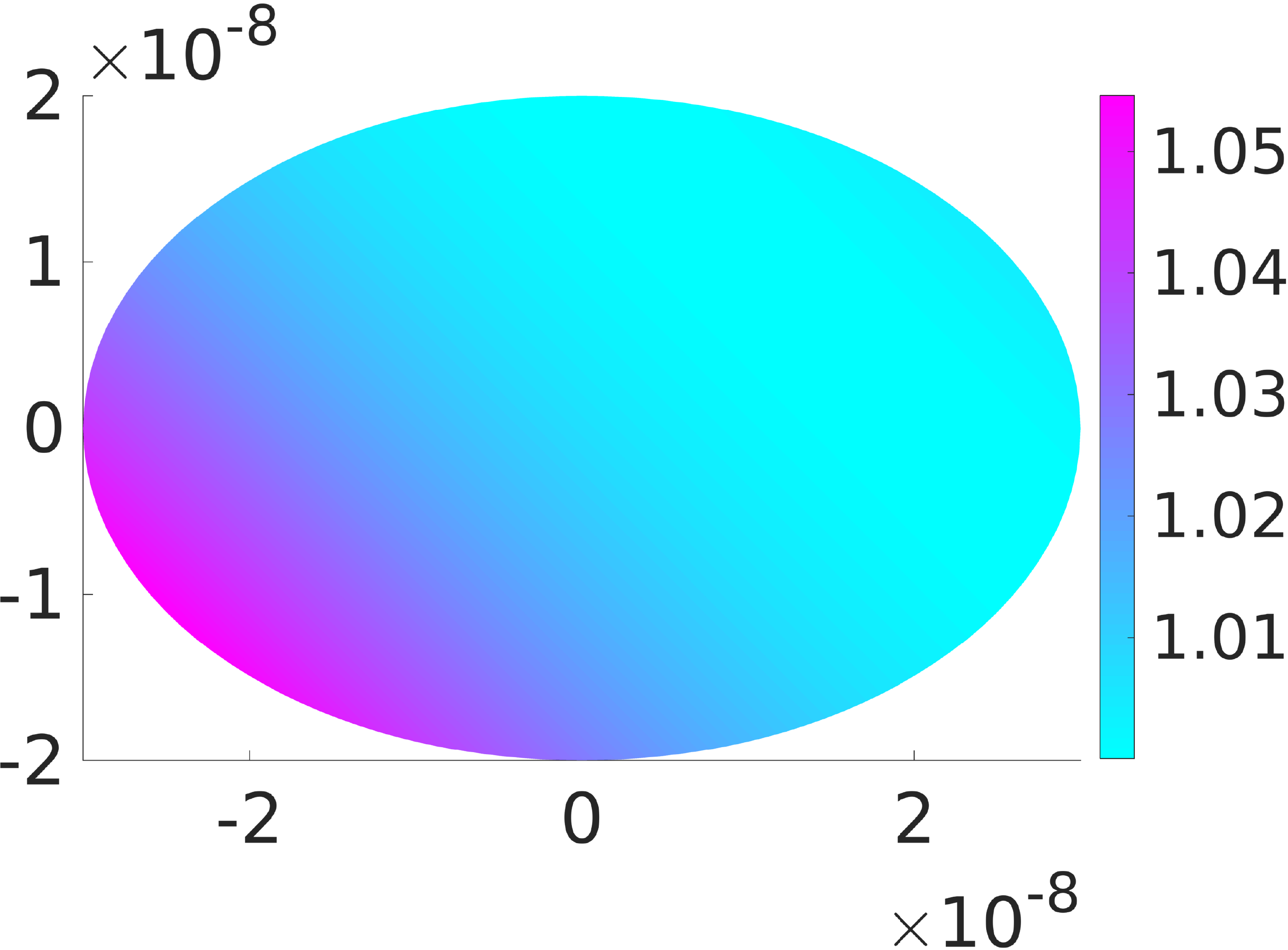}
	\end{minipage}
	\hspace{.5cm}
	\begin{minipage}[c]{.28\textwidth}
		\centering
		\includegraphics[width=\linewidth]{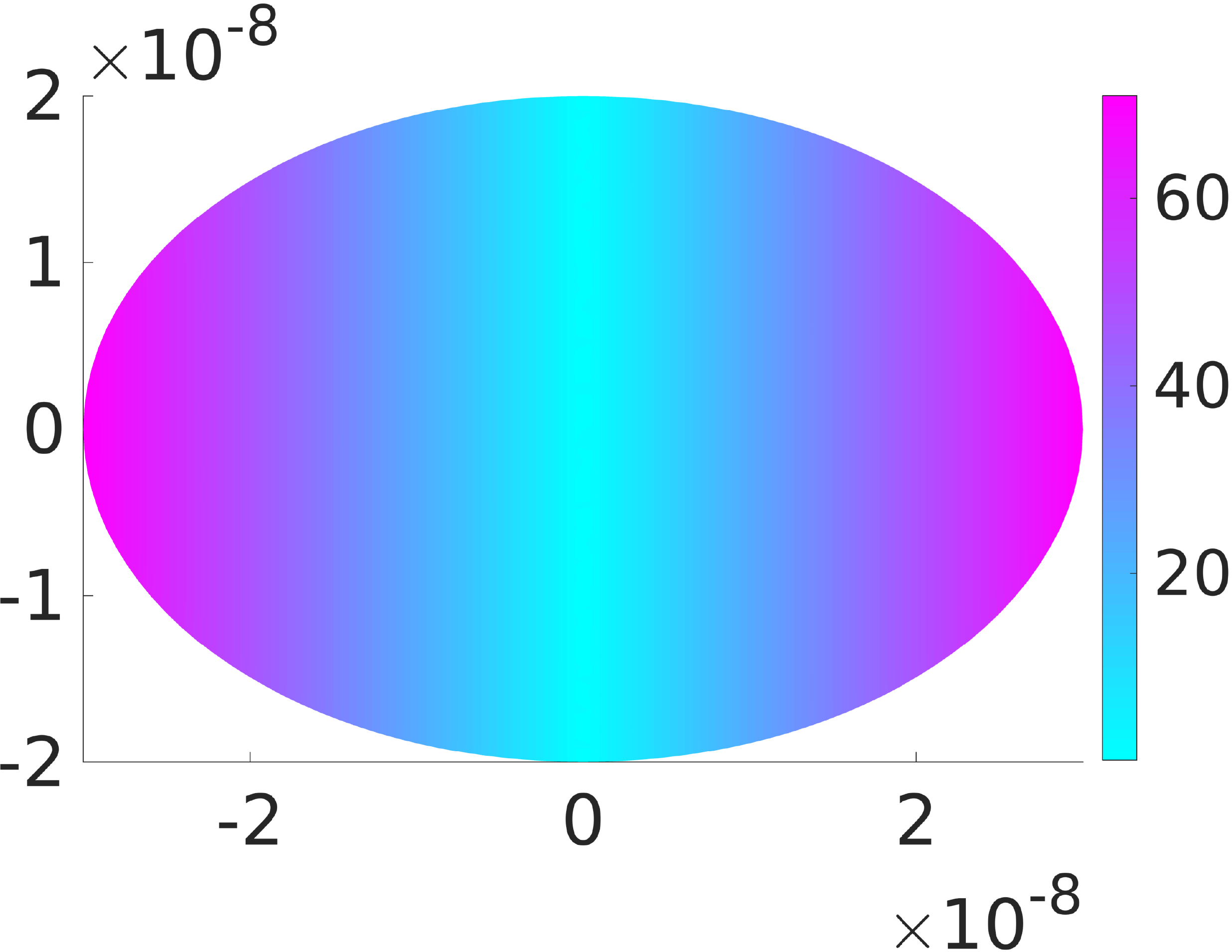}
	\end{minipage}
	\hspace{.5cm}
	\caption{First resonant mode of the nanoparticle decomposed in its first- and second-order term in the formula given by Theorem \ref{thm12d heat}. Both images are absolute values of the respective component.}
	\label{fig:1pHelmResonance2}
\end{figure}

\begin{figure}[h!]
	\centering
	\begin{minipage}[c]{.28\textwidth}
		\centering
		\scriptsize The $x$-component  \ \ \ \ \ \ \ \ \
	\end{minipage}
	\hspace{.5cm}
	\begin{minipage}[c]{.28\textwidth}
		\centering
		\scriptsize The $y$-component \ \ \ \ \ \ \ \ \
	\end{minipage}
	\hspace{.5cm}
	\\
	\begin{minipage}[c]{.28\textwidth}
		\centering
		\includegraphics[width=\linewidth]{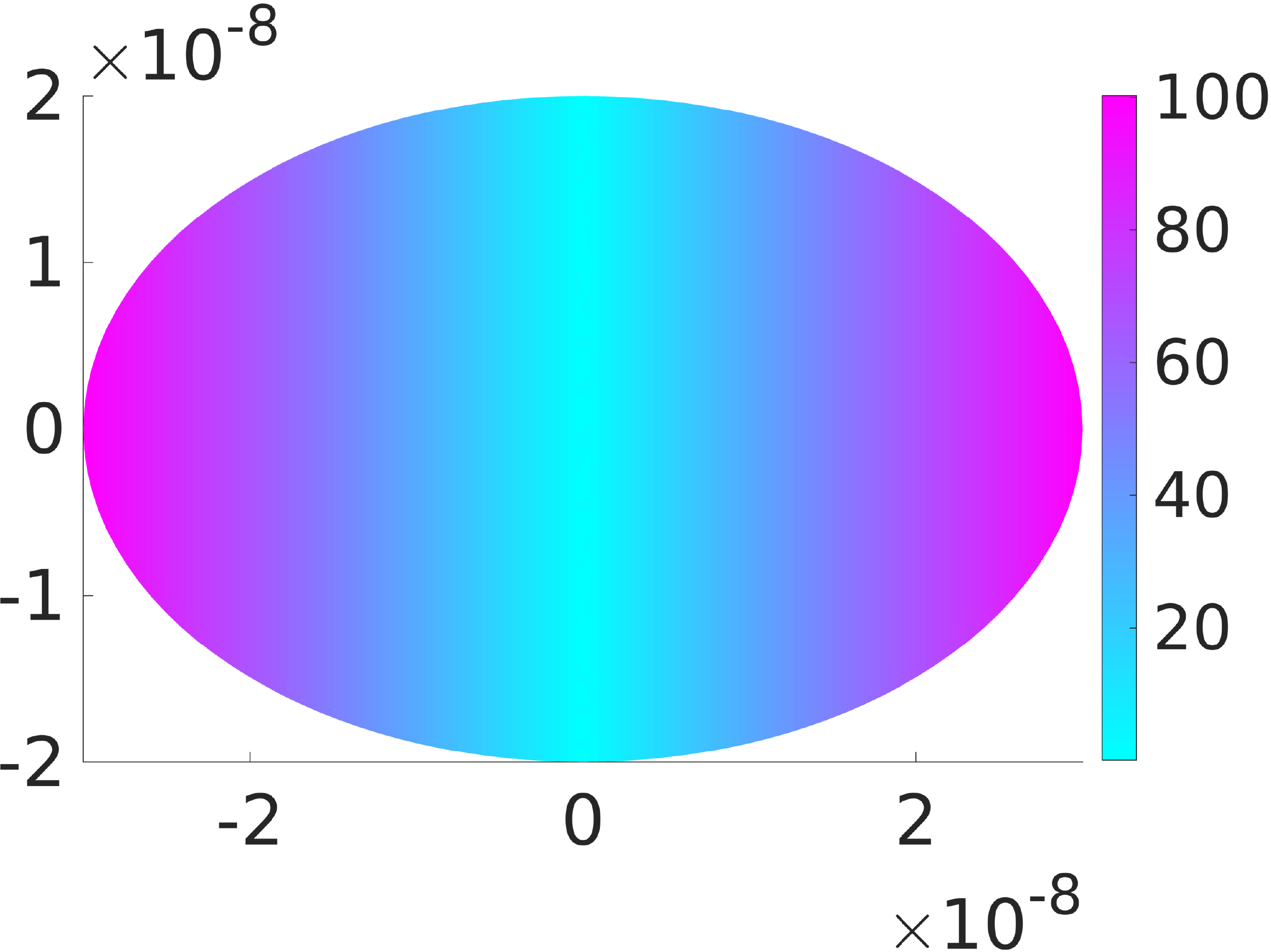}
	\end{minipage}
	\hspace{.5cm}
	\begin{minipage}[c]{.28\textwidth}
		\centering
		\includegraphics[width=\linewidth]{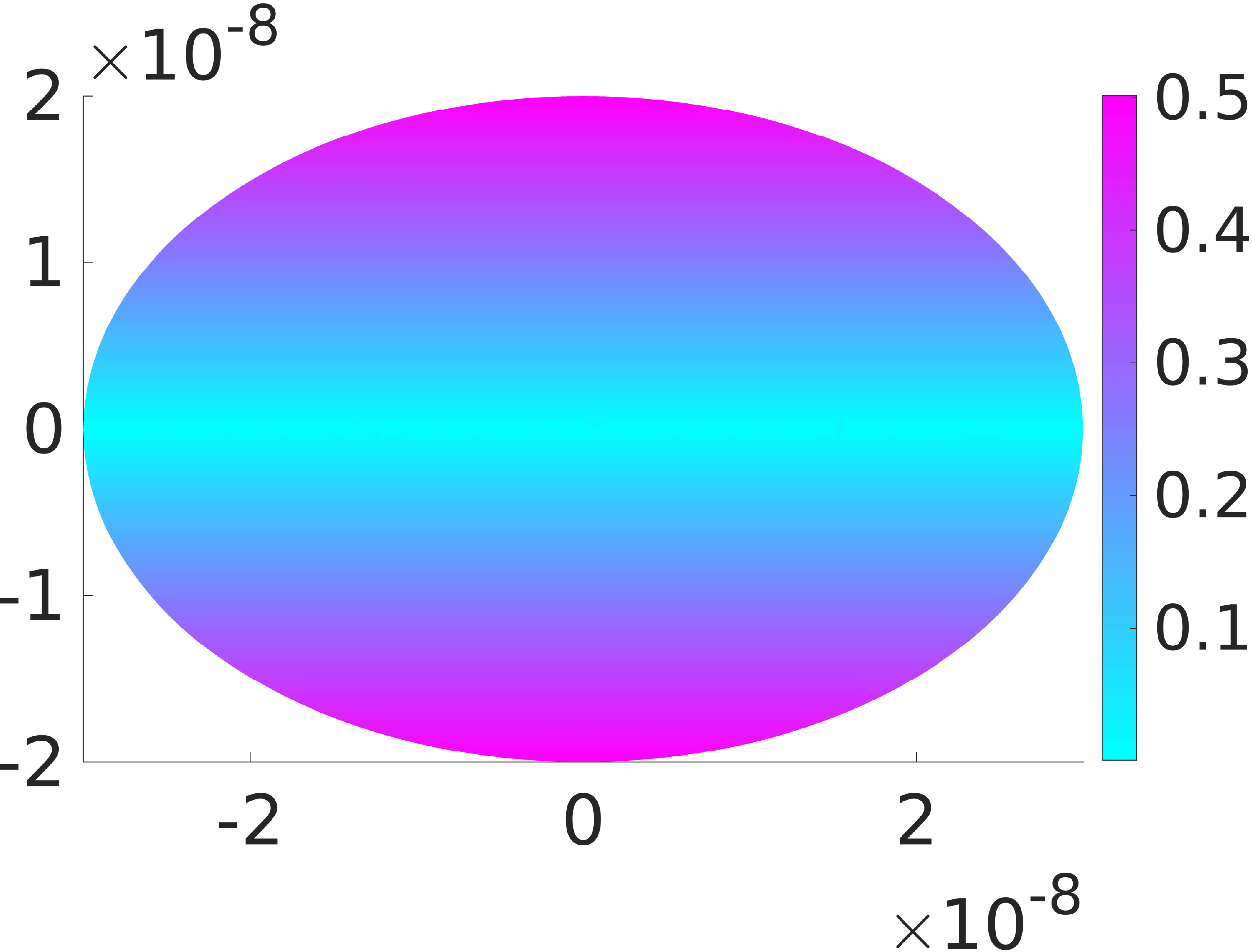}
	\end{minipage}
	\hspace{.5cm}
	\caption{Absolute value of the vectorial components of the first-order term for the first resonant mode.}
	\label{fig:1pHelmResonance3}
\end{figure}

From Figure \ref{fig:1pHelmResonance2}, we can see that when we excite the nanoparticle at its resonant mode, the largest contribution to the electromagnetic field comes from the first-order term of the small volume expansion formula established in Theorem \ref{thm12d heat}. 

Observing the vectorial components of the first-order term in Figure \ref{fig:1pHelmResonance3} tells us how important is the illumination direction as the $x$-component is significantly stronger than the $y$-component. If we wish to maximize the electromagnetic field and therefore the generated heat, the recommended illumination direction would be around $d = (1,0)^t$ (with $t$ being the transpose), as it was initially suggested by Figure \ref{fig:1pResonance}.

\subsubsection{Single-particle surface heat generation }

Considering the electromagnetic field inside the nanoparticle given by the first resonant mode presented in Figure \ref{fig:1pHelmResonance}, following the formula given by Theorem 
\ref{thm-Heat small volume heat}, we compute the generated heat on the surface of the nanoparticle. In Figure \ref{fig:1pHeat} we plot the generated heat in three dimensions and present a two dimensional plot obtained by parameterizing the boundary. In Figure \ref{fig:1pHeat2} we decompose the heat in its first- and second-order terms given by formula \ref{thm-Heat small volume heat}, being  $F_D(x,t,b_c)$ and $-\mathcal{V}_D^{b_c}(\lambda_{\gamma} I - \mathcal{K}_D^*)^{-1}[\df{\p F_D(\cdot,\cdot,b_c)}{\p \nu}](x,t)$ respectively. In Figure \ref{fig:1pHeat3}, we integrate the total heat on the boundary and  plot it as a function of time, for each component.

\begin{figure}[h!]
	\centering
	\begin{minipage}[c]{.35\textwidth}
		\centering
		\scriptsize 3D plot of generated heat at time $T = 1 $ \ \ \ \ \ \ \ \ \
	\end{minipage}
	\hspace{.5cm}
	\begin{minipage}[c]{.35\textwidth}
		\centering
		\scriptsize 2D plot of generated heat at time $T=1$ \ \ \ \ \ \ \ \ \
	\end{minipage}
	\hspace{.5cm}
	\\
	\begin{minipage}[c]{.35\textwidth}
		\centering
		\includegraphics[width=\linewidth]{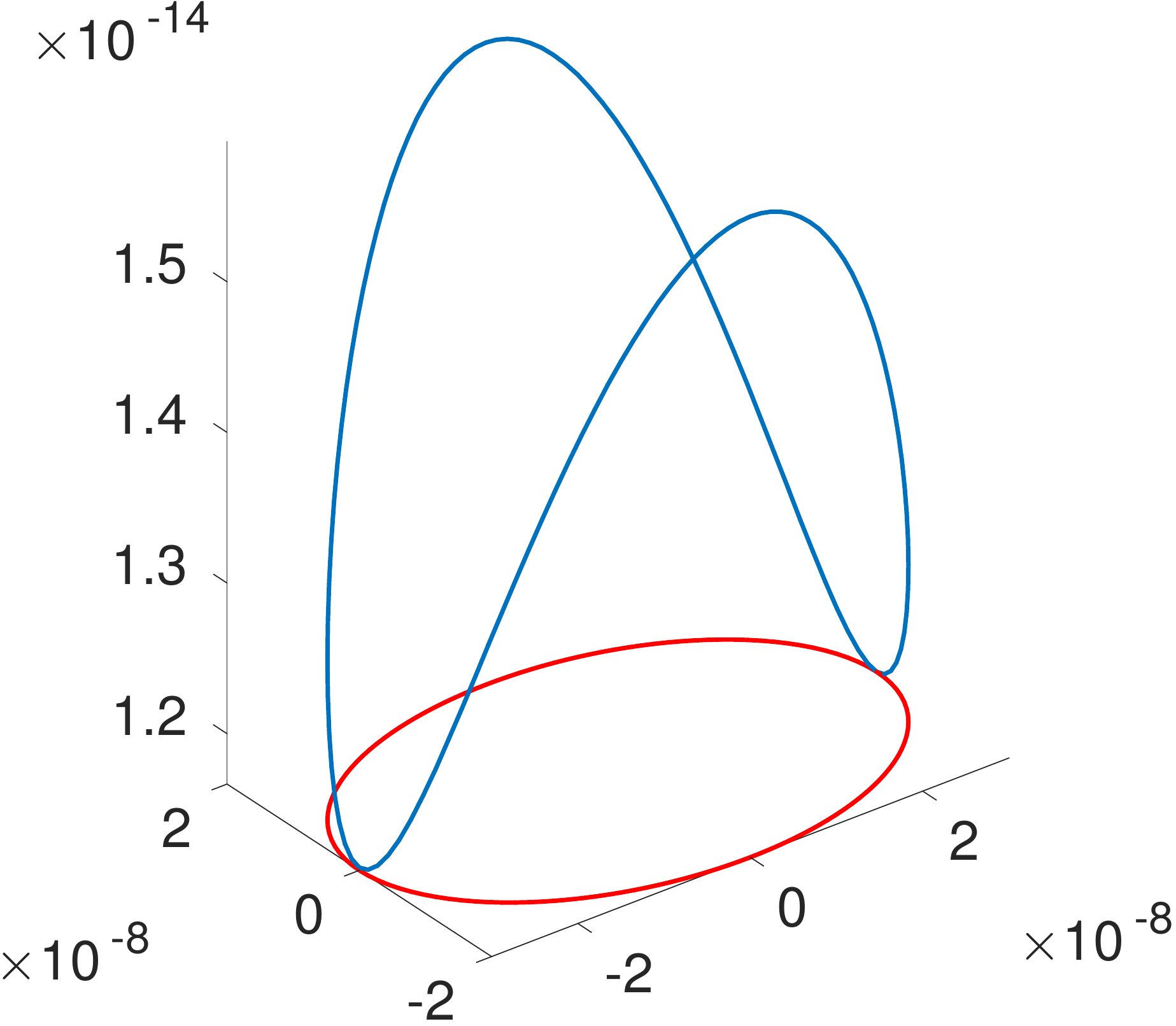}
	\end{minipage}
	\hspace{.5cm}
	\begin{minipage}[c]{.35\textwidth}
		\centering
		\includegraphics[width=\linewidth]{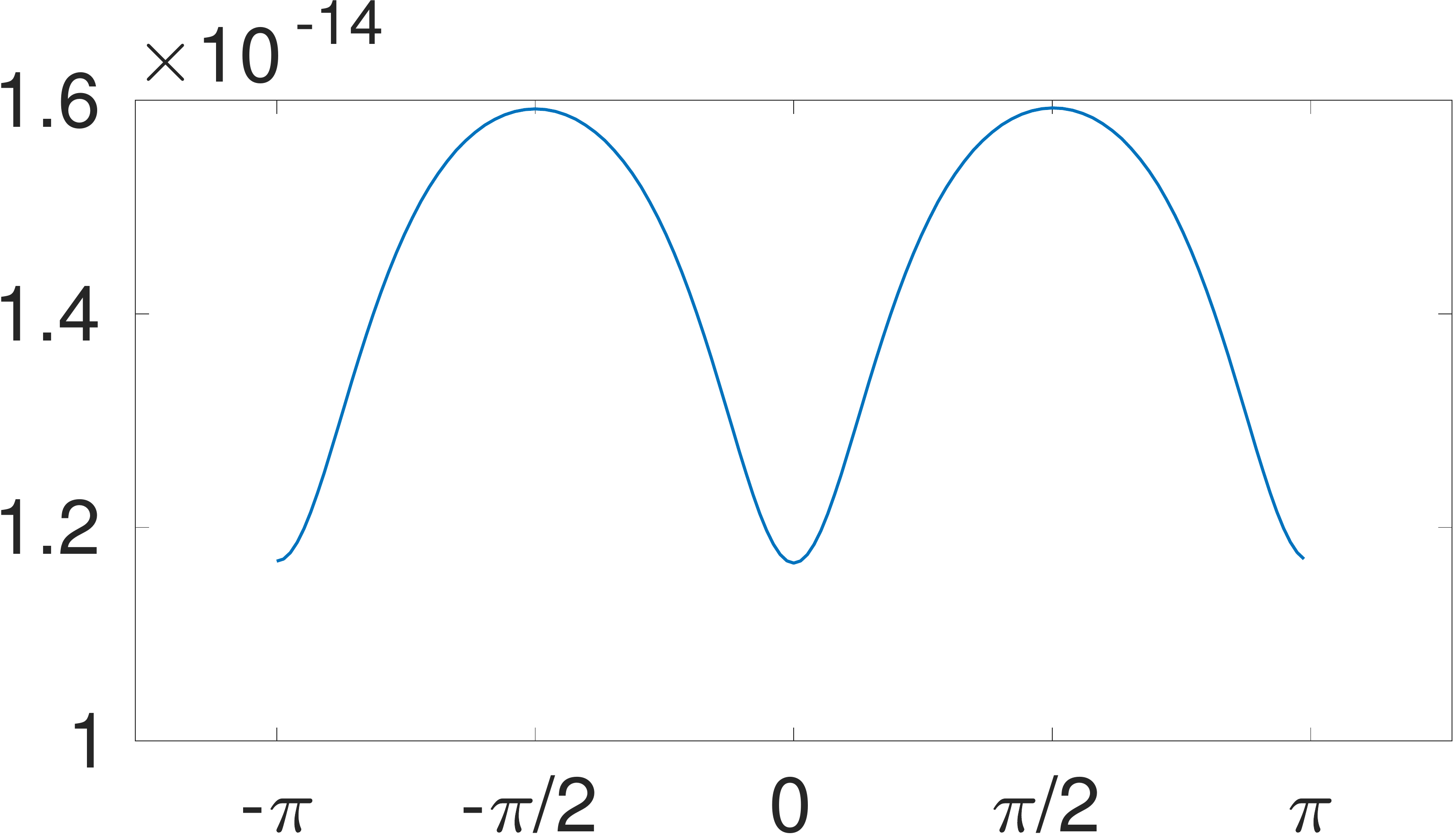}
	\end{minipage}
	\hspace{.5cm}
		\caption{At the left-hand side, we can see a three-dimensional plot of the nanoparticle heat, the red shape is a reference value to show where the nanoparticle is located. At the right-hand side we can see a two-dimensional plot of the 
		generated heat, where the boundary was parametrized following $p(\theta) = (a\cos(\theta), b\sin(\theta)), \ \theta \in [-\pi,\pi]$, with $a$ and $b$ being the semi-major and semi-minor axes, respectively.}
	\label{fig:1pHeat}
\end{figure}

\begin{figure}[h!]
	\centering
	\begin{minipage}[c]{.35\textwidth}
		\centering
		\scriptsize Two-dimensional plot of the zeroth-order term at $T=1$ \ \ \ \ \ \ \ \ \
	\end{minipage}
	\hspace{.5cm}
	\begin{minipage}[c]{.35\textwidth}
		\centering
		\scriptsize Two-dimensional plot of the first-order term at $T=1$  \ \ \ \ \ \ \ \ \
	\end{minipage}
	\hspace{.5cm}
	\\
	\begin{minipage}[c]{.35\textwidth}
		\centering
		\includegraphics[width=\linewidth]{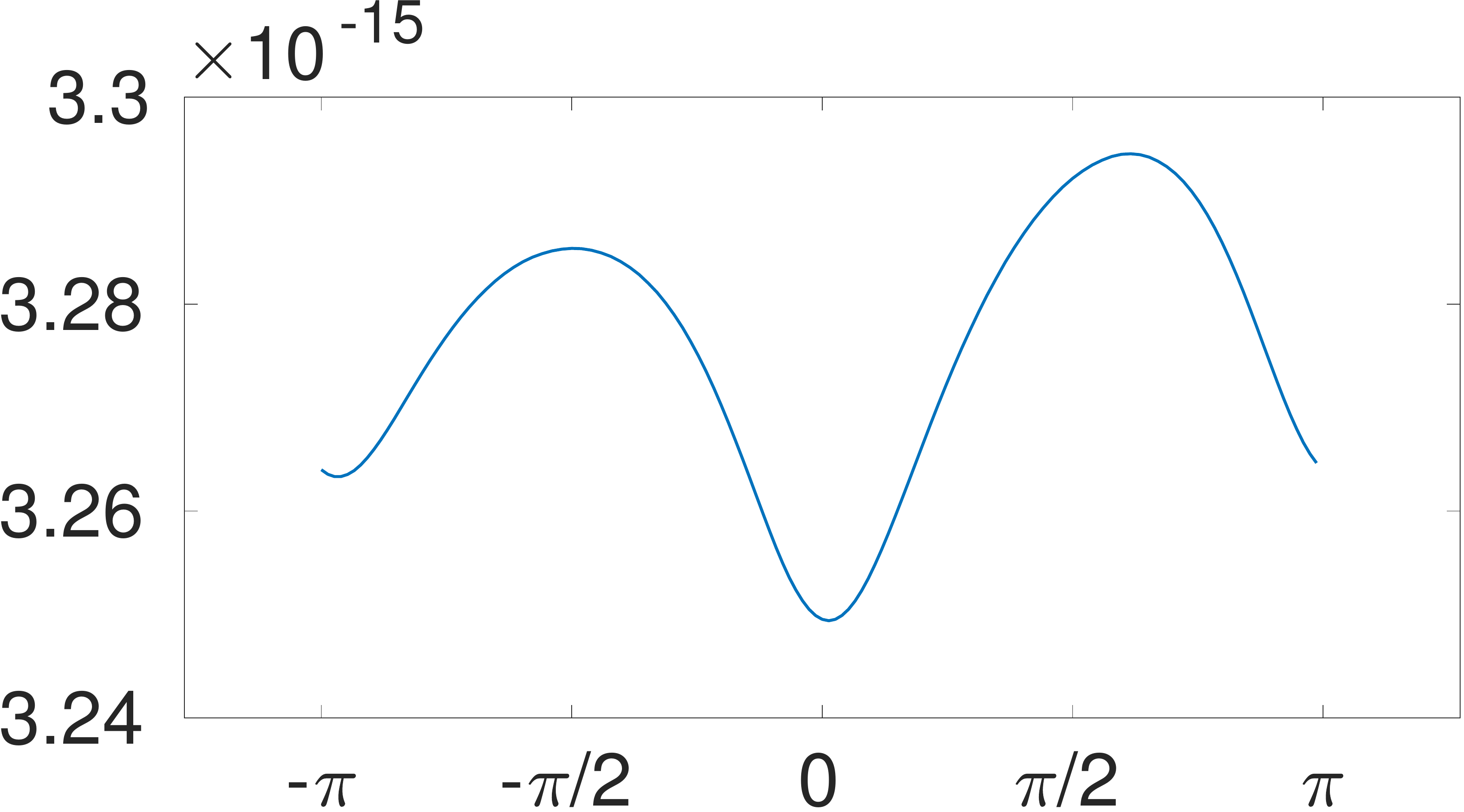}
	\end{minipage}
	\hspace{.5cm}
	\begin{minipage}[c]{.35\textwidth}
		\centering
		\includegraphics[width=\linewidth]{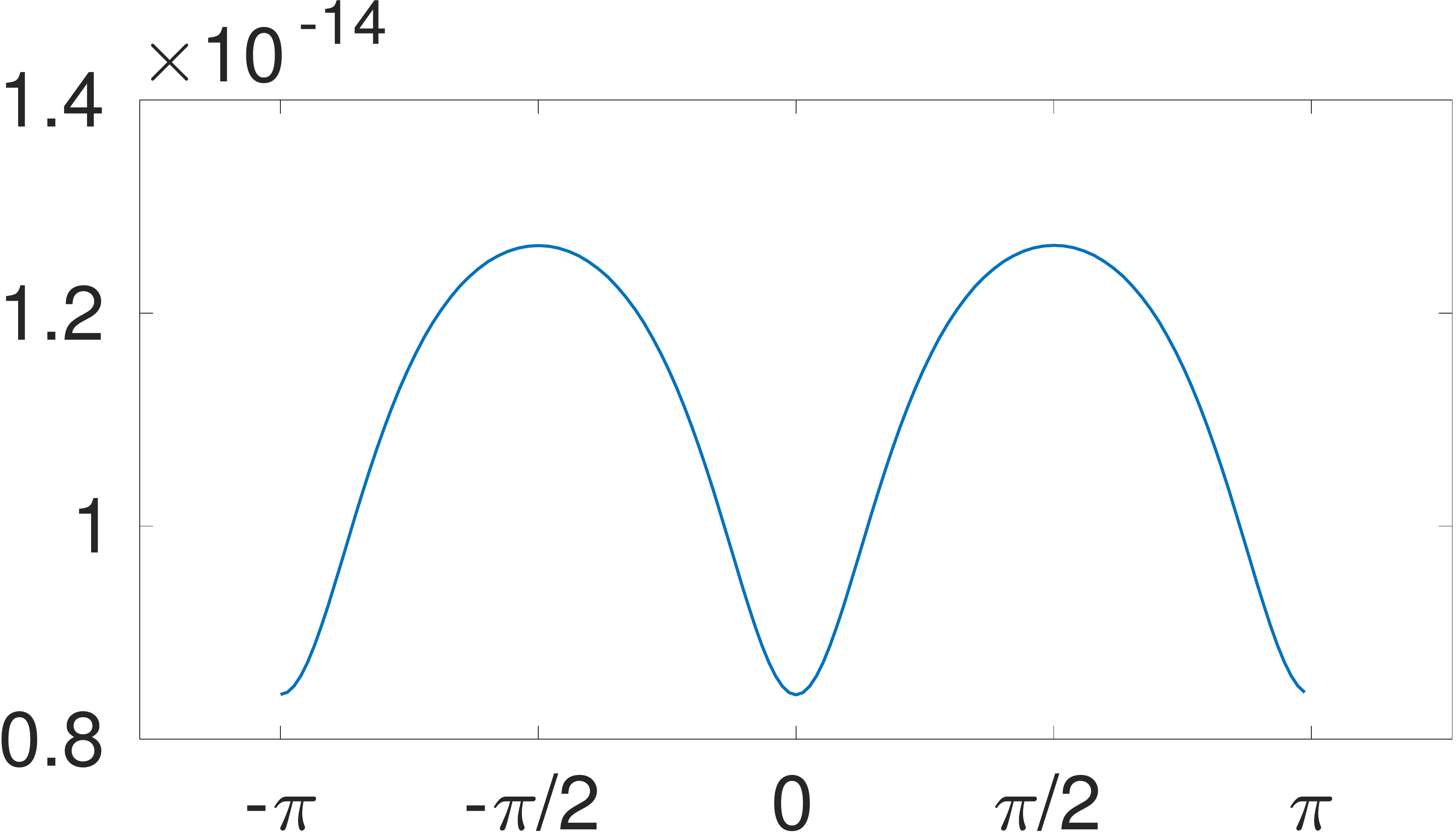}
	\end{minipage}
	\hspace{.5cm}
	\caption{Two-dimensional plots of the zeroth- and first-order components of the heat on the boundary when time is equal to one. As time goes on, each point of the graph increases, but the general shape is preserved.}
	\label{fig:1pHeat2}
\end{figure}

\begin{figure}[h!]
	\centering
	\begin{minipage}[c]{.28\textwidth}
		\centering
		\scriptsize Integrated heat over time \ \ \
	\end{minipage}
	\hspace{.5cm}
	\begin{minipage}[c]{.28\textwidth}
		\centering
		\scriptsize Integrated zeroth-order component of heat over time  \ \ \
	\end{minipage}
	\hspace{.5cm}
	\begin{minipage}[c]{.28\textwidth}
		\centering
		\scriptsize Integrated first-order component of heat over time \ \ \
	\end{minipage}
	\\
	\begin{minipage}[c]{.28\textwidth}
		\centering
		\includegraphics[width=\linewidth]{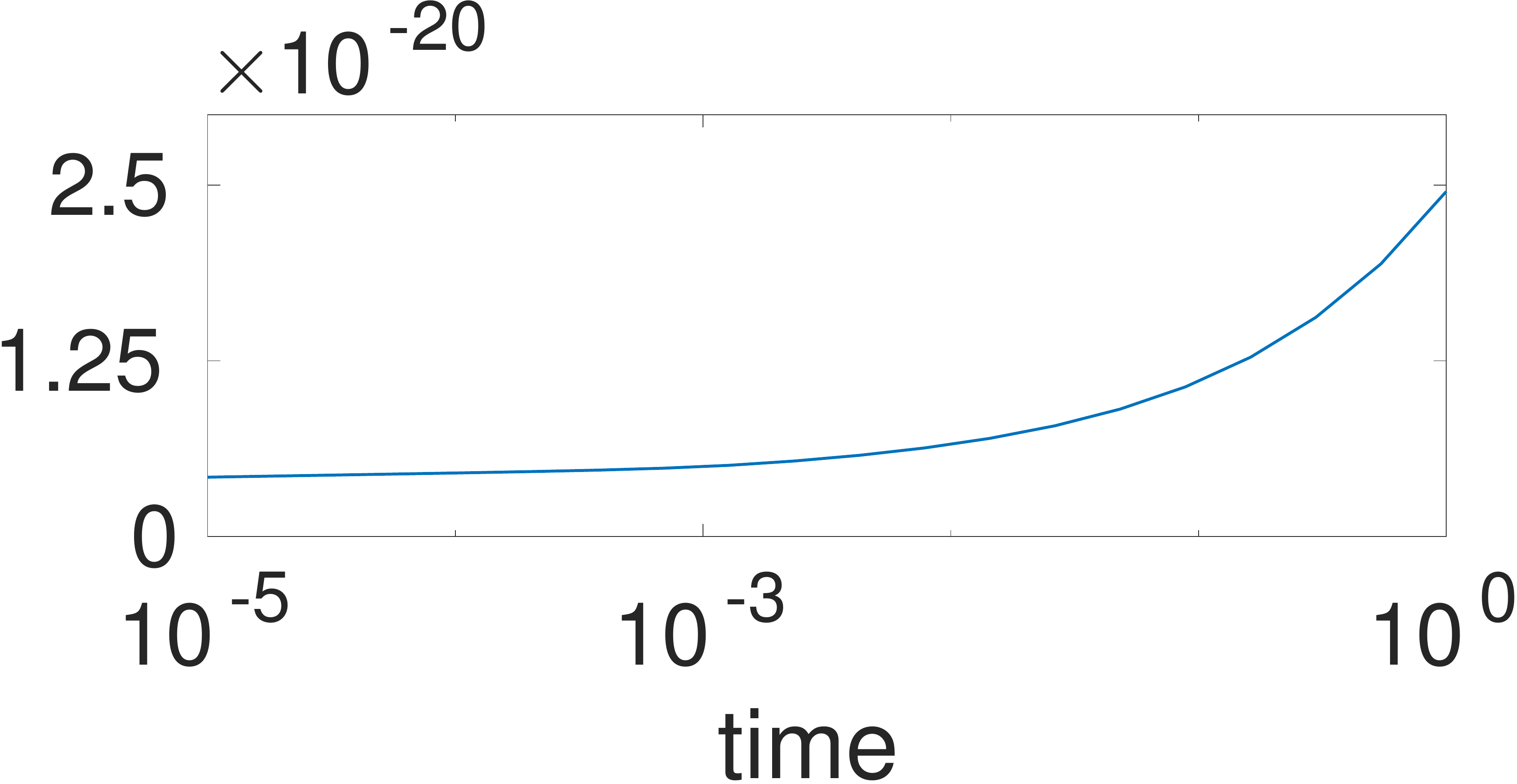}
	\end{minipage}
	\hspace{.5cm}
	\begin{minipage}[c]{.28\textwidth}
		\centering
		\includegraphics[width=\linewidth]{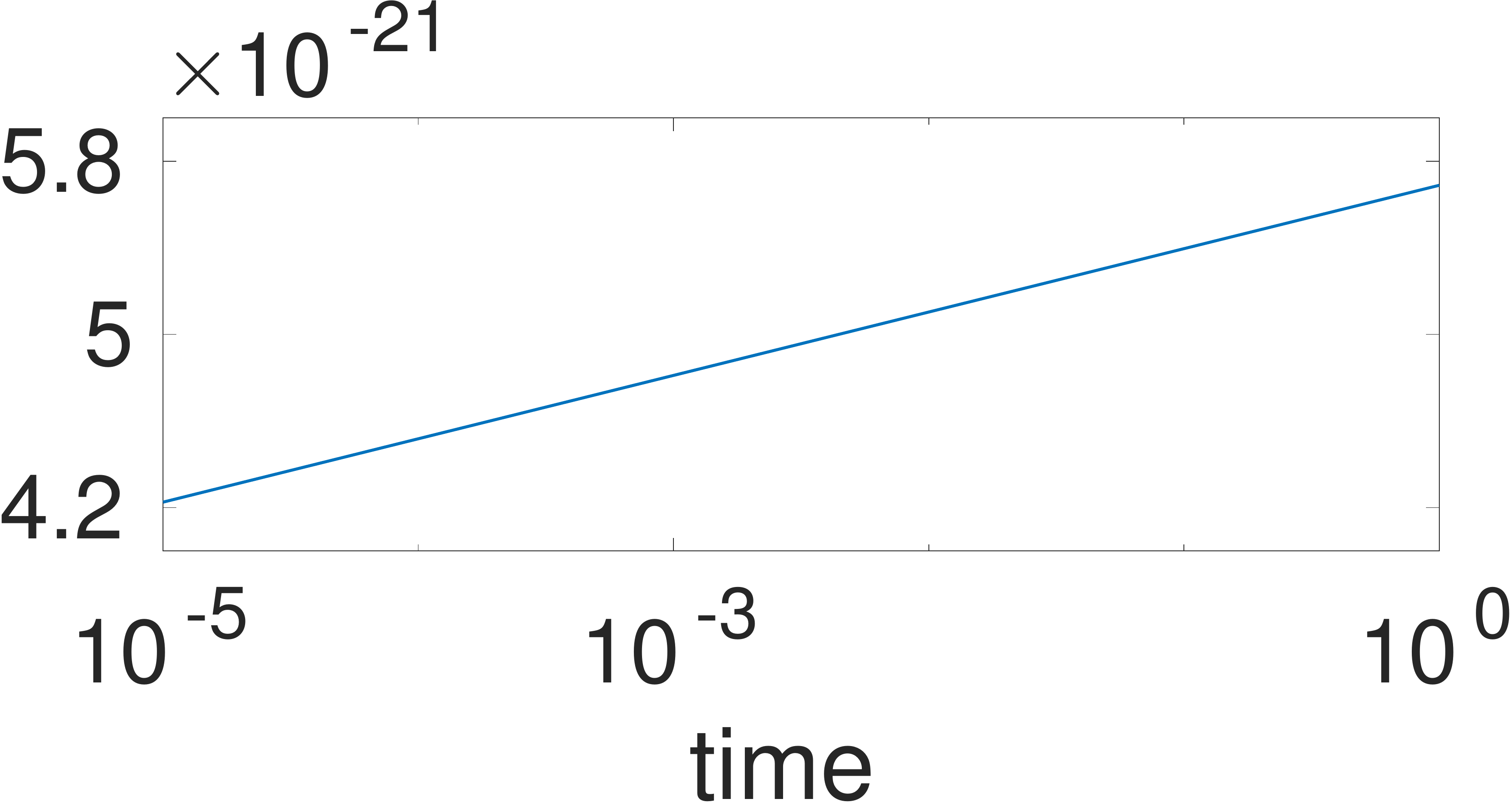}
	\end{minipage}
	\hspace{.5cm}
	\begin{minipage}[c]{.28\textwidth}
		\centering
		\includegraphics[width=\linewidth]{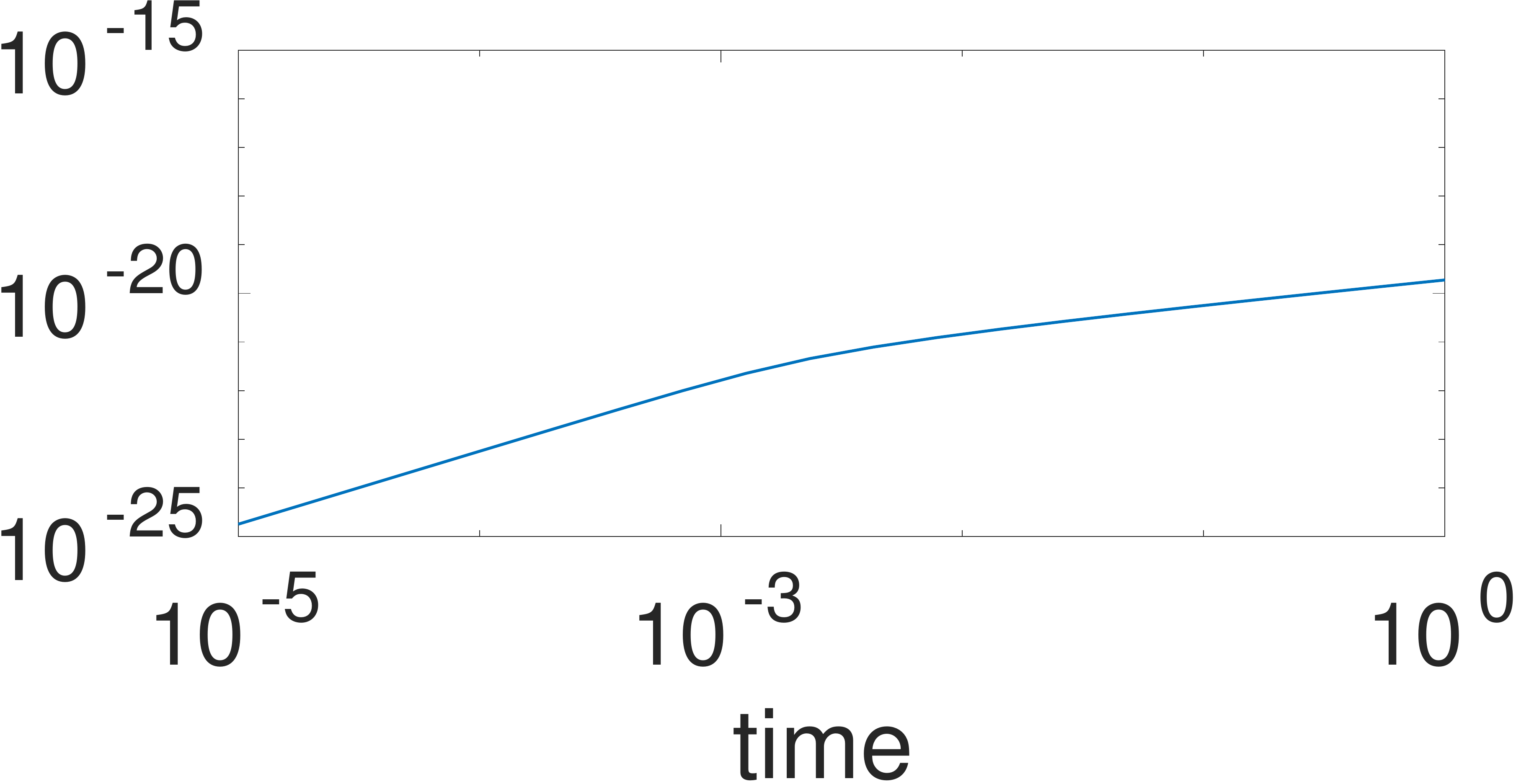}
	\end{minipage}
	\caption{Time-logarithmic plots showing the total heat on the boundary for each component of the heat. The values were obtained for each fixed time, by integrating over the boundary the computed heat. From left- to right-hand side: The total heat, the zeroth-order and its first-order, according to formula given by Theorem \ref{thm-Heat small volume heat}. Notice that the first-order term is plotted in a log-log scale.}
	\label{fig:1pHeat3}
\end{figure}

We can observe that the heat is not symmetric, this can be noticed from the total inner field for the first resonance mode in Figure \ref{fig:1pHelmResonance}. The reason behind this non symmetry is because we are illuminating with direction $d = (1,1)^t/\sqrt{2}$ over an ellipse. From Figure \ref{fig:1pHeat3} we can notice that  the first-order term converges, while the zeroth-order term increases logarithmically, as it is expected from the known solution of the heat equation for constant source in two dimensions that the heat increases logarithmically. 

\subsection{ Two particles simulation }

We consider two elliptical nanoparticles $D_1$, $D_2$, $D = D_1 \cup D_2$, with the same shape and orientation as the nanoparticle considered in the above example. The particle $D_1$ is centered at the origin and $D_2$ is centered at $(0,4.1\cdot 10^{-9})$,  resulting in a separation distance of $0.1nm$ between the two particles.

\subsubsection{ Two particles Helmholtz resonance }

Following the same analysis as the one for one particle, in Figure \ref{fig:2pResonance} we present the inner product between the eigenvectors of $\mathcal{K}_D^*$ with each component of the normal of $D$. We can observe that there are more available resonant modes. In particular we can see that when $\lambda_\epsilon$ approaches the 36th or 37th eigenvalues, we achieve strong resonant modes.

\begin{figure}[h!]
	\centering
	\begin{minipage}[c]{.6\textwidth}
		\centering
		\includegraphics[width=\linewidth]{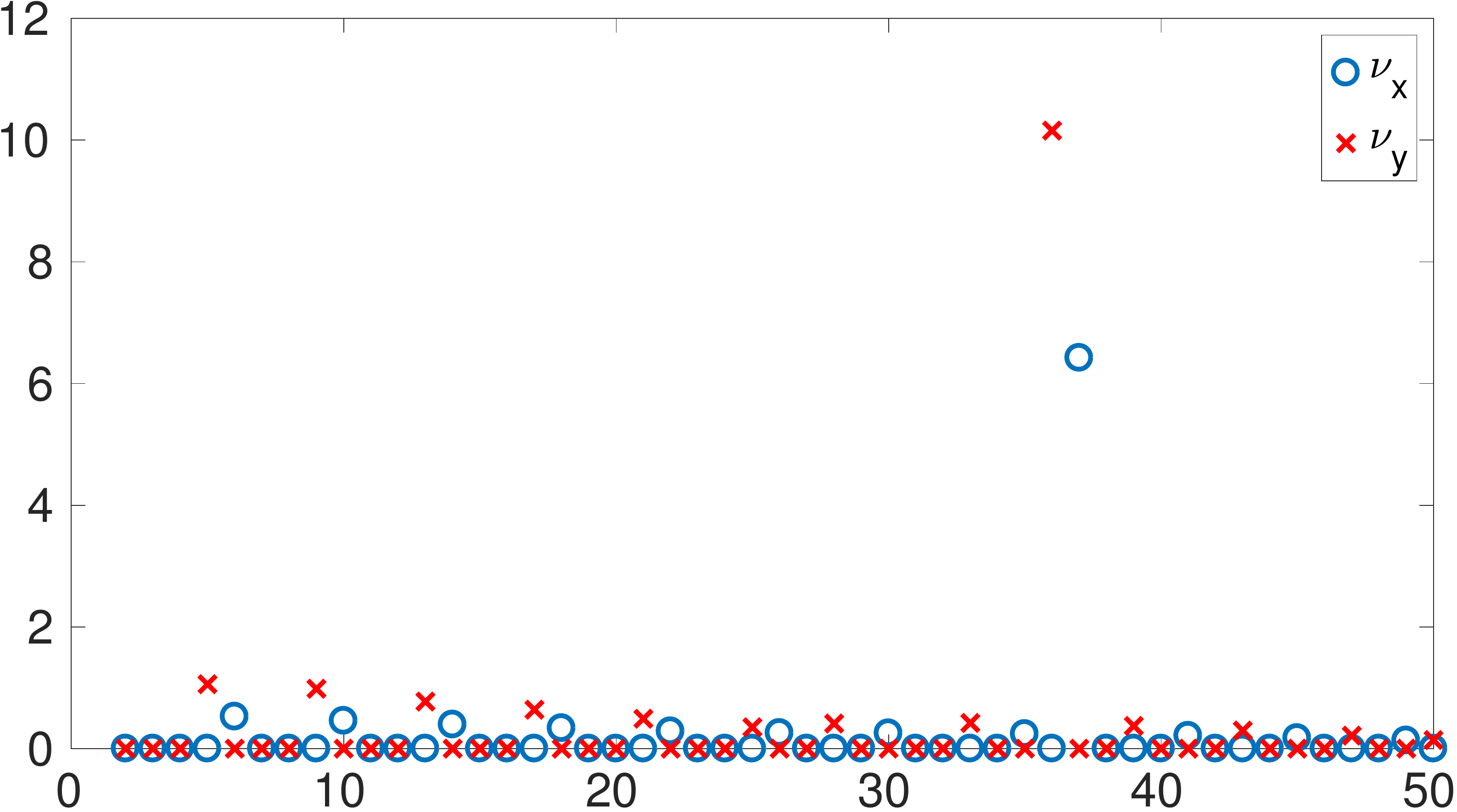}
	\end{minipage}
	\caption{inner product in $\mathcal{H}^*{\p D}$ between the eigenvalues of $\mathcal{K}_D^*$ and each component of the normal of $\partial D$, $\nu_x$ and $\nu_y$. }
	\label{fig:2pResonance}
\end{figure}

In Figure \ref{fig:1pHelmResonance} we present the absolute value of the inner field for the resonant modes corresponding to the 6th, 37th and 38th eigenvalues of $\mathcal{K}_D^*$. Similarly to the case with one particle, the dominant term in the electromagnetic field for each case is the first-order term. In Figure \ref{fig:2pHelmResonance2} we decompose the first-order term in its $x$-component and $y$-component.

\begin{figure}[h!]
	\centering
	\begin{minipage}[c]{.28\textwidth}
		\centering
		\scriptsize Resonant mode associated to the 6th eigenvalue of $\mathcal{K}_D^*$ \ \ \
	\end{minipage}
	\hspace{.5cm}
	\begin{minipage}[c]{.28\textwidth}
		\centering
		\scriptsize Resonant mode associated to the 37th eigenvalue of $\mathcal{K}_D^*$ \ \ \
	\end{minipage}
	\hspace{.5cm}
	\begin{minipage}[c]{.28\textwidth}
		\centering
		\scriptsize Resonant mode associated to the 38th eigenvalue of $\mathcal{K}_D^*$ \ \ \
	\end{minipage}
	\\
	\begin{minipage}[c]{.28\textwidth}
		\centering
		\includegraphics[width=\linewidth]{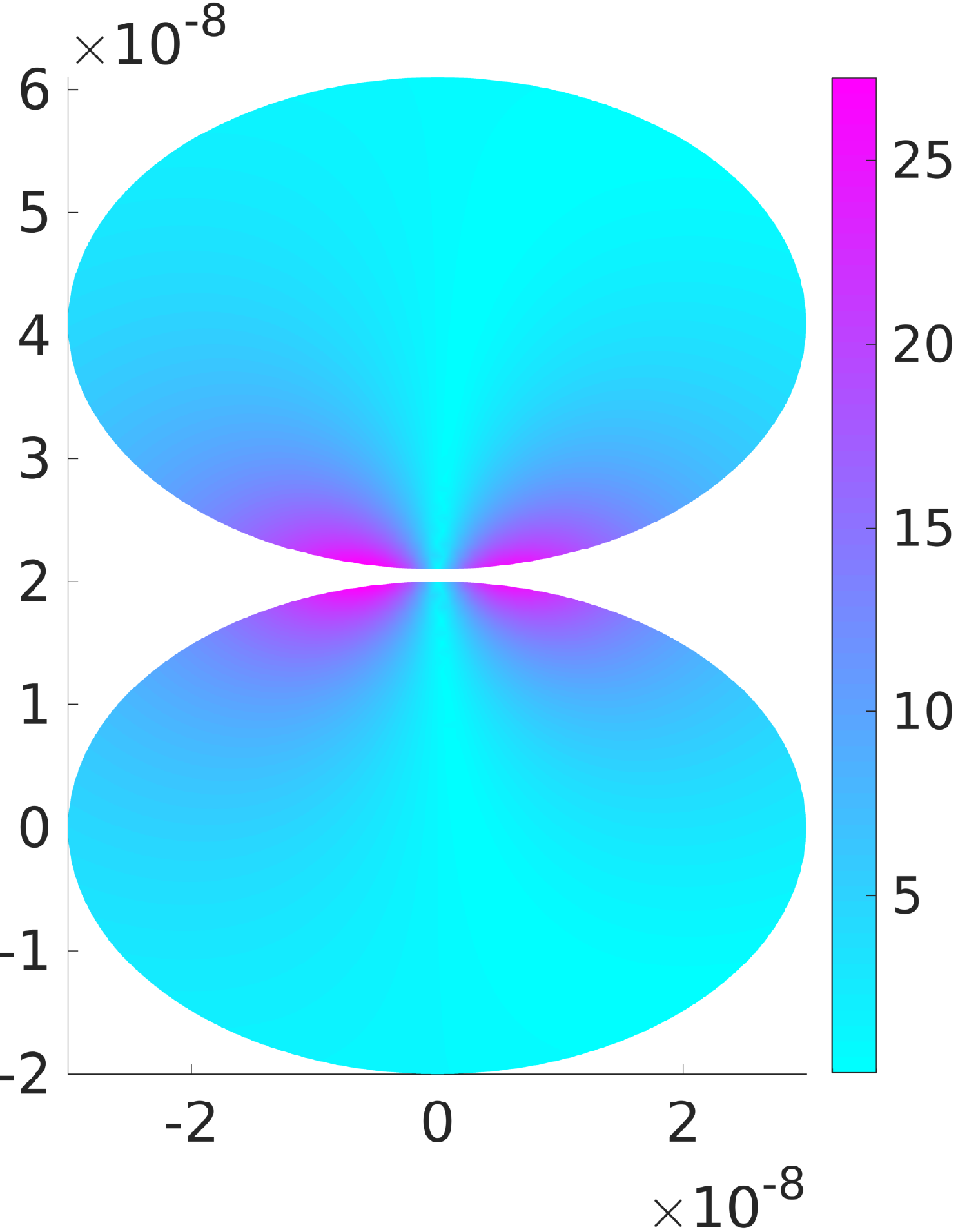}
	\end{minipage}
	\hspace{.5cm}
	\begin{minipage}[c]{.28\textwidth}
		\centering
		\includegraphics[width=\linewidth]{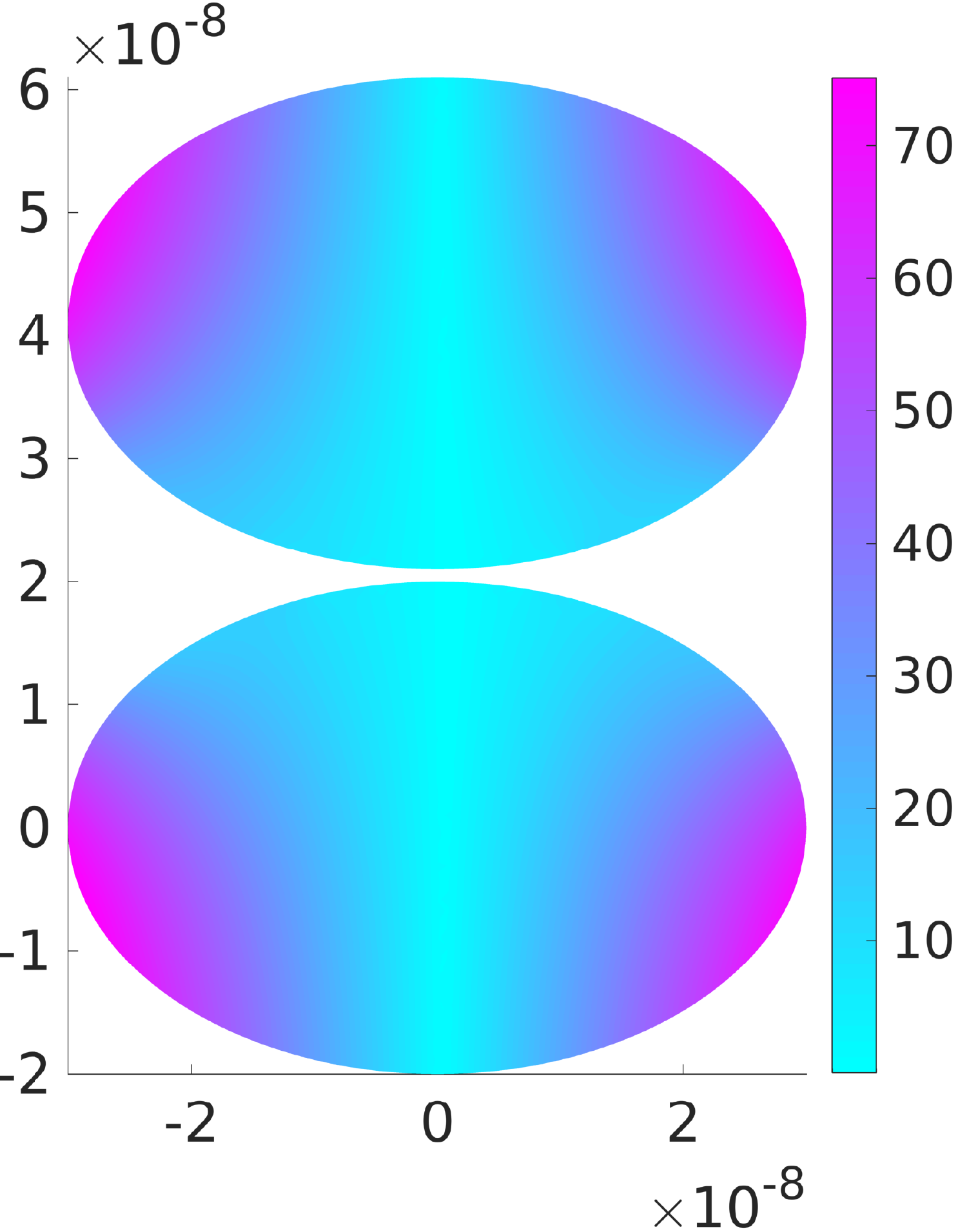}
	\end{minipage}
	\hspace{.5cm}
	\begin{minipage}[c]{.28\textwidth}
		\centering
		\includegraphics[width=\linewidth]{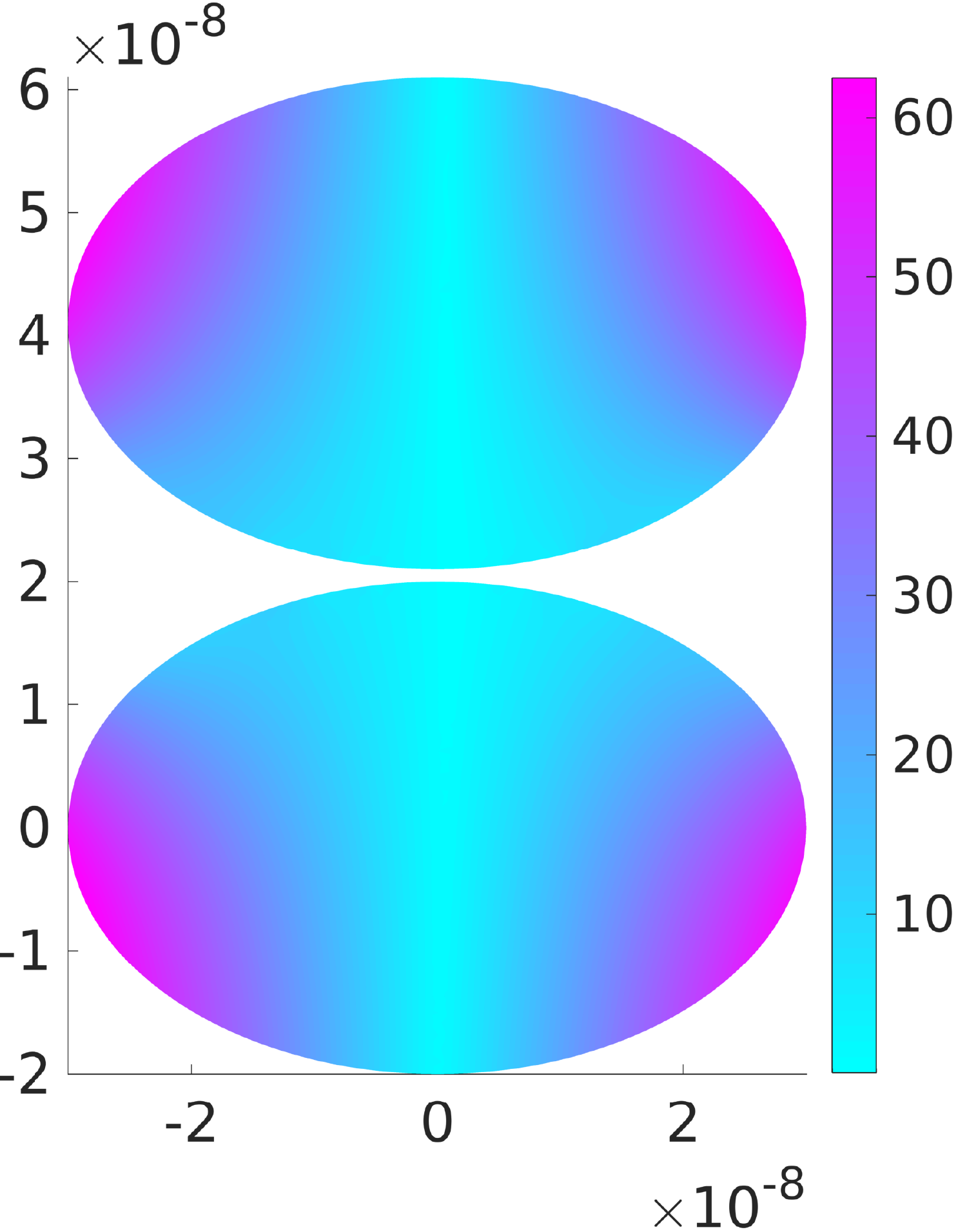}
	\end{minipage}
	\caption{Absolute value of the electromagnetic field inside the nanoparticle at the resonant modes associated to the 6th, 37th and 38th resonant modes, obtained when $\lambda_\epsilon$ approaches the respective eigenvalues of $\mathcal{K}_D^*$.}
	\label{fig:2pHelmResonance}
\end{figure}

\begin{figure}[h!]
	\centering
	\begin{minipage}[c]{.28\textwidth}
		\centering
		\scriptsize The $x$-component  \ \ \ \ \ \ \ \ \
	\end{minipage}
	\hspace{.5cm}
	\begin{minipage}[c]{.28\textwidth}
		\centering
		\scriptsize The $y$-component \ \ \ \ \ \ \ \ \
	\end{minipage}
	\hspace{.5cm}
	\\
	\begin{minipage}[c]{.28\textwidth}
		\centering
		\includegraphics[width=\linewidth]{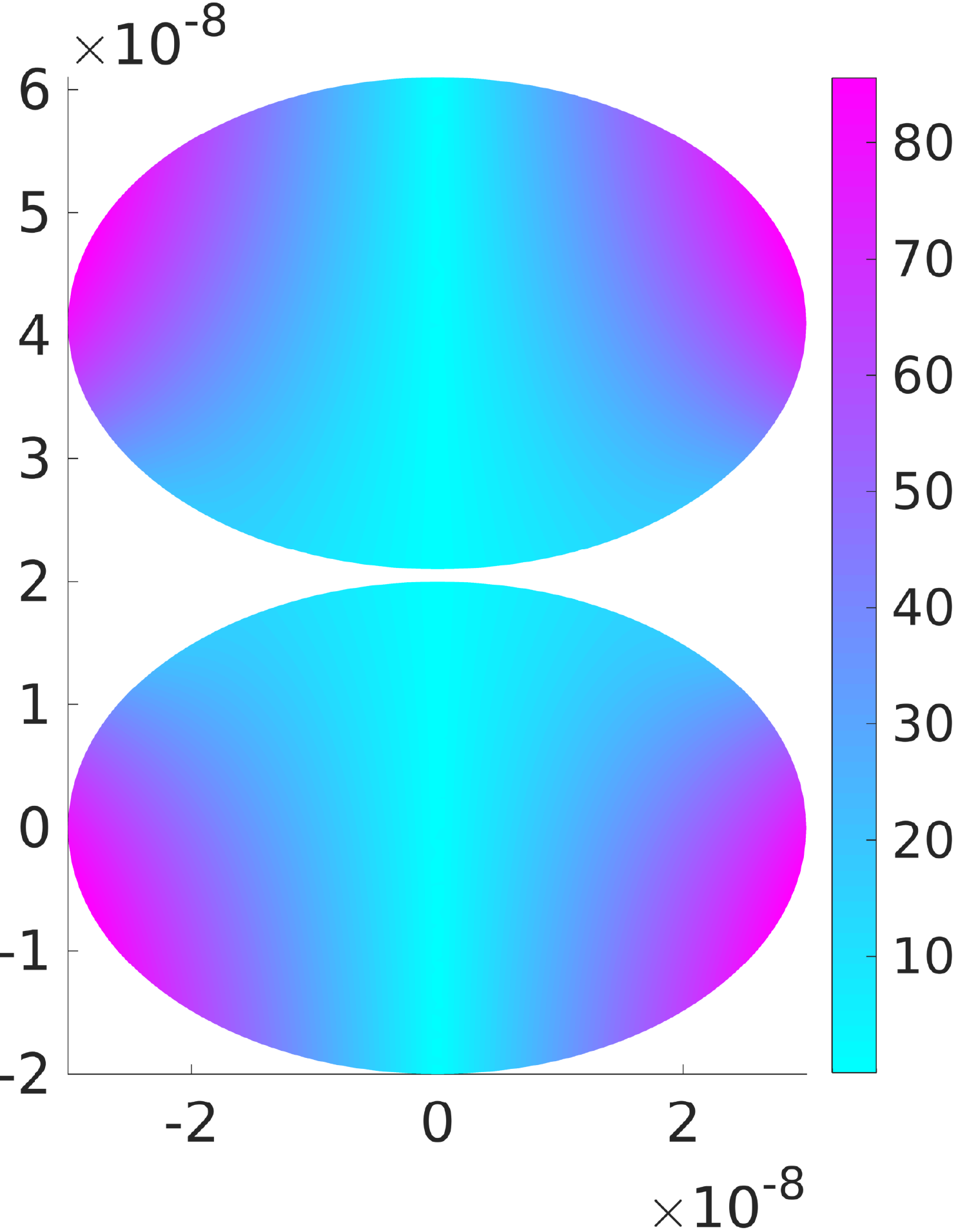}
	\end{minipage}
	\hspace{.5cm}
	\begin{minipage}[c]{.28\textwidth}
		\centering
		\includegraphics[width=\linewidth]{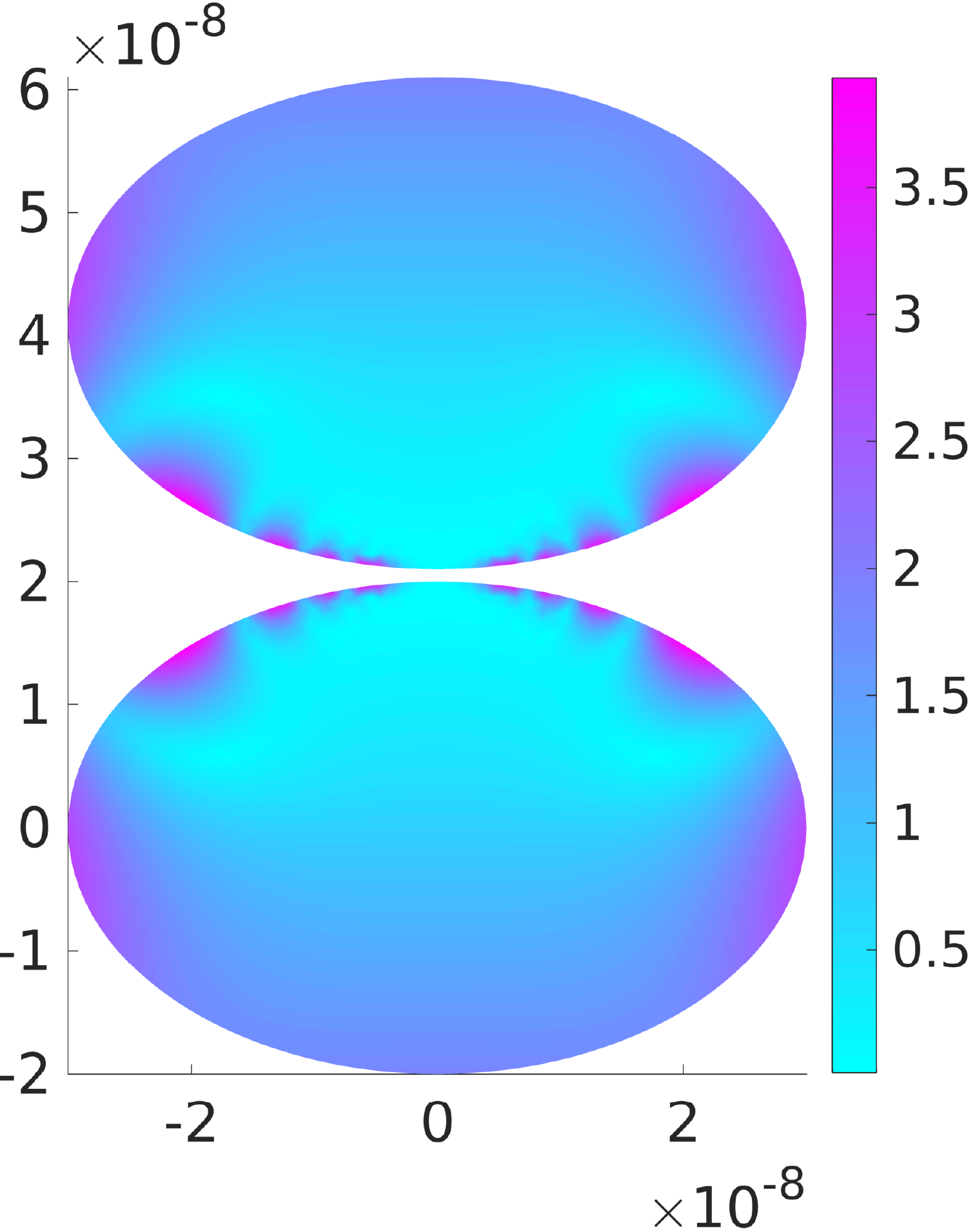}
	\end{minipage}
	\hspace{.5cm}
	\caption{Absolute value of the vectorial components of the first-order term for the 38th resonant mode.}
	\label{fig:2pHelmResonance2}
\end{figure}

As suggested by Figure \ref{fig:2pResonance}, for the resonant mode associated to the 38th eigenvector of $\mathcal{K}_D^*$, the stronger component is the one on the y direction, meaning that if we wish to maximize the electromagnetic field, and therefore the generated heat, it is suggested to consider the illumination vector $d = (0,1)^t$.

\subsubsection{ Two particles surface heat generation }

Similarly to the analysis carried out for one particle, we now compute the generated heat for these two particles while undergoing resonance for the resonant mode associated to the 38th eigenvalue of $\mathcal{K}_D^*$. In Figure \ref{fig:2pHeat} we plot generated heat in the boundary of the two nanoparticles. In Figure \ref{fig:2pHeat2} we decompose the generated heat in its zeroth and first-order component, explicited for each of the two nanoparticles. 

Similarly to the single nanoparticle case, there is no symmetry on the heat values on the boundary, which is due to the illumination. We have not provided the plots of the heat integrated along the boundary, as the conclusions are the same as the ones in the single nanoparticle case: The total heat on the boundary increases logarithmically, initially on time the dominant term is the fist-order one, but as time increases the zeroth-order term becomes the predominant one.

\begin{figure}[h!]
	\centering
	\begin{minipage}[c]{.45\textwidth}
					\begin{minipage}[c]{\textwidth}
						\centering \scriptsize
						3D plot of heat on $\partial (D_1 \cup D_2)$ at time $T=1$.
					\end{minipage}
				\begin{minipage}[c]{\textwidth}
				\centering
				\includegraphics[width=\linewidth]{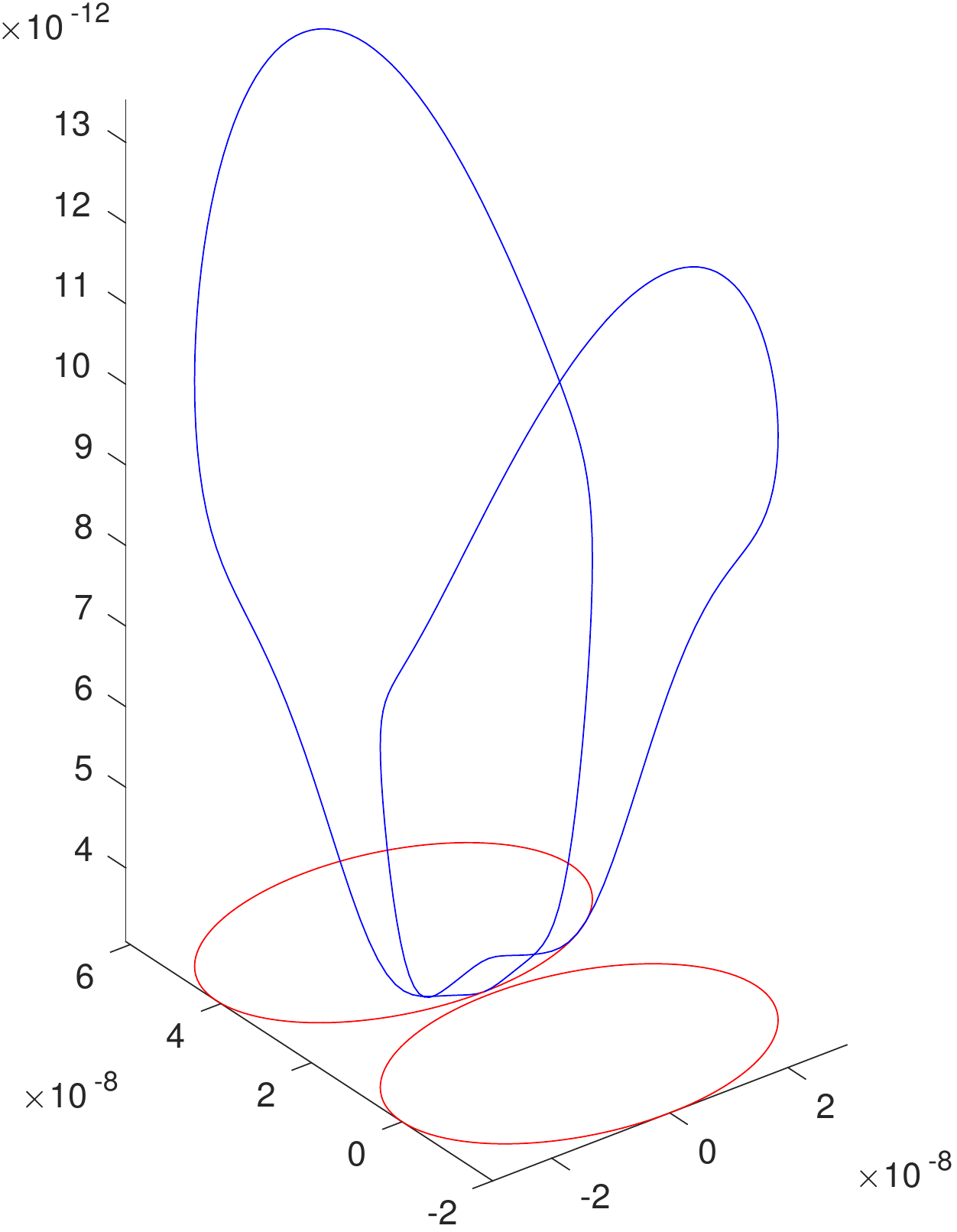}
			\end{minipage}
	\end{minipage}
	\hspace{.5cm}
	\begin{minipage}[c]{.45\textwidth}
			\begin{minipage}[c]{\textwidth}
				\centering \scriptsize
				Heat on $\partial D_1$ at time $T=1$
			\end{minipage}
			\begin{minipage}[c]{\textwidth}
						\centering
						\includegraphics[width=\linewidth]{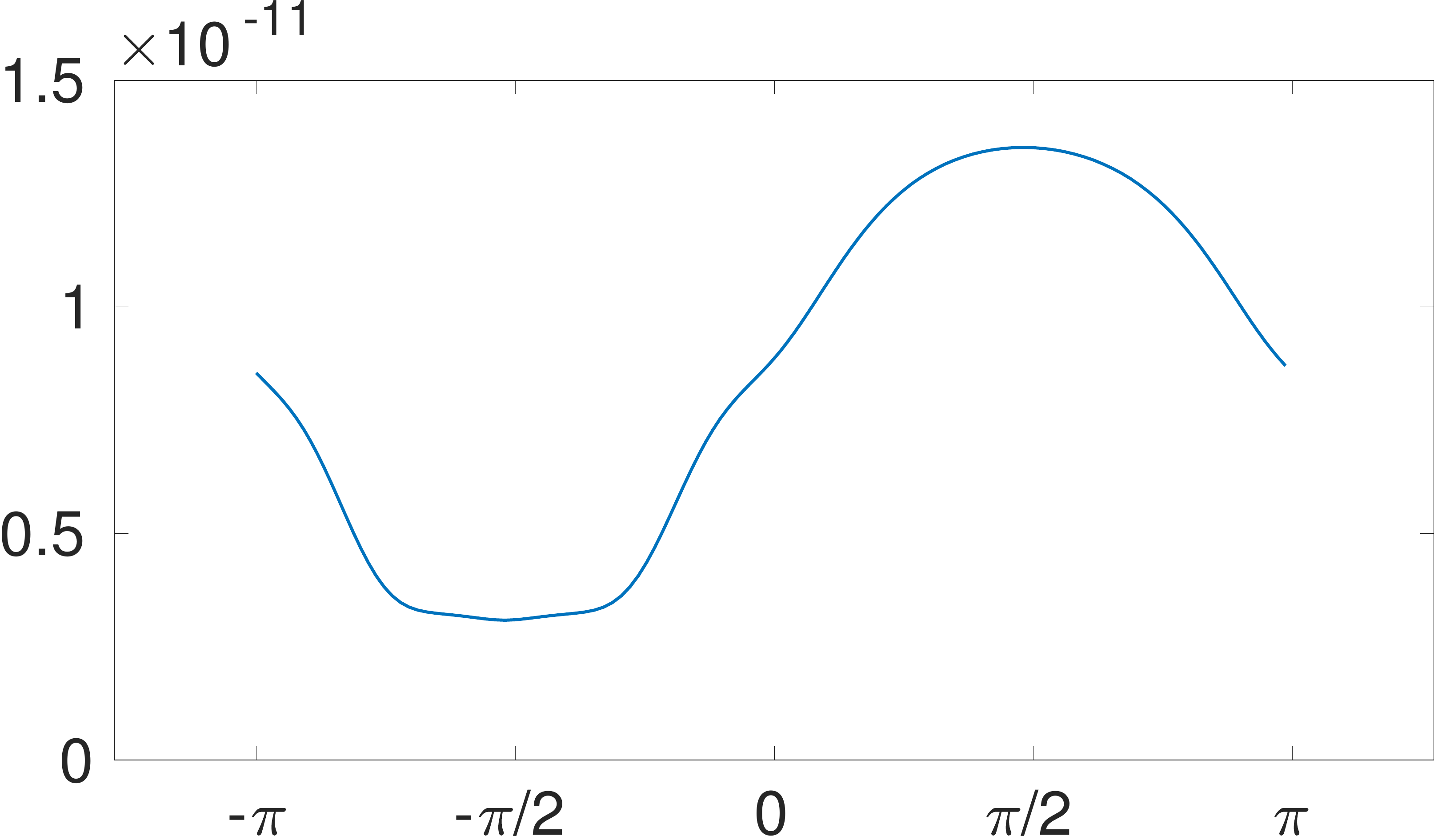}
			\end{minipage}
			\vspace{0.3cm}
			\\
			\begin{minipage}[c]{\textwidth}
				\centering \scriptsize
				Heat on $\partial D_2$ at time $T=1$
			\end{minipage}
			\\
			\begin{minipage}[c]{\textwidth}
				\centering
				\includegraphics[width=\linewidth]{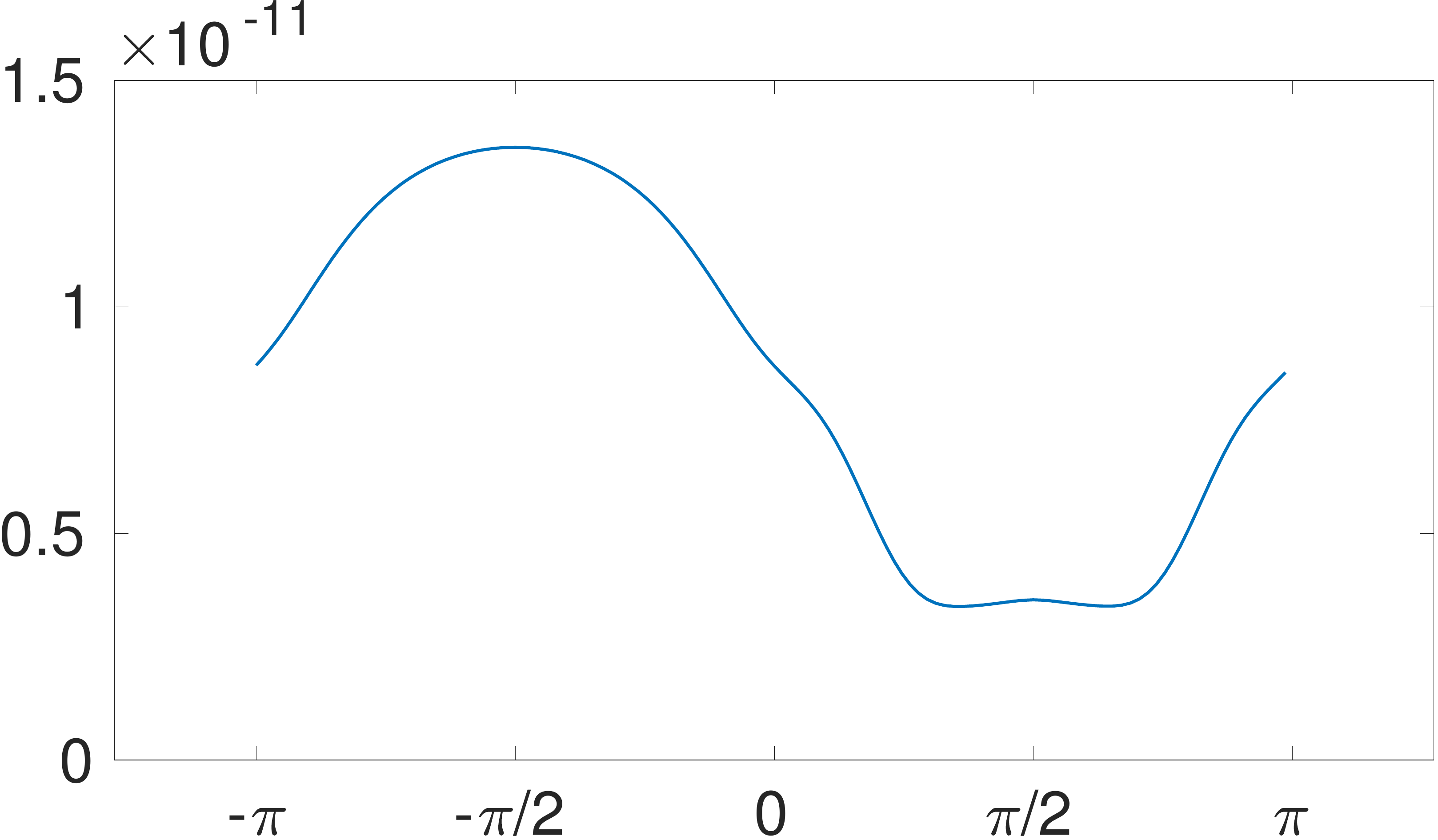}
			\end{minipage}
	\end{minipage}
	\caption{Generated heat on the boundary of the nanoparticles for time equal to 1. On the left we can see a three dimensional view of the heat, the red shapes are referential to show the location of the nanoparticles. On the right-hand side we can see the two dimensional heat plots corresponding to each nanoparticle. To obtain these plots we parameterized the boundary of each nanoparticle with $p(\theta) = (a\cos(\theta),b\sin(\theta)) + z, \ \theta \in [-\pi, \pi]$, where $z \in \mathbb{R}^2$ corresponds to the center of each nanoparticle. On the top we can see nanoparticle $D_2$ and on the bottom nanoparticle $D_1$.}
	\label{fig:2pHeat}
\end{figure}

\begin{figure}[h!]
	\centering
	\begin{minipage}[c]{.45\textwidth}
		\begin{minipage}[c]{\textwidth}
			\centering \scriptsize
			Heat zeroth component on $\partial D_2$
		\end{minipage}
		\begin{minipage}[c]{\textwidth}
			\centering 
			\includegraphics[width=\linewidth]{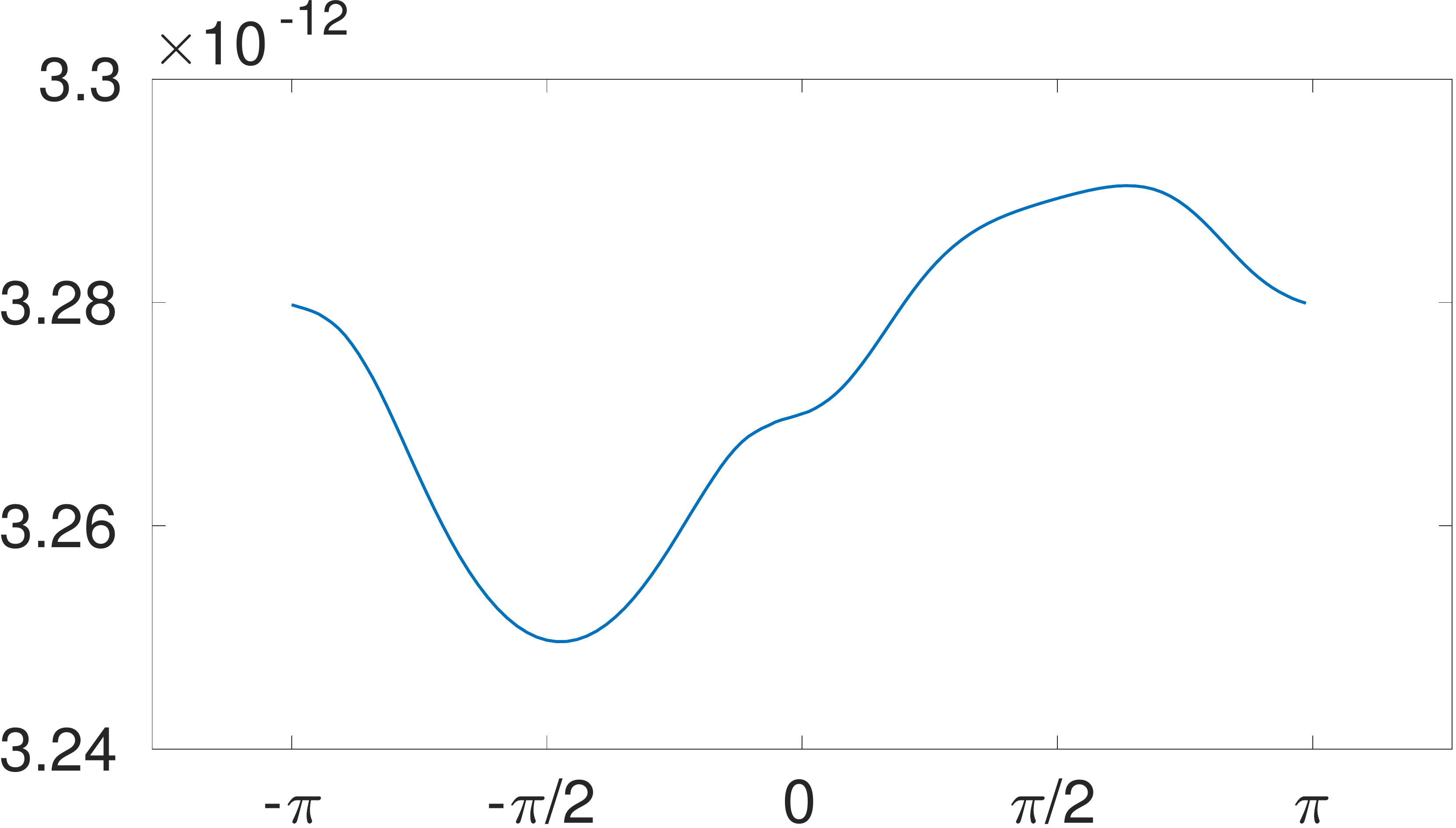}
		\end{minipage}
		\vspace{0.3cm}
		\\
		\begin{minipage}[c]{\textwidth}
			\centering \scriptsize
			Heat zeroth component on $\partial D_1$ at time $T=1$
		\end{minipage}
		\\
		\begin{minipage}[c]{\textwidth}
			\centering 
			\includegraphics[width=\linewidth]{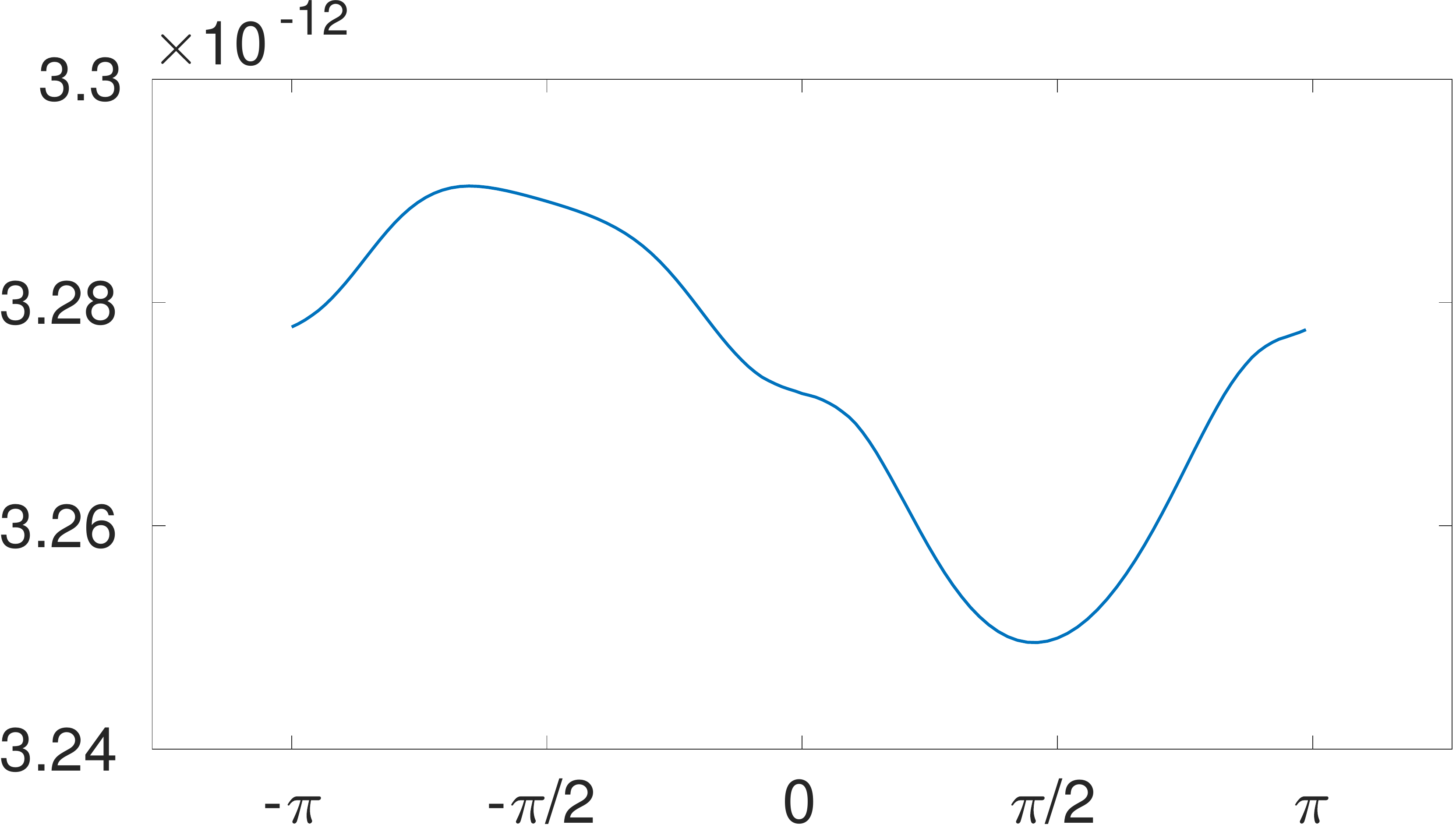}
		\end{minipage}
	\end{minipage}
	\hspace{.5cm}
	\begin{minipage}[c]{.45\textwidth}
		\begin{minipage}[c]{\textwidth}
			\centering \scriptsize
			Heat first component on $\partial D_2$ at time $T=1$
		\end{minipage} 
		\begin{minipage}[c]{\textwidth}
			\centering
			\includegraphics[width=\linewidth]{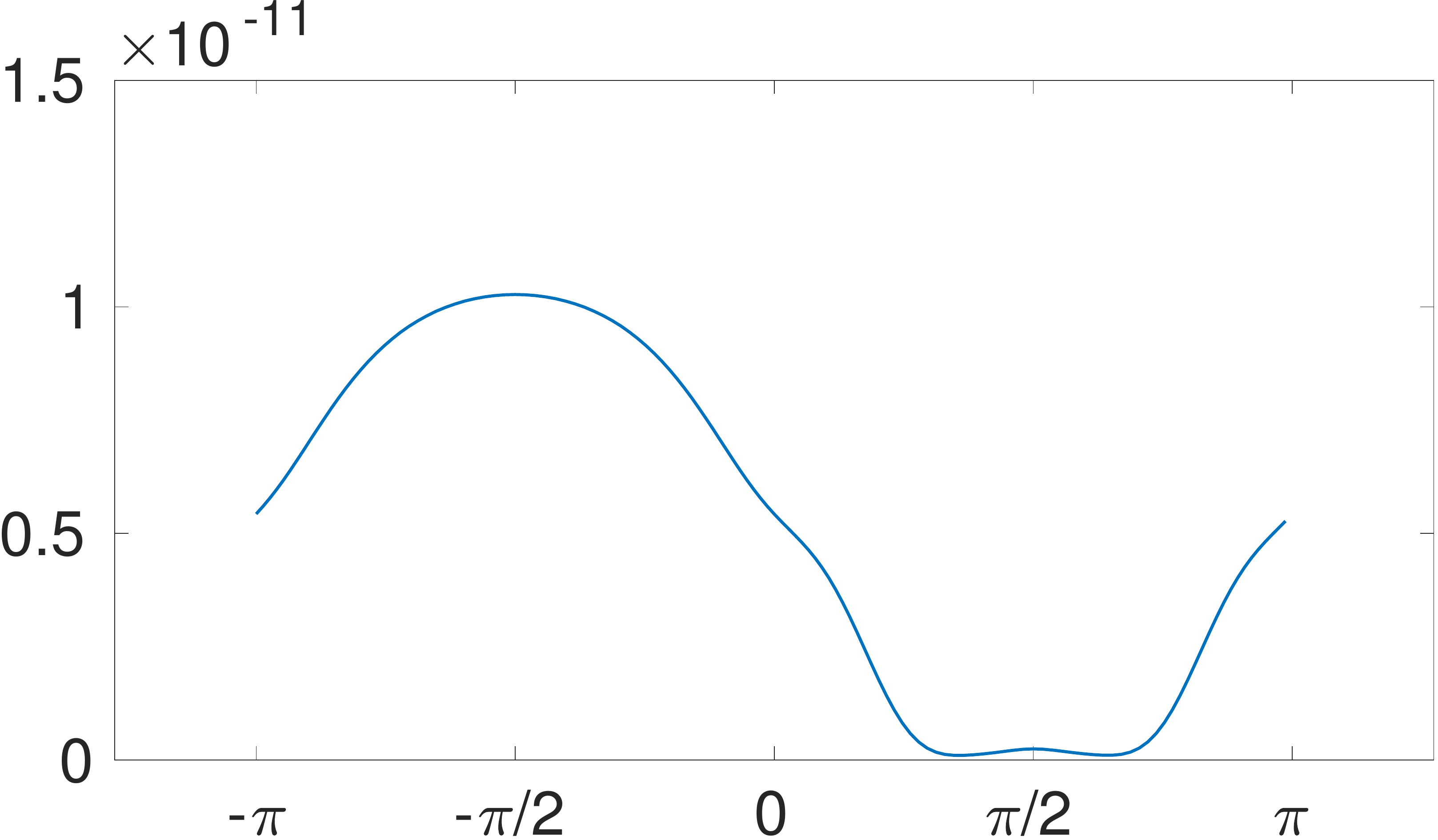}
		\end{minipage}
		\vspace{0.3cm}
		\\
		\begin{minipage}[c]{\textwidth}
			\centering \scriptsize
			Heat first component on $\partial D_1$ at time $T=1$
		\end{minipage}
		\\
		\begin{minipage}[c]{\textwidth}
			\centering
			\includegraphics[width=\linewidth]{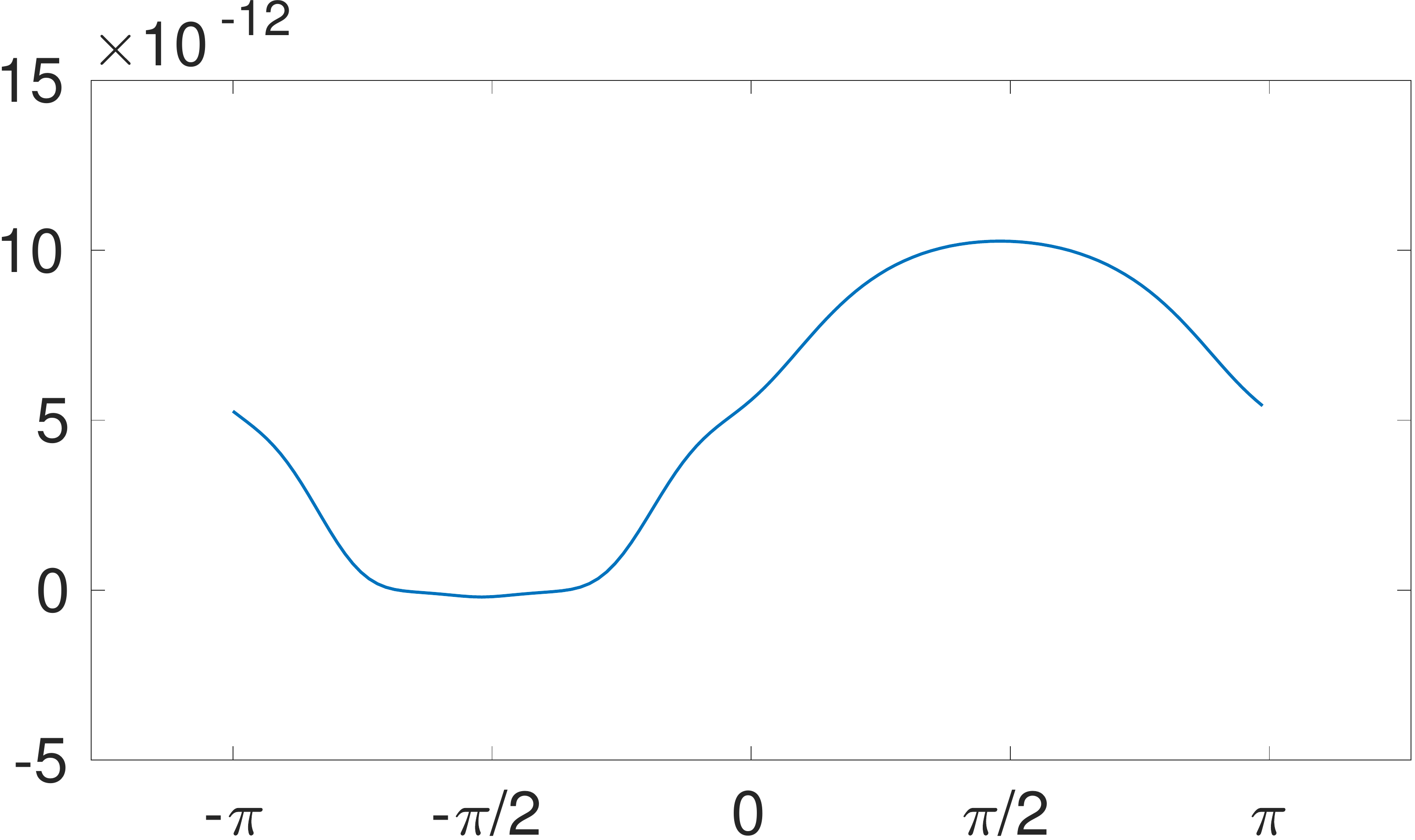}
		\end{minipage}
	\end{minipage}
	\caption{Two-dimensional plots of the zeroth and first component of the heat at time 1, for each nanoparticle. On the left column we have the zeroth component of the heat, on the right-hand side column we have the first component of the heat. On top we show the values for nanoparticle $D_2$, on the bottom we show the values for nanoparticle $D_1$.}
	\label{fig:2pHeat2}
\end{figure}

\section{Concluding remarks}

In this paper we have derived asymptotic formula for the temperature elevation due to plasmonic nanoparticles. We have considered thermal coupling within close-to-touching nanoparticles, where the temperature field deviates significantly from the one generated by a single nanoparticle. 
Our results can be used for the thermal detection and localization of the nanoparticles \cite{boccara}.
They can also be used for monitoring temperature elevation due to plasmonic nanoparticles based on the photoacoustic signal recently analyzed in \cite{triki2}. Thermoacoustic signals generated by nanoparticle heating can be computed numerically based on the successive resolution of the thermal diffusion problem considered in this paper and a thermoelastic problem, taking into account the size and shape of the nanoparticle, thermoelastic and elastic properties of both the particle and its environment, and the temperature-dependence of the thermal expansion coefficient of the environment. For sufficiently high illumination fluences, this temperature dependence yields a nonlinear relationship between the photoacoustic amplitude and the fluence \cite{prost}. The investigation of this nonlinear model will be the subject of a forthcoming publication.

\appendix
\section{Asymptotic analysis of the single-layer potential in two dimensions} \label{append1 heat}
In this section we make an analysis of the single-layer potential $\mathcal{S}_D^{ k}$ for small values of $ k $, i.e $| k|\ll 1$. We use this, in section \ref{sec-field expansions heat} , to make and expansion on $\delta$ of the operator $\mathcal{A}_B(\delta)$ and its inverse. 

The results in this section were first established in \cite{kang1} for a connected domain $D$. Here we generalize them to non connected domains.

\subsection{Layer potentials for the Laplacian in two dimensions}
Recall the definition of the single-layer potential and Neumann-Poincar\'e operators for the Laplacian:
\beas
\mathcal{S}_{D} [\varphi](x) &=&  \int_{\p D} \df{1}{2\pi}\log|x-y| \varphi(y) d\sigma(y),  \quad x \in  \p {D},\\
\mathcal{K}_{D}^* [\varphi](x) &=&  \int_{\p D}\df{1}{2\pi} \df{(x-y,\nu(x))}{|x-y|} \varphi(y) d\sigma(y),  \quad x \in  \p {D}.
\eeas 

In $\mathbb{R}^2$ the single-layer potential $\mathcal{S}_D:H^{-1/2}(\partial D)\rightarrow H^{1/2}(\partial D)$ is not, in general, invertible. Hence, $-( u, \mathcal{S}_D[v])_{-\f{1}{2},\f{1}{2}}$ does not define an inner product and the symmetrization technique described in \cite[subsection 2.1.4]{book3} is no longer valid.

Here and throughout, $(\cdot, \cdot)_{-\f{1}{2},\f{1}{2}}$ denotes the duality pairing between $H^{-1/2}(\p D)$ and  $H^{1/2}(\p D)$. 

To overcome this difficulty, we will introduce a substitute of $\mathcal{S}_D$, in the same way as in \cite{kang1}.

We first need the following lemma.

\begin{lem} \label{lem-dim Ker S heat}
Let $\mathcal{C} = \{\varphi \in  H^{-1/2}(\partial D) ; \; \exists \; \alpha \in \mathbb{C}, \; \mathcal{S}_D[\varphi] = \alpha\}$. We have $\textnormal{dim}(\mathcal{C})=1$.
\end{lem}
\begin{proof}
It is known that 
\beas
\mathcal{A}_D: H^{-1/2}(\partial D)\times \mathbb{C} &\rightarrow& H^{1/2}(\partial D)\times \mathbb{C} \\
(\varphi, a) &\rightarrow& \Big(\mathcal{S}_D[\varphi] + a, \int_{\p D}\varphi d\sigma\Big),
\eeas 
is invertible \cite[Theorem 2.26]{book2}. 

We can see that $\mathcal{C}=\Pi_1\mathcal{A}_D^{-1}(0,\mathbb{C})$, where $\Pi_1[(\varphi,a)]=\varphi$.
The invertibility of $\mathcal{A}_D$ implies that $\mbox{Ker}(\Pi_1\mathcal{A}_D^{-1}(0,\cdot))=\{0\}$. Thus, by the range theorem we have
\beas
1 = \mbox{dim}(\mbox{Im}(\Pi_1\mathcal{A}_D^{-1}(0,\cdot))) + \mbox{dim}(\mbox{Ker}(\Pi_1\mathcal{A}_D^{-1}(0,\cdot))) =  \mbox{dim}(\mbox{Im}(\Pi_1\mathcal{A}_D^{-1}(0,\cdot)))=\mbox{dim}(\mathcal{C}).
\eeas
\end{proof}
\begin{definition} \label{def-phi_0 heat}
We call $\varphi_0$ the unique element of $\mathcal{C}$ such that $\int_{\p D}\varphi_0 d\sigma =1$.
\end{definition} Note that for every $\varphi \in H^{-1/2}(\partial D)$ we have the decomposition 
\beas
\varphi = \varphi -  \Big(\int_{\p D}\varphi d\sigma\Big) \varphi_0 + \Big(\int_{\p D}\varphi d\sigma\Big) \varphi_0 := \psi + \Big(\int_{\p D}\varphi d\sigma\Big) \varphi_0,
\eeas 
where we can see that $(\psi,1)_{-\f{1}{2},\f{1}{2}}=0$. This kind of decomposition, $\varphi = \psi + \alpha\varphi_0$, with $(\psi,1)_{-\f{1}{2},\f{1}{2}}=0$ is unique.

	Note that we can decompose $H^{-1/2}$ as a direct sum of elements with zero-mean and multiples of $\varphi_0$, $H^{-1/2}(\partial D) = H^{-1/2}_0(\p D) \oplus \{\mu \varphi_0, \mu \in \mathbb{C} \}$. This allows us to define the following operator.

\begin{definition} \label{def-S_tilde heat}
Let $\widetilde{\mathcal{S}}_D$ be the linear operator that satisfies
\beas
\widetilde{\mathcal{S}}_D: H^{-1/2}(\partial D) &\rightarrow& H^{1/2}(\partial D)\nonumber \\
\varphi &\rightarrow& \left\{ \begin{array}{cc}
\mathcal{S}_D[\varphi] & \quad \mbox{if } 
(\varphi,1)_{-\f{1}{2},\f{1}{2}}=0, \\
-1 & \quad \mbox{if } \varphi_0 = \varphi.
\end{array}
\right.
\eeas
\end{definition}

\begin{rmk}
	When $\mathcal{S}_D$ is invertible, $\widetilde{\mathcal{S}}_D$ is similar enough to keep the invertibility. When $\mathcal{S}_D$ is not invertible, then $\mathcal{C} = \text{ker}(\mathcal{S}_D)$ and the operator $\widetilde{\mathcal{S}}_D$ becomes an invertible alternative to $\mathcal{S}_D$ that images the kernel $\mathcal{C}$ to the space $\{ \mu \chi(\p D), \mu \in \mathbb{C}\}$.
\end{rmk}
\begin{rmk}
 $\widetilde{\mathcal{S}}_D: H^{-1/2}(\partial D) \rightarrow H^{1}(D)$ follows the same definition. 
\end{rmk}

\begin{thm} \label{thm S_tilde invertible heat}
$\widetilde{\mathcal{S}}_D$ is invertible, self-adjoint and negative for $(\cdot,\cdot)_{-\f{1}{2},\f{1}{2}}$ and satisfies the following Calder\'on identity: $\widetilde{\mathcal{S}}_D  \mathcal{K}^*_D= \mathcal{K}_D\widetilde{\mathcal{S}}_D$. 
\end{thm}
\begin{proof}
The invertibility is a direct consequence of Lemma \ref{lem-dim Ker S heat}. 

Indeed, since $\mathcal{S}_D$ is Fredholm of zero index, so is $\widetilde{\mathcal{S}}_D$. Therefore, we only need the injectivity. Suppose that, $\exists \; \varphi \neq 0$ such that $\widetilde{\mathcal{S}}_D[\varphi] = 0$. This mean that, $\exists \; \alpha \neq 0 \in \mathbb{C}$ such that $\varphi = \alpha \varphi_0$. Therefore, $\widetilde{\mathcal{S}}_D[\varphi] = \alpha\widetilde{\mathcal{S}}_D[\varphi_0] = -\alpha = 0$, which is a contradiction. Hence $\varphi = 0$.

The self-adjointness comes directly form that of $\mathcal{S}_D$. Noticing that $\varphi_0$ is an eigenfunction of eigenvalue $1/2$ of $\mathcal{K}_D^*$ we get the Calder\'on identity from a similar one  satisfied by $\mathcal{S}_D$: $\mathcal{S}_D\mathcal{K}^*_D = \mathcal{K}_D\mathcal{S}_D$; see \cite[Lemma 2.12]{book3}.

It is known that $\int_{\p D}\psi \mathcal{S}_D[\psi]d\sigma <0$ if $(\psi,1)_{-\f{1}{2},\f{1}{2}} = 0$ and $\psi \neq 0$, see \cite[Lemma 2.10]{book3}. Therefore, writing $\varphi = \psi + \Big(\int_{\p D}\varphi d\sigma\Big) \varphi_0$, with $\psi =\varphi - \Big(\int_{\p D}\varphi d\sigma\Big) \varphi_0$,  and noticing that $\int_{\p D} \varphi_0 \widetilde{\mathcal{S}}_D[\psi] d\sigma = \int_{\p D} \widetilde{\mathcal{S}}_D [\varphi_0] \psi d\sigma = -\int_{\p D} \psi d\sigma = 0$,  we have
\beas
\int_{\p D}\varphi \widetilde{\mathcal{S}}_D[\varphi]d\sigma &=& \int_{\p D}\psi \widetilde{\mathcal{S}}_D[\psi]d\sigma +  \Big(\int_{\p D}\varphi d\sigma\Big)^2\widetilde{\mathcal{S}}_D[\varphi_0]\\
&=& \int_{\p D}\psi \mathcal{S}_D[\psi]d\sigma - \Big(\int_{\p D}\varphi d\sigma\Big)^2 <0,
\eeas
if $\varphi \neq 0$.
%

\end{proof}

\begin{definition} \label{addeq5}
	We define the space $\mathcal{H}^*(\partial D)$ as the Hilbert space resulting from endowing $H^{-1/2}(\partial D)$ with the inner product
	\be
		(u,v)_{\mathcal{H}^*} := -( u, \widetilde{\mathcal{S}}_D[v])_{-\f{1}{2},\f{1}{2}}. 
	\ee
	Similarly, we let $\mathcal{H}$ to be the Hilbert space resulting from endowing $H^{1/2}$ with the inner product 
	\be
		(u, v)_{\mathcal{H}}= - (\widetilde{\mathcal{S}}_{D}^{-1}[u], v)_{-\f{1}{2},\f{1}{2}}.
	\ee
\end{definition}

If $D$ is $\mathcal{C}^{1,\alpha}$, we have the following result.
\begin{lem} \label{lem-K_star_properties2d heat}
Let $D$ be a $\mathcal{C}^{1,\alpha}$ bounded domain of $\mathbb{R}^2$ and let $\widetilde{\mathcal{S}}_D$ be the operator introduced in Definition \ref{def-S_tilde heat}. Then
\begin{enumerate}
\item[(i)]
The operator $\mathcal{K}_D^*$ is compact self-adjoint in the Hilbert space $\mathcal{H}^*(\p D)$ and $\mathcal{H}^*(\p D)$ is equivalent to $H^{-\f{1}{2}}(\p D)$; Similarly, the Hilbert space $\mathcal{H}(\p D)$ is equivalent to $H^{\f{1}{2}}(\p D)$.
\item[(ii)]
Let $(\lambda_j,\varphi_j) $, $j=0, 1, 2, \ldots,$ be the eigenvalue and normalized eigenfunction pair of $\mathcal{K}_D^*$ with $\lambda_0 = \f{1}{2}$.
Then, $\lambda_j \in (-\f{1}{2}, \f{1}{2}]$ and $\lambda_j \rightarrow 0$ as $j \rightarrow \infty$;
\item[(iii)]
The following representation formula holds: for any $\varphi \in H^{-1/2}(\p D)$,
$$
\mathcal{K}_D^* [\varphi]
= \sum_{j=0}^{\infty} \lambda_j (\varphi, \varphi_j)_{\mathcal{H}^*} \otimes \varphi_j.
$$
\end{enumerate}
\end{lem}

The following lemmas are needed in the proof of Theorem \ref{thm12d heat} and Theorem \ref{thm-Heat small volume heat}.
\begin{lem} \label{lem-change norm scaling heat}
Let $D = z + \delta B$ and $\eta$ be the function such that, for every $\varphi\in\mathcal{H}^*(\p D)$, $\eta(\varphi)(\tilde{x})= \varphi(z + \delta \tilde{x})$, for almost all $\tilde{x}\in \p B$. Then
\beas
\|\varphi\|_{\mathcal{H}^*(\p D)} = \delta\|\eta(\varphi)\|_{\mathcal{H}^*(\p B)}.
\eeas
Similarly, if for every $\varphi\in L^2(D)$, $\eta(\varphi)(\tilde{x})= \varphi(z + \delta \tilde{x})$, for almost all $\tilde{x}\in B$, then
\beas
\|\varphi\|_{L^2(D)} = \delta\|\eta(\varphi)\|_{L^2(B)}.
\eeas  
\end{lem}
\begin{proof}
We only prove the scaling in $\mathcal{H}^*(\p D)$. From the proof of Theorem \ref{thm S_tilde invertible heat}, we have
\beas
\|\varphi\|^2_{\mathcal{H}^*(\p D)} = -\int_{\p D}\psi \mathcal{S}_D[\psi]d\sigma + \Big(\int_{\p D}\varphi d\sigma\Big)^2,
\eeas
where $\psi =\varphi - \Big(\int_{\p D}\varphi d\sigma\Big) \varphi_0$. Note that $(\psi,1)_{-\f{1}{2},\f{1}{2}} = 0$ and so, $(\eta(\psi),\chi(\p B))_{-\f{1}{2},\f{1}{2}} = 0$ as well.

By a  rescaling argument we find that
\beas
\|\varphi\|^2_{\mathcal{H}^*(\p D)} &=& -\delta^2\int_{\p B} \int_{\p B}\f{1}{2\pi}\log |\delta(\tilde{x}-\tilde{y})|\eta(\psi)(\tilde{x}) \eta(\psi)(\tilde{y}) d\sigma(\tilde{x}) d\sigma(\tilde{y}) + \delta^2\Big(\int_{\p B}\eta(\varphi) d\sigma\Big)^2\\
&=& -\f{1}{2\pi}\delta^2\log(\delta)\Big(\int_{\p B}\eta(\psi) d\sigma\Big)^2  + 
\delta^2\left(-\int_{\p B}\eta(\psi) \mathcal{S}_D[\eta(\psi)]d\sigma + \Big(\int_{\p B}\eta(\varphi) d\sigma\Big)^2\right)\\
&=& \delta^2\|\eta(\varphi)\|^2_{\mathcal{H}^*(\p B)}.
\eeas
\end{proof}

\begin{lem} \label{lem-d/dn(u) 1/2-K^*u heat}
Let $g\in H^{1}(D)$ be such that $\Delta g = f$ with $f\in L^2(D)$. Then, in $\mathcal{H}^*(\p D)$,  
\beas
(\dfrac{1}{2}I - \mathcal{K}_{D}^*) \widetilde{\mathcal{S}}_{D}^{-1}[g] = -\df{\p g}{\p \nu} + \mathcal{T}_f.
\eeas
For some $\mathcal{T}_f \in \mathcal{H}^*(\p D)$ and $\|\mathcal{T}_f \|_{\mathcal{H}^*}\leq C \|f\|_{L^2(D)}$ for a constant $C$.

Moreover, if $g\in H^1_{loc}(\R^2)$, $\Delta g = 0$ in $\R^2\backslash \bar{D}$, $\lim_{|x|\rightarrow\infty}g(x) = 0$,  then 
\beas
\mathcal{T}_f = c_f\varphi_0 + \widetilde{\mathcal{S}}_{D}^{-1} [g],
\eeas
with
\beas
c_f = \int_D f(x) dx - \int_{\p D}\widetilde{\mathcal{S}}_{D}^{-1} [g](y) d\sigma(y),
\eeas
where $\varphi_0$ is given in Definition \ref{def-phi_0 heat}.
Here, by an abuse of notation, we still denote by $g$ the trace of $g$ on $\p D$.
\end{lem}
\begin{proof}
Let $\varphi \in \mathcal{H}^*(\p D)$. Then 
\beas
\Big((\f{1}{2}I - \mathcal{K}_{D}^*) \widetilde{\mathcal{S}}_{D}^{-1} [g], \varphi \Big)_{\mathcal{H}^*} &=& -\Big( \widetilde{\mathcal{S}}_{D}^{-1}[g],   \big(\f{1}{2}I - \mathcal{K}_{D}\big) \widetilde{\mathcal{S}}_{D}[\varphi]  \Big)_{-\f{1}{2}, \f{1}{2}}\\
&=&  -\Big( \widetilde{\mathcal{S}}_{D}^{-1}[g],  \widetilde{\mathcal{S}}_{D}\big(\f{1}{2}I - \mathcal{K}_{D}^*\big)[\varphi]  \Big)_{-\f{1}{2}, \f{1}{2}}\\
&=&  -\Big(g,  \big(\f{1}{2}I - \mathcal{K}_{D}^*\big)[\varphi]  \Big)_{-\f{1}{2}, \f{1}{2}}\\
&=&  -\Big(g, -\f{\p \widetilde{\mathcal{S}}_{D}[\varphi]}{\p \nu}\Big\vert_-   \Big)_{-\f{1}{2}, \f{1}{2}}\\
&=&  \int_{\p D}\df{\p g}{\p \nu}\widetilde{\mathcal{S}}_{D}[\varphi]d\sigma - \int_{D}\Big(f \widetilde{\mathcal{S}}_{D}[\varphi]-\Delta  \widetilde{\mathcal{S}}_{D}[\varphi]\big(g\big)\Big)dx\\
&=&  -\Big(\df{\p g}{\p \nu} ,\varphi \Big)_{\mathcal{H}^*} - \int_{D}f \widetilde{\mathcal{S}}_{D}[\varphi]dx.
\eeas
We have used the fact that $\widetilde{\mathcal{S}}_{D}$ is harmonic in $D$.

Consider the linear application $\mathcal{T}_f[\varphi] := -\int_{D}f \widetilde{\mathcal{S}}_{D}[\varphi]dx$. We have
\beas
|\mathcal{T}_f[\varphi]| \leq C \|f\|_{L^2(D)} \|\widetilde{\mathcal{S}}_{D}[\varphi]\|_{L^2(D)}\leq C_f\|\widetilde{\mathcal{S}}_{D}[\varphi]\|_{H^1(D)} \leq C_f \|\widetilde{\mathcal{S}}_{D}[\varphi]\|_{H^{\f{1}{2}}(\p D)} \leq C_f\|\varphi\|_{H^{-\f{1}{2}}(\p D)}.
\eeas
Here we have used Holder's inequality, a standard Sobolev embedding, the trace theorem and the fact that $\widetilde{\mathcal{S}}_{D}: H^{-\f{1}{2}}(\p D) \rightarrow H^{\f{1}{2}}(\p D)$ is continuous. By the Riez representation theorem, there exists $v \in \mathcal{H}^*(\p D)$ such that $\mathcal{T}_f[\varphi] = (v, \varphi)_{\mathcal{H}^*}\; ,\forall
\varphi \in \mathcal{H}^*(\p D)$. 

By abuse of notation we still denote $\mathcal{T}_f:= v$ to make explicit the dependency on $f$. It follows that
\beas
\|\mathcal{T}_f\|^2_{\mathcal{H}^*}  = -\int_{D}f \widetilde{\mathcal{S}}_{D}[\mathcal{T}_f]dx &\leq &
 C \|f\|_{L^2(D)} \|\widetilde{\mathcal{S}}_{D}[\mathcal{T}_f]\|_{L^2(D)} \\ 
 &\leq & C \|f\|_{L^2(D)}\|\widetilde{\mathcal{S}}_{D}[\mathcal{T}_f]\|_{H^1(D)} \\
 &\leq & C \|f\|_{L^2(D)}\|\widetilde{\mathcal{S}}_{D}[\mathcal{T}_f]\|_{H^{\f{1}{2}}(\p D)}\\
 &\leq & C \|f\|_{L^2(D)}\|\mathcal{T}_f\|_{\mathcal{H}^*}.
\eeas


We now show that in $\mathcal{H}^*_0(\p D)$, $\mathcal{T}_f = \widetilde{\mathcal{S}}_{D}^{-1}[g]$.

Indeed, let $\varphi\in \mathcal{H}^*_0(\p D)$, then
\beas
\Big(\widetilde{\mathcal{S}}_{D}^{-1} [g], \varphi \Big)_{\mathcal{H}^*} &=& -\Big( \widetilde{\mathcal{S}}_{D}^{-1}[g],   \widetilde{\mathcal{S}}_{D}[\varphi]  \Big)_{-\f{1}{2}, \f{1}{2}}\\
&=&  -\Big(g, \varphi  \Big)_{-\f{1}{2}, \f{1}{2}}\\
&=&  -\Big(g,  \f{\p \widetilde{\mathcal{S}}_{D}[\varphi]}{\p \nu}\Big\vert_+  -\f{\p \widetilde{\mathcal{S}}_{D}[\varphi]}{\p \nu}\Big\vert_-    \Big)_{-\f{1}{2}, \f{1}{2}}\\
&=&  \int_{\p D}\df{\p g}{\p \nu}\widetilde{\mathcal{S}}_{D}[\varphi]d\sigma  - \int_{\p D}\df{\p g}{\p \nu}\widetilde{\mathcal{S}}_{D}[\varphi]d\sigma + \int_{\p B_{\infty}}\df{\p g}{\p \nu}\widetilde{\mathcal{S}}_{D}[\varphi]d\sigma  - \int_{\p B_{\infty}}g\df{\p \widetilde{\mathcal{S}}_{D}[\varphi]}{\p \nu}d\sigma \\
 && - \int_{\R^2}\Big(f \widetilde{\mathcal{S}}_{D}[\varphi]-\Delta  \widetilde{\mathcal{S}}_{D}[\varphi]\big(g\big)\Big)dx\\
&=& -\int_{D}f \widetilde{\mathcal{S}}_{D}[\varphi]dx.
\eeas

Here we have used the assumption on $g$, the fact that $\widetilde{\mathcal{S}}_{D}[\varphi]$ is harmonic in $D$ and $\R^2\backslash \bar{D}$ and that for $\varphi\in\mathcal{H}^*_0(\p D)$ we have $\widetilde{\mathcal{S}}_{D}[\varphi](x) = O(\f{1}{|x|})$ and $\df{\p \widetilde{\mathcal{S}}_{D}[\varphi]}{\p \nu}(x) =  O(\f{1}{|x|})$ for $|x|\rightarrow \infty$.

Therefore, 
\beas
\mathcal{T}_f  = (\mathcal{T}_f-\widetilde{\mathcal{S}}_{D}^{-1} [g], \varphi_0 )_{\mathcal{H}^*}\varphi_0 + \widetilde{\mathcal{S}}_{D}^{-1} [g].
\eeas
Finally, recaling the definition of $\varphi_0$ given in Definition $\ref{def-phi_0 heat}$ we obtain that
\beas
(\mathcal{T}_f-\widetilde{\mathcal{S}}_{D}^{-1} [g], \varphi_0 )_{\mathcal{H}^*} = \int_D f(x) dx - \int_{\p D}\widetilde{\mathcal{S}}_{D}^{-1} [g](y) d\sigma(y).
\eeas

\end{proof}


%

\subsection{Asymptotic expansions}
Let us now consider the single-layer potential for the Helmholtz equation in $\mathbb{R}^2$ given by
$$
\mathcal{S}_{D}^{ k} [\varphi](x) =  \int_{\p D} G(x, y,  k) \varphi(y) d\sigma(y),  \quad x \in  \p {D},
$$
where $G(x, y,  k)= -\dfrac{i}{4}H_0^{(1)}( k|x-y|)$ and $H_0^{(1)}$ is the Hankel function of first kind and order $0$. We have, for $k\ll 1$, 
$$
-\dfrac{i}{4}H_0^{(1)}( k|x-y|) = \dfrac{1}{2\pi}\log |x-y|+\tau_ k+\sum_{j=1}^{\infty}(b_j\log  k|x-y|+c_j)( k|x-y|)^{2j},
$$
where
$$
\tau_ k = \dfrac{1}{2\pi}(\log  k+ \gamma_e-\log 2)-\dfrac{i}{4}, \quad b_j = \dfrac{(-1)^j}{2\pi}\dfrac{1}{2^{2j}(j!)^2}, \quad c_j = -bj\left( \gamma_e -\log 2-\dfrac{i\pi}{2}-\sum_{n=1}^j\dfrac{1}{n}\right),
$$
and
$ \gamma_e$ is the Euler constant. Thus, we get
\be \label{series-s2d heat}
\mathcal{S}_{D}^{ k}=  \hat{\mathcal{S}}_{D}^ k +\sum_{j=1}^{\infty} \left( k^{2j}\log  k\right) \mathcal{S}_{D, j}^{(1)}+\sum_{j=1}^{\infty}  k^{2j} \mathcal{S}_{D, j}^{(2)},
\ee
where
\beas
\hat{\mathcal{S}}_{D}^ k[\varphi](x) &=& \mathcal{S}_{D}[\varphi](x) + \tau_ k\int_{\partial D}\varphi d\sigma, \\
\mathcal{S}_{D, j}^{(1)} [\varphi](x) &=& \int_{\p D} b_j|x-y|^{2j} \varphi(y)d\sigma(y),\\
\mathcal{S}_{D, j}^{(2)} [\varphi](x) &=& \int_{\p D} |x-y|^{2j}(b_j\log|x-y|+c_j)\varphi(y)d\sigma(y).
\eeas
\begin{lem} \label{lem-appendix11_2d heat}
The norms $\| \mathcal{S}_{D, j}^{(1)} \|_{\mathcal{L}({\mathcal{H}^*(\p D)}, \mathcal{H}(\p D))}$ and $\| \mathcal{S}_{D, j}^{(2)} \|_{\mathcal{L}({\mathcal{H}^*(\p D)}, \mathcal{H}(\p D))}$ are uniformly bounded with respect to $j$. Moreover, the series in \eqref{series-s2d heat} is convergent in $\mathcal{L}({\mathcal{H}^*(\p D)}, \mathcal{H}(\p D))$ for $k<1$.
\end{lem}
Observe that
$$
\left(\mathcal{S}_{D}-\widetilde{\mathcal{S}}_{D}\right)[\varphi] =  \left(\mathcal{S}_{D}-\widetilde{\mathcal{S}}_{D}\right)[\mathcal{P}_{\mathcal{H}^*_0}[\varphi]+(\varphi,\varphi_0)_{\mathcal{H}^*}\varphi_0] = (\varphi,\varphi_0)_{\mathcal{H}^*}\left(\mathcal{S}_D[\varphi_0]+1 \right).
$$
\\
Then it follows that
$$
\hat{\mathcal{S}}_{D}^ k[\varphi] = \widetilde{\mathcal{S}}_{D}[\varphi]+(\varphi,\varphi_0)_{\mathcal{H}^*}\left(\mathcal{S}_D[\varphi_0]+1 \right)+\tau_ k\int_{\partial D}\mathcal{P}_{\mathcal{H}^*_0}[\varphi]+(\varphi,\varphi_0)_{\mathcal{H}^*}\varphi_0 d\sigma = \widetilde{\mathcal{S}}_{D}[\varphi]+\Upsilon_ k[\varphi],
$$
where \be \label{defUpsilon heat} \Upsilon_ k[\varphi] = (\varphi,\varphi_0)_{\mathcal{H}^*}\left(\mathcal{S}_D[\varphi_0]+1 +\tau_ k \right).\ee

Therefore, we arrive at the following result.
\begin{lem} \label{addeq2}
For $ k$ small enough, $\hat{\mathcal{S}}_{D}^{ k} : \mathcal{H}^*(\p D) \rightarrow \mathcal{H}(\p D)$ is invertible.
\end{lem}
\begin{proof}
$\Upsilon_ k$ is clearly a compact operator. Since $\widetilde{\mathcal{S}}_{D}$ is invertible, the invertibility of $\hat{\mathcal{S}}_{D}^{ k}$ is equivalent to that of $\hat{\mathcal{S}}_{D}^{ k}\widetilde{\mathcal{S}}_{D}^{-1} = I + \Upsilon_ k\widetilde{\mathcal{S}}_{D}^{-1}$. By the Fredholm alternative, we only need to prove the injectivity of $I + \Upsilon_ k \widetilde{\mathcal{S}}_{D}^{-1}$.\\
Since $\forall \; v\in H^{1/2}(\p D),\;\Upsilon_ k \widetilde{\mathcal{S}}_{D}^{-1}[v]\in \mathbb{C}$, for $\left(I + \Upsilon_ k \widetilde{\mathcal{S}}_{D}^{-1}\right)[v]=0$, we need to show that $v = \widetilde{\mathcal{S}}_{D}[\alpha \varphi_0] = -\alpha \in \mathbb{C}$.\\
We have
$$
\left(I + \Upsilon_ k \widetilde{\mathcal{S}}_{D}^{-1}\right)\widetilde{\mathcal{S}}_{D}[\alpha \varphi_0] = \alpha(\mathcal{S}_D[\varphi_0] +\tau_ k ) = 0 \quad \mbox{iff} \quad \mathcal{S}_D[\varphi_0] = -\tau_ k \mbox{ or } \alpha = 0.
$$
Since we can always find a small enough $ k$ such that $\mathcal{S}_D[\varphi_0] \neq -\tau_ k$, we need $\alpha = 0$. This yields the stated result.
\end{proof}

\begin{lem} \label{addeq3}
For $ k$ small enough, the operator $\mathcal{S}_{D}^{ k} : \mathcal{H}^*(\p D) \rightarrow \mathcal{H}(\p D) $ is invertible.
\end{lem}
\begin{proof} The operator
$\mathcal{S}_{D}^{ k}-\hat{\mathcal{S}}_{D}^{ k}: \mathcal{H}^*(\p D) \rightarrow \mathcal{H}(\p D)$ is a compact operator. Because $\hat{\mathcal{S}}_{D}^{ k}$ is invertible for $ k$ small enough, by the Fredholm alternative only the injectivity of $\mathcal{S}_{D}^{ k}$ is necessary. From the uniqueness of a solution to the Helmholtz equation we get the result. 
\end{proof}

\begin{lem} The following asymptotic expansion  holds for $ k$ small enough:  
$$
\begin{array}{lll}
(\mathcal{S}_{D}^{ k})^{-1} &=&\ds \mathcal{P}_{\mathcal{H}^*_0}\widetilde{\mathcal{S}}_{D}^{-1}+\mathcal{U}_ k- k^2\log  k \mathcal{P}_{\mathcal{H}^*_0}\widetilde{\mathcal{S}}_{D}^{-1}\mathcal{S}_{D,1}^{(1)}\mathcal{P}_{\mathcal{H}^*_0}\widetilde{\mathcal{S}}_{D}^{-1} + O( k^2)
\end{array}$$
with \begin{equation} \label{addeq1}
 \mathcal{U}_ k= -\dfrac{(\widetilde{\mathcal{S}}_{D}^{-1} [\cdot],\varphi_0)_{\mathcal{H}^*}}{\mathcal{S}_D[\varphi_0]+\tau_ k}\varphi_0.
 \end{equation} Note that $\mathcal{U}_ k = O(1/\log  k)$.
\end{lem}
\begin{proof}
We can write \eqref{series-s2d heat} as
$$
\mathcal{S}_{D}^{ k}=  \hat{\mathcal{S}}_{D}^ k+\mathcal{G}_ k,
$$
where $\mathcal{G}_ k =  k^2\log  k \mathcal{S}_{D,1}^{(1)} + O( k^2)$. From Lemma \ref{addeq2} and Lemma \ref{addeq3} we get the identity
$$
(\mathcal{S}_{D}^{ k})^{-1} = \left(I + (\hat{\mathcal{S}}_{D}^ k)^{-1}\mathcal{G}_ k\right)^{-1}(\hat{\mathcal{S}}_{D}^ k)^{-1}.
$$
Hence, we have $$ (\hat{\mathcal{S}}_{D}^ k)^{-1} = \underbrace{\left( \widetilde{\mathcal{S}}_{D}^{-1} \hat{\mathcal{S}}_{D}^ k\right)^{-1}}_{\Lambda_ k^{-1}}\widetilde{\mathcal{S}}_{D}^{-1}.$$
Here,
\beas
\Lambda_ k &=& I - (\cdot,\varphi_0)_{\mathcal{H}^*}(\mathcal{S}_D[\varphi_0]+1+\tau_ k)\varphi_0\\
&=& \mathcal{P}_{\mathcal{H}^*_0} - (\cdot,\varphi_0)_{\mathcal{H}^*}(\mathcal{S}_D[\varphi_0]+\tau_ k)\varphi_0.
\eeas
Then,
\beas
\Lambda_ k^{-1} = \mathcal{P}_{\mathcal{H}^*_0} - (\cdot,\varphi_0)_{\mathcal{H}^*}\frac{1}{\mathcal{S}_D[\varphi_0]+\tau_ k}\varphi_0,
\eeas
and therefore,
$$
(\hat{\mathcal{S}}_{D}^ k)^{-1} =  \mathcal{P}_{\mathcal{H}^*_0}\widetilde{\mathcal{S}}_{D}^{-1} - \dfrac{(\widetilde{\mathcal{S}}_{D}^{-1} [\cdot],\varphi_0)_{\mathcal{H}^*}}{\mathcal{S}_D[\varphi_0]+\tau_ k}\varphi_0.
$$
It is clear that $\|(\hat{\mathcal{S}}_{D}^ k)^{-1}\|_{\mathcal{L}\left(\mathcal{H}(\partial D),\mathcal{H}^*(\partial D)\right)}$ is bounded for $ k$ small. Since $||\mathcal{G}_k||_{\mathcal{L}\left(\mathcal{H}(\partial D),\mathcal{H}^*(\partial D)\right)}$ goes to zero as $k$ goes to zero, for $ k$ small enough,  we can  write
$$
(\mathcal{S}_{D}^{ k})^{-1} = (\hat{\mathcal{S}}_{D}^ k)^{-1}-(\hat{\mathcal{S}}_{D}^ k)^{-1}\mathcal{G}_ k(\hat{\mathcal{S}}_{D}^ k)^{-1}  +O\left( k^4(\log k)^2  \right),
$$
which yields the desired result.
\end{proof}

We now consider the expansion for the boundary integral operator $(\mathcal{K}_{D}^{ k})^*$. We have
\be \label{series-k2d heat}
(\mathcal{K}_{D}^{ k})^* = \mathcal{K}_D^* +\sum_{j=1}^{\infty} \left( k^{2j}\log  k\right) \mathcal{K}_{D, j}^{(1)}+\sum_{j=1}^{\infty}  k^{2j} \mathcal{K}_{D, j}^{(2)},
\ee
where
\beas
\mathcal{K}_{D, j}^{(1)} [\varphi](x) &=& \int_{\p D} b_j\dfrac{\partial |x-y|^{2j}}{\partial \nu(x)}\varphi(y)d\sigma(y),\\
\mathcal{K}_{D, j}^{(2)} [\varphi](x) &=& \int_{\p D} \dfrac{\partial \left( |x-y|^{2j}(b_j\log|x-y|+c_j)\right)}{\p \nu(x)}\varphi(y)d\sigma(y).
\eeas

\begin{lem} \label{lem-appendix132d heat} The norms
$\| \mathcal{K}_{D, j}^{(1)} \|_{ \mathcal{L}(\mathcal{H}^*(\p D), \mathcal{H}^*(\p D))}$ and $\| \mathcal{K}_{D, j}^{(2)} \|_{\mathcal{L}( \mathcal{H}^*(\p D), \mathcal{H}^*(\p D))}$ are uniformly bounded for $j \geq 1$.
Moreover, the series in (\ref{series-k2d heat}) is convergent in $\mathcal{L}({\mathcal{H}^*(\p D)}, {\mathcal{H}^*(\p D)})$.
\end{lem}


\end{document}